\documentclass[11pt,b5paper,notitlepage]{article}
\usepackage[b5paper, margin={0.5in,0.65in}]{geometry}
\usepackage{amsmath,amscd,amssymb,amsthm,mathrsfs,amsfonts,layout,indentfirst,graphicx,caption,mathabx, stmaryrd,appendix,calc,imakeidx,upgreek} 
\usepackage[dvipsnames]{xcolor}
\usepackage{palatino}  
\usepackage{slashed} 
\usepackage{mathrsfs} 
\usepackage{extarrows} 
\usepackage{enumitem} 

\usepackage{csquotes} 
\usepackage{epigraph}   

\usepackage{fancyhdr} 

\usepackage{ulem}  

\usepackage{relsize} 

\usepackage{verbatim}  

\usepackage{halloweenmath} 

\usepackage{tcolorbox}
\tcbuselibrary{theorems}

\usepackage{pdfrender}

\usepackage{tikz}
\usetikzlibrary{fadings}
\usetikzlibrary{patterns}
\usetikzlibrary{shadows.blur}
\usetikzlibrary{shapes}

\usepackage{tikz-cd}
\usepackage[nottoc]{tocbibind}   

\makeindex

\usepackage{lipsum}  
\let\OLDthebibliography\thebibliography
\renewcommand\thebibliography[1]{
	\OLDthebibliography{#1}
	\setlength{\parskip}{0pt}
	\setlength{\itemsep}{2pt} 
}

\allowdisplaybreaks  
\usepackage{latexsym}
\usepackage{chngcntr}
\usepackage[colorlinks,linkcolor=blue,anchorcolor=blue, linktocpage,
]{hyperref}
\hypersetup{ urlcolor=cyan,
	citecolor=[rgb]{0,0.5,0}}

\counterwithin{figure}{section}

\theoremstyle{definition}
\newtheorem{df}{Definition}[section]
\newtheorem{eg}[df]{Example}

\newtheorem{rem}[df]{Remark}

\newtheorem{cv}[df]{Convention}

\newtheorem{sett}[df]{Setting}

\theoremstyle{plain}
\newtheorem{thm}[df]{Theorem}

\newtheorem{pp}[df]{Proposition}
\newtheorem{co}[df]{Corollary}
\newtheorem{lm}[df]{Lemma}

\usepackage{calligra}
\DeclareMathOperator{\shom}{\mathscr{H}\text{\kern -3pt {\calligra\large om}}\,}
\DeclareMathOperator{\sext}{\mathscr{E}\text{\kern -3pt {\calligra\large xt}}\,}
\DeclareMathOperator{\Rel}{\mathscr{R}\text{\kern -3pt {\calligra\large el}~}\,}
\DeclareMathOperator{\sann}{\mathscr{A}\text{\kern -3pt {\calligra\large nn}}\,}
\DeclareMathOperator{\send}{\mathscr{E}\text{\kern -3pt {\calligra\large nd}}\,}
\DeclareMathOperator{\stor}{\mathscr{T}\text{\kern -3pt {\calligra\large or}}\,}

\usepackage{aurical}
\DeclareMathOperator{\VVir}{\text{\Fontlukas V}\text{\kern -0pt {\Fontlukas\large ir}}\,}

\newcommand{\fk}{\mathfrak}
\newcommand{\mc}{\mathcal}
\newcommand{\wtd}{\widetilde}
\newcommand{\wht}{\widehat}
\newcommand{\wch}{\widecheck}
\newcommand{\ovl}{\overline}

\newcommand{\Tr}{\mathrm{Tr}}
\newcommand{\idt}{\mathbf{1}}
\newcommand{\id}{\mathrm{id}}
\newcommand{\Hom}{\mathrm{Hom}}

\newcommand{\Rep}{\mathrm{Rep}}
\newcommand{\Repf}{\mathrm{Rep}^{\mathrm f}}

\newcommand{\Dom}{\scr D}

\newcommand{\uni}{\mathrm{u}}

\newcommand{\Diffp}{\mathrm{Diff}^+}
\newcommand{\Diff}{\mathrm{Diff}}
\newcommand{\PSU}{\mathrm{PSU}(1,1)}
\newcommand{\Vir}{\mathrm{Vir}}

\newcommand{\Span}{\mathrm{Span}}

\newcommand{\bk}[1]{\langle {#1}\rangle}

\newcommand{\GA}{\mathscr G_{\mathcal A}}

\newcommand{\scr}{\mathscr}

\newcommand{\Jtd}{\widetilde{\mathcal J}}

\newcommand{\xk}{\mathfrak x}
\newcommand{\yk}{\mathfrak y}

\newcommand{\im}{\mathbf{i}}

\newcommand{\RepV}{{\mathrm{Rep}^\uni(V)}}

\newcommand{\RepU}{\mathrm{Rep}^\uni(U)}

\newcommand{\mbb}{\mathbb}

\newcommand{\blt}{\bullet}

\newcommand{\Cbb}{\mathbb C}
\newcommand{\Nbb}{\mathbb N}
\newcommand{\Zbb}{\mathbb Z}

\newcommand{\Rbb}{\mathbb R}

\newcommand{\Hbb}{\mathbb H}

\newcommand{\wt}{\mathrm{wt}}

\newcommand{\Rng}{\mathrm{Rng}}

\newcommand{\pr}{\mathrm {pr}}

\newcommand{\UPSU}{\widetilde{\mathrm{PSU}}(1,1)}
\newcommand{\Sbb}{{\mathbb S}}
\newcommand{\Gc}{\mathscr G_c}
\newcommand{\Obj}{\mathrm{Obj}}
\newcommand{\bpr}{{}^\backprime}

\newcommand{\Imag}{\mathrm{Im}}

\newcommand{\mk}{\mathfrak m}

\newcommand{\eps}{\varepsilon}

\newcommand{\CWX}{{\scriptscriptstyle \mathrm{CWX}}}
\newcommand{\FVCWX}{{\mathfrak F^V_{\scriptscriptstyle \mathrm{CWX}}}}
\newcommand{\tFVCWX}{{\widetilde{\mathfrak F}^V_{\scriptscriptstyle \mathrm{CWX}}}}
\newcommand{\VOA}{{\scriptscriptstyle \mathrm{VOA}}}
\newcommand{\CN}{{\scriptscriptstyle \mathrm{CN}}}
\newcommand{\RepUP}{\mathrm{Rep}^{\mathrm u}(U_P)}
\newcommand{\Connes}{\mathrm{Connes}}

\usepackage{tipa} 

\usepackage{tipx}

\numberwithin{equation}{section}

\title{Comparison of Extensions of Unitary Vertex Operator Algebras and Conformal Nets}
\author{{\sc Bin Gui}
}
\date{}
\begin{document}\sloppy 
	\pagenumbering{arabic}
	\setcounter{section}{-1}


	\maketitle

\begin{abstract}
Let $V$ be one of the following unitary strongly-rational VOAs: unitary WZW models, discrete series $W$-algebras of type $ADE$, even lattice VOAs, parafermion VOAs, their tensor products, and their strongly-rational cosets.

Let $\mc A_V$ be the conformal net associated to $V$ in the sense of Carpi-Kawahigashi-Longo-Weiner (CKLW). According to \cite{Gui26},  the $*$-functor $\fk F_\CWX^V:\RepV\rightarrow\Rep(\mc A_V)$ of Carpi-Weiner-Xu (CWX) can be defined and extended to a braided $*$-functor $(\fk F^V_\CWX,\fk W^V)$ implementing an isomorphism of braided $C^*$-tensor categories $\RepV\simeq \Repf(\mc A_V)$ where $\RepV$ is the category of unitary $V$-modules, and $\Repf(\mc A_V)$ is the category of dualizable (i.e. finite-index) $\mc A_V$-modules. The tensorator $\fk W^V$, originally due to \cite{Was98}, is called the Wassermann tensorator.

In this paper, we show that if $U$ is a (conformal) VOA extension of $V$ (which is automatically unitary \cite{CGGH23}), then the CKLW net $\mc A_U$ and the CWX functor $\fk F^U_\CWX$ can be defined. (Namely, $U$ is strongly local, and all unitary $U$-modules are strongly integrable.) Moreover, identifying $\RepV$ with $\Repf(\mc A_V)$ via $(\fk F^V_\CWX,\fk W^V)$, we prove three comparison results: 
\begin{enumerate}
\item If $P$ is the $C^*$-Frobenius algebra in $\RepV$ associated to $U$, and if $\mc B$ is the conformal net extension of $\mc A_V$ defined by $P$, then $\mc B=\mc A_U$.
\item The inverse of the canonical braided $*$-functor $(\fk F_\VOA,\fk V^\boxdot):\Rep^0(P)\xrightarrow{\simeq}\RepU$ (where $\Rep^0(P)$ is the category of unitary dyslectic $P$-modules), composed with the canonical one $(\fk F_\CN,\fk N^\boxtimes):\Rep^0(P)\xrightarrow{\simeq}\Repf(\mc B)$, is equal to the CWX functor $\fk F^U_\CWX:\RepU\rightarrow\Repf(\mc B)$ together with a tensorator.
\item This tensorator is equal to the Wassermann tensorator whenever the later can be defined (e.g., when $V$ is one of the following VOAs: unitary  WZW models and discrete series $W$-algebras of type $ADE$, even lattice VOAs, parafermions VOAs of type $ADE$, their tensor products, their strongly-rational cosets, and their VOA extensions).
\end{enumerate}
\end{abstract}

\tableofcontents

\section{Introduction}

This article deals with two types of problems about the relationship between unitary extensions of VOAs and conformal nets in a uniform framework: 
\begin{itemize}
\item (Analytic problems) Strong locality and strong integrability.
\item (Algebraic problems) Comparison of extensions and braided $*$-functors.
\end{itemize}
Let me introduce these two types of questions in turn.

\subsection{Strong locality and strong integrability for unitary VOA extensions}

\subsubsection{Energy bounds and strong locality}

Unitary vertex operator algebras (VOAs) and conformal nets are two seemingly different but conjecturally equivalent mathematical formulations of two-dimensional chiral conformal field theories. A systematic study of the relationship between these two objects was initiated by Carpi-Kawahigashi-Longo-Weiner \cite{CKLW18}. In this work, the authors showed that many familiar unitary VOAs $V$ have an associated conformal net $\mc A_V$ defined on the Hilbert space completion $\mc H_0$ of $V$.

The construction of the \textbf{CKLW net} $\mc A_V$ from a unitary VOA $V$ belongs to the tradition of constructing a Haag-Kastler net from a Wightmann QFT. (See the Introduction of \cite{CKLW18} and the reference therein.) In this tradition, one of the biggest challenges in constructing nets of algebras satisfying the Haag-Kastler axioms \cite{Haag96} is to show that two smeared operators localized in spacelike separated regions are strongly commuting. In the present context, this means that if $f\in C_c^\infty(I)$ and $g\in C_c^\infty(J)$ where $I$ and $J$ are disjoint (nonempty non-dense open) intervals of the unit circle $\Sbb^1$, and if $u,v\in V$, then $Y(u,f)$ and $Y(v,g)$ \textbf{commute strongly} in the sense that the smallest von Neumann algebra to which $Y(u,f)$ is affiliated commutes with the one for $Y(v,g)$. Any unitary $V$ satisfying this property is called \textbf{strong locality}.

In \cite{CKLW18}, the treatment of strong locality is aided and simplified by another useful analytic property called \textbf{(polynomial) energy bounds}. It says roughly that smeared operators $Y(v,f)$ (and their products) are bounded by $(1+\ovl{L_0})^r$ for some $r\geq0$ (depending on $v$ but not on $f$), where $\ovl{L_0}\geq0$ is the energy operator. ($L_0$ is the restriction of $\ovl{L_0}$ to the subspace of vectors with finite $\ovl{L_0}$-spectra.) 

One of the most remarkable properties of energy bounds is that many important dense subspaces of $\mc H_0^\infty=\bigcap_{r\geq0}\Dom(1+\ovl{L_0}^r)$ (called \textbf{quasi-rotation invariant (QRI)} spaces in our paper, cf. Def. \ref{lb32}), including $\mc H_0^\infty$ itself, are cores for any linear combination of energy-bounded smeared operators, cf. \cite[Lem. 7.2]{CKLW18}. This property plays a crucial role in the proof of Thm. 8.1 of \cite{CKLW18}, which says that strong locality can be passed from a generating set of (quasi-primary) field operators to all field operators. (See also \cite{DSW86}.)

Thanks to \cite[Thm. 8.1]{CKLW18} and the fact that field operators with \textbf{linear energy bounds} (i.e., they are ``bounded by $1+\ovl L_0$") satisfy strong locality (cf. e.g. \cite{GJ87,DF77,BS90}), any unitary VOA generated by linearly energy-bounded (quasi-primary) fields (e.g. unitary affine VOAs, unitary Virasoro VOAs, Moonshine VOA, their tensor products) are strongly local. Consequently, all subalgebras of them (notably those obtained by coset construction and orbifold construction) are (energy-bounded and) strongly local. Recently, this technique has been generalized to unitary vertex operator superalgebras \cite{CGH23}.

\subsubsection{Strong locality for unitary VOA extensions}

Unfortunately, the above technique does not apply to many important unitary VOAs. To prove strong locality (and energy-bounds) for more examples, one has the following two obvious options:
\begin{enumerate}
\item[(1)] Weaken the condition of linear energy bounds.
\item[(2)] Besides taking tensor products and taking subalgebras, show that other ways of constructing new unitary VOAs from old ones also preserve (energy-boundedness and) strong locality. 
\end{enumerate} 
The first option was taken by Carpi-Tanimoto-Weiner: In \cite{CTW22}, the authors introduced the notion of local energy bounds, and used it to prove that $\mc W_3$ algebras (which are unitary by \cite{CTW23}) are strongly local.

In this paper, we pursue the second direction: Starting from a unitary (energy-bounded and) strongly local VOA $V$, let $U$ be a unitary VOA extension of $V$. (In this paper, we always assume that the extension has the same conformal vector as that of $V$.) We want to study the strong locality of $U$.

Given that proving strong locality for all unitary VOAs currently appears hopeless, establishing a general theorem on the strong locality of unitary VOA extensions is equally unattainable, since every unitary VOA can be viewed as a unitary extension of a Virasoro VOA. Therefore, we must impose additional restrictions on $V$.

Recently, Carpi and Tomassini proved in \cite{CT23} that if the field operators of $V$ acting on $U$ are energy-bounded, and if either (a) $V=U^G$ for a compact group $G$ of unitary automorphisms of $U$, or (b) $V$ is $C_2$-cofinite, then $U$ is energy-bounded. This powerful result implies for instance that all unitary extensions of unitary affine VOAs are energy-bounded. In view of \cite{CT23}, to study strong locality for a unitary energy-bounded extension $U$ of $V$ (where $V$ is strongly local), it is natural to focus on the following two cases:
\begin{enumerate}
\item[(a)] $V=U^G$ for a compact group $G$ of unitary automorphisms of $U$.
\item[(b)] $V$ is $C_2$-cofinite and rational.
\end{enumerate}

In \cite{Gui21a}, we studied (a) in the special case that $U$ is an even lattice VOA, and $V$ is the Heisenberg subalgebra of $U$. The result is that every even lattice VOA is (energy bounded \cite{TL04,Gui19c} and) strongly local, and it was generalized to all integer lattice VOSAs in \cite{CGH23}. Though we believe that it is worthwhile to study case (a) more thoroughly, in this paper we consider case (b).

\subsubsection{Categorical extensions and strong locality}

We assume that the VOA $V$ (which from now on is understood to be unitary and of CFT-type) is \textbf{strongly energy-bounded}, i.e., the field operators of $V$ are energy-bounded when acting on any unitary $V$-module. We also assume that $V$ is \textbf{completely unitary}, which means that $V$ is $C_2$-cofinite and rational, all $V$-modules are unitarizable, and the canonical Hermitian structure on the category $\RepV$ of unitary $V$-modules is positive-definite (i.e. unitary).  So $\RepV$ is a \textit{unitary} modular tensor category \cite{Hua08b,Gui19b}. (For example, $V$ can be any unitary affine VOA \cite{Gui19a,Gui19b,Gui19c,Gui21a,Ten19a,Ten19b,Ten24,Gui26}.) Let $U$ be a VOA extension of $V$, which is automatically a \textit{unitary} VOA extension in a unique way by \cite[Thm. 4.7]{CGGH23}. It follows that $U$ is strongly energy-bounded \cite{CT23} and completely unitary \cite{Gui22}. Finally, assume that $V$ is strongly local.

Our strategy of studying the strong locality of $U$ is the same as the treatment of even lattice VOAs in \cite{Gui21a}: we transform the problem about extensions to a problem about subalgebras. As pointed out in the Introduction of \cite{Gui21a}, the complete unitarity of $V$ implies that the intertwining operators (i.e. charged operators) of $V$ form the ``universal" extension of $V$ containing $U$ as a subtheory. This universal extension is called \textbf{categorical extension}. The relationship between the categorical extension and the unitary VOA extension $U$ is similar to that between a free group and a general group. A group is isomorphic to a free group modulo some relations. In the context of VOAs, the role of these relations is played by $C^*$-Frobenius algebras ($\approx$ Q-systems) in $\RepV$, originally introduced by Longo to understand type III subfactors \cite{Lon94}.

Thus, we view $U$ as the categorical extension of $V$ modulo the relations defined by a (haploid commutative) $C^*$-Frobenius algebra $P$. Therefore, the strong locality of $U$ will follow from that of the categorical extensions $V$, where the latter means the ``strong braiding" of smeared intertwining operators. Thanks to \cite{Gui26}, under the assumption that the intertwining operators involved are energy-bounded, the strong braiding follows from the \textbf{strong intertwining property}, a special case of strong braiding much easier to verify. It means roughly that if $\mc Y$ is an intertwining operator of $V$ of type $W_k\choose W_iW_j$ (where $W_i,W_j,W_k$ are unitary $V$-modules), then for each (homogeneous) $v\in V,w^{(i)}\in W_i$ and each $\wtd f=(f,\arg_I)$ where $f\in C_c^\infty(I)$ and $\arg_I$ is an arg-function of $I$, and $g\in C_c^\infty(J)$ where $I\cap J=\emptyset$, then ``$\mc Y(w^{(i)},f)$ intertwines $Y(u,g)$ strongly". (See Def. \ref{lb101} for details.) 

For example, if $V$ is a unitary affine VOA of type $ADE$, then all intertwining operators of $V$ are energy-bounded. Then, by the fact that $V$ is generated by linearly energy-bounded fields, the intertwining operators of $V$ satisfy the strong intertwining property. Thus, the categorical extension of $V$ satisfies strong braiding. (See \cite[Sec. 3.7]{Gui26} for details.) Therefore, all VOA extensions of $V$ are strongly local.

\subsubsection{Conditions I and II, and results on strong locality}

Recall that $V$ is always assumed to be completely unitary, strongly energy-bounded, and strongly local. $U$ is an (automatically unitary \cite{CGGH23}) VOA extension of $V$.

\begin{df}[$=$ Def. \ref{lb34}]
We say that $V$ satisfies \textbf{Condition I} if every intertwining operator is energy-bounded and satisfies the strong intertwining property. We say that $V$ satisfies \textbf{Condition II} (= Condition B of \cite{Gui26}) if there is a set $\mc F^V$ of unitary $V$-modules tensor-generating $\RepV$ such that every intertwining operator of type $\blt\choose W_i~\star$ (where $W_i\in\mc F^V$) is energy-bounded and satisfies the strong intertwining property.
\end{df}

It is clear that Condition I implies Condition II. The primary reason for introducing Condition I---despite not being able to verify it for all unitary affine VOAs---is that it is preserved under (unitary) VOA extensions, unlike Condition II; see Thm. \ref{lb104}.

\begin{eg}\label{lb111}
The following examples satisfy Condition I: All unitary affine VOAs of type $ADE$, all even lattice VOAs, all discrete series $W$-algebras of type $ADE$ (in the sense of \cite{ACL19}), all parafermion VOAs of type $ADE$ (in the sense of \cite{DR17}), their tensor products, the $C_2$-cofinite rational cosets of their $C_2$-cofinite rational unitary subalgebras, their (unitary) VOA extensions.
\end{eg}

\begin{eg}\label{lb112}
The following examples satisfy Condition II: All VOAs satisfying Condition I, all unitary affine VOAs, all parafermion VOAs, their tensor products, the $C_2$-cofinite rational cosets of their $C_2$-cofinite rational unitary subalgebras.
\end{eg}

\begin{proof}
See Exp. \ref{lb118} and Thm. \ref{lb116}.
\end{proof}

As explained above, if $V$ satisfies Condition I, then $U$ is strongly local. With the help of Thm. \ref{lb13} (proved in \cite[Sec. 2.6]{Gui26}), the categorical extension version of \cite[Thm. 8.1]{CKLW18}, we can weaken Condition I to Condition II:

\begin{thm}[$\subset$ Thm. \ref{lb78}]\label{lb103}
If $V$ satisfies Condition II, then $U$ is strongly local.
\end{thm}


\begin{co}
Any VOA extension of a tensor product of unitary affine VOAs and even lattice VOAs is strongly local.
\end{co}

Among the examples covered in this corollary we have:

\begin{itemize}
\item All holomorphic VOAs with central charge $c=24$, except the moonshine VOA.\footnote{The strong locality of moonshine VOA was proved in \cite{CKLW18}.} The strong locality of such VOAs was first proved in \cite[Thm. 5.5]{CGGH23}. 
\item The VOA $V_{G,k}$ associated to a compact connected Lie group $G$ and $k\in H_+^4(BG,\Zbb)$, cf. \cite{Hen17}. It is a simple current extension of a tensor product of unitary affine VOAs and lattice VOAs, and they correspond almost bijectively to (connected but non-necessarily simply-connected) chiral WZW models.
\item The classification of extensions of unitary affine VOAs with small ranks is an important topic in the literature. See e.g. \cite{CIZ87,Gan94,KO02,EP09,EM23,Gan23}.
\end{itemize}

\begin{co}
Any unitary VOA with central charge $c<1$ is strongly local.
\end{co}
This is due to the fact that such VOA is a unitary extension of a unitary minimal model Virasoro VOA, and that the latter is a type $A$ discrete series $W$-algebra.

\subsubsection{Strong intertwining property and strong integrability}

Carpi-Weiner-Xu introduced the notion of \textbf{strong integrability} for any unitary module $(W_i, Y_i)$ (with energy-bounded $Y_i$) of a strongly local VOA $V$ (not necessarily satisfying Condition II), which means that there is a (necessarily unique) $\mc A$-module $(\mc H_i,\pi_i)$ such that for every (homogeneous) $v\in V$, every interval $I$, and every $f\in C_c^\infty(I)$, we have $\pi_{i,I}(\ovl{Y(v,f)})=\ovl{Y_k(v,f)}$. If all unitary $V$-modules are strongly integrable, then the construction $W_i\mapsto\mc H_i$ gives a fully-faithful $*$-functor from the $C^*$-category of unitary $V$-modules to that of $\mc A_V$-modules. Cf. \cite{CWX,Gui19b}, or Thm. \ref{lb42} for details. We call this functor the \textbf{CWX functor} of $V$ and denote it by
\begin{align*}
\fk F^V_\CWX:\RepV\rightarrow\Rep(\mc A_V)
\end{align*}
The functor $\fk F_\CWX^V$ is a natural generalization of the construction of representations of loop group conformal nets from  (positive-energy highest weight) representations of affine Lie algebras \cite{GF93,Was98,TL04}.

It was proved in \cite{Gui19b} that the strong integrability follows from the strong intertwining property of sufficiently many intertwining operators. Returning to the setting that $V$ satisfies Condition II, we conclude that all unitary $V$-modules are strongly integrable, and hence $\fk F^V_\CWX$ can be defined. See \cite[Thm. 3.16]{Gui26} or Thm. \ref{lb48} for details.

Recall that $U$ is a (unitary) VOA extension of $V$. Recall that $Y_i$ is energy-bounded for every $W_i\in\Obj(\RepV)$ by (the proof of) \cite{CT23}. Then, similar to the proof of the strong locality of $V$, the strong locality of the categorical extension of $V$ implies that sufficiently many intertwining operators of $U$ satisfy the strong intertwining property, and that all intertwining operators satisfy the strong intertwining property if they are energy-bounded. (Note that an intertwining operator of $U$ is also one of $V$.) Thus, we have:

\begin{thm}[$\subset$ Thm. \ref{lb78}]\label{lb105}
Suppose that $V$ satisfies Condition II. Then all unitary $U$-modules are strongly integrable. Thus, we have the (fully-faithful) CWX functor $\fk F^U_\CWX:\RepU\rightarrow\Rep(\mc A_U)$ for $U$.
\end{thm}

\begin{thm}[$=$ Cor. \ref{lb90}]\label{lb104}
Suppose that $V$ satisfies Condition I. Then $U$ also satisfies Condition I.
\end{thm}

Unfortunately, we cannot show that $U$ satisfies Condition II if $V$ does: The fact that $V$ has sufficiently many energy-bounded intertwining operators does not imply the same property for $U$. (Thus, for example, we know that all extensions of type $A,D,E$ unitary affine VOAs satisfy Condition I and hence Condition II. But we do not know whether all extensions of type $B,C,F,G$ unitary affine VOAs satisfy Condition II.) However, the above mentioned argument shows that the only obstacle lies in proving that sufficiently many intertwining operators of $U$ are energy-bounded. See Cor. \ref{lb81} for details.

\subsubsection{Equivalence of braided $C^*$-tensor categories}\label{lb109}

Assume that $V$ satisfies Condition II, and let 
\begin{align}
\Rep^V(\mc A_V)=\fk F_\CWX^V(\RepV)
\end{align}
which is a full replete subcategory of $\Rep(\mc A_V)$. Then we have an isomorphism of $C^*$-tensor categories
\begin{align}
\fk F_\CWX^V:\RepV\xlongrightarrow{\simeq}\Rep^V(\mc A_V)
\end{align}
It was shown in \cite{Gui26} that $\fk F_\CWX^V$ can be extended to a braided $*$-functor implementing an isomorphism of braided $C^*$-tensor categories
\begin{align}\label{eq95}
(\fk F^V_\CWX,\fk W^V): \big(\RepV,\boxdot,\ss\big)\xlongrightarrow{\simeq}\big(\Rep^V(\mc A_V),\boxtimes,\mathbb B\big)
\end{align}
where $(\boxdot,\ss)\equiv(\boxdot_V,\ss^V)$ is the braided $C^*$-tensor structure on $\RepV$ defined by Huang-Lepowsky theory, and $(\boxtimes,\mathbb B)\equiv(\boxtimes_{\mc A_V},\ss^{\mc A_V})$ is the braided $C^*$-tensor structure on $\Rep(\mc A_V)$ (restricted to $\Rep^V(\mc A_V)$) defined by Connes fusion. For each $W_i,W_j\in\Obj(\RepV)$, set
\begin{align*}
\mc H_i=\fk F_\CWX^V(W_i)\qquad \mc H_j=\fk F_\CWX^V(W_j)\qquad \mc H_i\boxdot\mc H_j=\fk F_\CWX^V(W_i\boxdot W_j)
\end{align*}
Then $\fk W^V$ is a natural unitary map associating to each $W_i,W_j$ unitary $\mc A_V$-module morphisms
\begin{align*}
\fk W^V_{i,j}:\mc H_i\boxdot\mc H_j\xlongrightarrow{\simeq}\mc H_i\boxtimes\mc H_j
\end{align*}
intertwining the associators, unitors, and braiding operators of $\RepV$ and $\Rep^V(\mc A_V)$. The tensorator $\fk W^V$ is defined using smeared intertwining operators in the same spirit as Wassermann's computation of Connes fusion for loop group conformal nets \cite{Was98}. Thus, we call $\fk W^V$ the \textbf{Wassermann tensorator} for $V$.

Now, suppose that $V$ satisfies Condition I. Then any VOA extension $U$ also satisfies Condition I and hence Condition II. Then we have a similar isomorphism of braided $C^*$-tensor categories
\begin{align*}
(\fk F^U_\CWX,\fk W^U): \big(\RepV,\boxdot_U,\ss^U\big)\xlongrightarrow{\simeq}\big(\Rep^U(\mc A_U),\boxtimes_{\mc A_U},\mathbb B^{\mc A_U}\big)
\end{align*}
However, if $V$ only satisfies Condition II, then $U$ is not known to satisfy Condition II in general. So the Wassermann tensorator $\fk W^U$ cannot be defined although we still have $\fk F^U_\CWX$ by Thm. \ref{lb105}. Can $\fk F_\CWX^U:\RepV\xrightarrow{\simeq}\Rep^U(\mc A_U)$ be extended to a braided $*$-functor implementing an isomorphism of braided $C^*$-tensor categories? We will address this question in the next subsection.

\subsection{Comparison of extensions and braided $*$-functors}

\subsubsection{$C^*$-Frobenius algebras}

The notion of $C^*$-Frobenius algebras (or Q-systems) in a (braided) $C^*$-tensor category, originally introduced by Longo in \cite{Lon94}, is a powerful tool for the study of unitary extensions in 2d rational conformal field theory: If $\mc A$ is a conformal net, then haploid commutatitive $C^*$-Frobenius algebras $\Theta$ in $\Rep(\mc A)$ correspond bijectively to finite-index conformal net extensions $\mc B_\Theta$ of $\mc A$, and we have a canonical isomorphism of braided $C^*$-tensor categories
\begin{align}\label{eq93}
(\fk F_\CN,\fk N^\boxtimes):\big(\Rep^0(\Theta),\boxtimes_\Theta,\mathbb B^\Theta\big)\xlongrightarrow{\simeq}\big(\Rep(\mc B_\Theta),\boxtimes_{\mc B_\Theta},\mathbb B^{\mc B_\Theta}\big)
\end{align}
where the source is the category of (unitary) dyslectic $\Theta$-modules (together with the canonical braided $C^*$-tensor structure), and the target is (as before) the category of $\mc B_\Theta$-modules whose braided $C^*$-tensor structure is defined by Connes fusion/DHR superselection theory. Cf. for example \cite{LR95,EP03,BKL15,BKLR15,Gui21c}. (Our approach here follows \cite{Gui21c}. See Subsec. \ref{lb106}-\ref{lb40} for a detailed exposition.)

Similarly, if $V$ is completely unitary, then haploid commutative $C^*$-Frobenius algebras $P$ in $\RepV$ correspond bijectively to (automatically unitary) VOA extensions of $V$, and we have a canonical isomorphism of braided $C^*$-tensor categories
\begin{align}\label{eq94}
(\fk F_\VOA,\fk V^\boxdot):\big(\Rep^0(P),\boxdot_P,\ss^P\big)\xlongrightarrow{\simeq}\big(\RepUP,\boxdot_{U_P},\ss^{U_P} \big)
\end{align}
where the source is (as in \eqref{eq93}) the braided $C^*$-tensor category of (unitary) dyslectic $P$-modules, and the target is (as in \eqref{eq95}) the Huang-Lepowsky braided $C^*$-tensor category of unitary $U_P$-modules. Cf. \cite{KO02,HKL15,CKM24,Gui22,CGGH23}. (Here, our approach follows \cite{Gui22}. See Subsec. \ref{lb107} and \ref{lb108} for details.)

\subsubsection{Comparison of conformal net extensions}

Assume that $V$ satisfies Condition II. Recall from Subsec. \ref{lb109} that the braided $*$-functor $(\fk F^V_\CWX,\fk W^V)$ implements an isomorphism $(\RepV,\boxdot,\ss)\simeq(\Rep^V(\mc A_V),\boxtimes,\mathbb B)$. Thus, a haploid commutative $C^*$-Frobenius algebra $P$ in $\RepV$ corresponds to one $\Theta$ in $\Rep^V(\mc A_V)$. Through this correspondence, an extension $U_P$ of $V$ corresponds to an extension $\mc B_\Theta$ of $\mc A:=\mc A_V$. Thus, it is natural to ask:
\begin{itemize}
\item[\large\textcircled{\small 1}] Is $\mc B_\Theta$ isomorphic to $\mc A_{U_P}$, the CKLW net of $U_P$? \footnote{This question was communicated to me by Sebastiano Carpi.}
\end{itemize}

Let us recall the data $P=(W_a,\mu,\iota)$, where $W_a\in\Obj(\RepV)$, $\mu\in\Hom_V(W_a\boxdot W_a,W_a)$, and $\iota\in\Hom_V(V,W_a)$ is assumed for simplicity to be an isometry. Viewing $\mc H_a=\fk F^V_\CWX(W_a)$ as a Hilbert space, both $\mc B_\Theta$ and $\mc A_{U_P}$ are conformal nets acting on $\mc H_a$ (since $U_P$ equals $W_a$ as vector space).

Here is a natural way to think about problem \textcircled{\small 1}: The actions of $\mc B_\Theta$ and $\mc A_{U_P}$ on $\mc H_a$, when restricted to $\mc A_V$, are equal since they are given by the $\mc A_V$-module $\mc H_a$. This means that if we let $\Theta'$ be the haploid commutative $C^*$-Frobenius algebra such that $\mc A_{U_P}=\mc B_{\Theta'}$, then $\Theta$ and $\Theta'$ are equal as objects of $\Rep^V(\mc A_V)$. In many cases (e.g., $V$ is a unitary $c<1$ Virasoro VOA, or a type $A$ unitary affine VOA of small rank), the equality of $\Theta$ and $\Theta'$ as objects implies that they are equivalent (and hence unitarily equivalent \cite{CGGH23}) as $C^*$-Frobenius algebras, cf. e.g. \cite{Gan23} and the reference therein. 

Unfortunately, at this moment, it is not known in general whether two $C^*$-Frobenius algebras are isomorphic as \textit{algebras} if they are isomorphic as \textit{objects}. Therefore, in this paper, we do not use this approach for problem \textcircled{\small 1}. Instead, we prove directly that $\mc B_\Theta$ equals $\mc A_{U_P}$ as conformal nets acting on $\mc H_a$. We will explain the idea of the proof later. For now, let me make two comments on this result: 

\begin{rem}
Note that the braided $C^*$-tensor structure on $\Rep(\mc A_V)$ can be defined either using Connes fusion or using DHR superselection theory \cite{FRS89,FRS92}. These two structures are equivalent, cf. \cite[Sec. 6]{Gui21a}. However, to prove that $\mc B_\Theta$ is equal to (but not just isomorphic to) $\mc A_{U_P}$, we use Connes fusion. In fact, as suggested in \cite{Gui21a,Gui26}, for the purpose of comparing the representation categories of VOAs and conformal nets using smeared operators, it is more convenient to use Connes fusion than DHR superselection theory. 
\end{rem}

\begin{rem}
The equality $\mc B_\Theta=\mc A_{U_P}$ relies on our choice of $\Theta$, which is the pushforward of $P$ to $\Rep^V(\mc A_V)$ via the braided $*$-functor $(\fk F^V_\CWX,\fk W^V)$. In particular, \uline{the Wassermann tensorator $\fk W^V$ is essential for the equation $\mc B_\Theta=\mc A_{U_P}$ to hold}. Defining $\Theta$ by a different tensorator will only give $\mc B_\Theta\simeq\mc A_{U_P}$ but not $\mc B_\Theta=\mc A_{U_P}$.
\end{rem}

\subsubsection{Comparison of braided $*$-functors}\label{lb110}

In this subsubsection, we make the identification
\begin{align}\label{eq103}
\big(\RepV,\boxdot,\ss\big)=\big(\Rep^V(\mc A_V),\boxtimes,\mathbb B\big)
\end{align}
via the braided $*$-functor $(\fk F^V_\CWX,\fk W^V)$. Then the haploid $C^*$-Frobenius algebra $P$ in $\RepV$ is identified with its pushforward $\Theta=(\fk F^V_\CWX,\fk W^V)(P)$ in $\Rep^V(\mc A_V)$. Therefore, we also have a canonical identification
\begin{align}\label{eq102}
\big(\Rep^0(P),\boxdot_P,\ss^P\big)=\big(\Rep^0_{\Rep^V(\mc A_V)}(\Theta),\boxtimes_\Theta,\mathbb B^\Theta\big)
\end{align}
The subscript $\Rep^V(\mc A_V)$ in $\Rep^0_{\Rep^V(\mc A_V)}(\Theta)$ emphasizes that $\Rep^0_{\Rep^V(\mc A_V)}(\Theta)$ is the category of dyslectic $\Theta$-modules in $\Rep^V(\mc A_V)=\fk F^V_\CWX(\RepV)$ (but not e.g. in $\Rep(\mc A_V)$).

Now, by \eqref{eq93} and \eqref{eq94}, the braided $*$-functors $(\fk F_\CN,\fk N^\boxtimes)$ and $(\fk F_\VOA,\fk V^\boxdot)$ give an isomorphism of braided $C^*$-tensor categories
\begin{align}\label{eq96}
\big(\RepUP,\boxtimes_{U_P},\ss^{U_P} \big)\xlongrightarrow{\simeq}\big(\Rep_{\Rep^V(\mc A_V)}(\mc B_\Theta),\boxtimes_{\mc B_\Theta},\mathbb B^{\mc B_\Theta}\big)
\end{align}
where $\Rep_{\Rep^V(\mc A_V)}(\mc B_\Theta)$ is the category of $\mc B_\Theta$-modules whose restrictions to $\mc A_V$ are objects in $\Rep^V(\mc A_V)$. 

Since we have $\mc B_\Theta=\mc A_{U_P}$, \eqref{eq96} gives an isomorphism of braided $C^*$-tensor categories $\RepUP\simeq\Rep_{\Rep^V(\mc A_V)}(\mc A_{U_P})$. Can we give this abstract isomorphism a more concrete description? To understand the importance of this question, look at the following scenario:

\begin{eg}
According to \cite[Appendix]{DMNO13}, we have a VOA extension $V\subset U$ of unitary affine VOAs where $V$ is $G_2$ level $3$ and $U$ is $E_6$ level $1$. Note that both $V$ and $U$ satisfy Condition II. By \eqref{eq96}, the abstract machinery of $C^*$-Frobenius algebras gives a braided $C^*$-equivalence $\RepU\simeq\Rep_{\Rep^V(\mc A_V)}(\mc A_U)$ implemented by a braided $*$-functor. Is this braided $*$-functor equal to $(\fk F^U_\CWX,\fk W^U)$? In particular, forgetting the tensorator, this question asks whether this functor corresponds to the construction of loop group representations through exponentiating the corresponding affine Lie algebra representations (as in  \cite{GF93,Was98,TL99} for example).
\end{eg}

Let us summarize our questions:
\begin{itemize}
\item[\large\textcircled{\small 2}] Suppose that $V$ satisfies Condition II. Forgetting the tensorators, is the functor in \eqref{eq96} equal to $\fk F^{U_P}_\CWX$?
\item[\large\textcircled{\small 3}] Suppose that both $V$ and $U$ satisfy Condition II. Is the braided $*$-functor \eqref{eq96} equal to $(\fk F^{U_P}_\CWX,\fk W^{U_P})$?
\end{itemize}

In this paper, we show that the answers to \large\textcircled{\small 1}, \large\textcircled{\small 2}, \large\textcircled{\small 3} \normalsize are all affirmative: 

\begin{thm}
Assume that $V$ satisfies Condition II. Then, under the identifications \eqref{eq103} and \eqref{eq102}, we have $\mc A_{U_P}=\mc B_\Theta$, and we have a commutative diagram of $*$-functors
\begin{equation}\label{eq100}
\begin{tikzcd}[row sep=large, column sep=small]
             & \Rep^0(P) \arrow[ld,"\simeq","\fk F_\VOA"'] \arrow[rd,"\fk F_\CN","\simeq"'] &   \\
\RepUP \arrow[rr,"{\fk F^{U_P}_\CWX}","\simeq"'] &                         & \Rep_{\Rep^V(\mc A_V)}(\mc B_\Theta)
\end{tikzcd}
\end{equation}
which can be extended to the following commutative diagram of braided $*$-functors when $U$ also satisfies Condition II:
\begin{equation}\label{eq101}
\begin{tikzcd}[row sep=large, column sep=-0.3cm]
             & \big(\Rep^0(P),\boxdot_P,\ss^P\big) \arrow[ld,"\simeq","{(\fk F_\VOA,\fk V^\boxdot)}"'] \arrow[rd,"{(\fk F_\CN,\fk N^\boxtimes)}","\simeq"'] &   \\
\big(\RepUP,\boxdot_{U_P},\ss^{U_P}\big) \arrow[rr,"{(\fk F^{U_P}_\CWX,\fk W^{U_P})}","\simeq"'] &                         & \big(\Rep_{\Rep^V(\mc A_V)}(\mc B_\Theta),\boxtimes_{\mc B_\Theta},\mathbb B^{\mc B_\Theta}\big)
\end{tikzcd}
\end{equation}
\end{thm}

This theorem is almost Thm. \ref{lb95}, the main comparison theorem of this article, with the only difference being that the identifications \eqref{eq103} and \eqref{eq102} are not assumed in Thm. \ref{lb95}.

In fact, if $V$ is one of the VOAs in Exp. \ref{lb112}, the commutative diagram \eqref{eq100} becomes
\begin{equation}
\begin{tikzcd}[row sep=large, column sep=small]
             & \Rep^0(P) \arrow[ld,"\simeq","\fk F_\VOA"'] \arrow[rd,"\fk F_\CN","\simeq"'] &   \\
\RepUP \arrow[rr,"{\fk F^{U_P}_\CWX}","\simeq"'] &                         & \Repf(\mc B_\Theta)
\end{tikzcd}
\end{equation}
where $\Repf(\mc B_\Theta)$ is the category of dualizable (i.e. finite index) representations of $\mc B_\Theta$. If $V$ is one of the VOAs in Exp. \ref{lb111}, then \eqref{eq101} becomes
\begin{equation}
\begin{tikzcd}[row sep=large, column sep=-0.3cm]
             & \big(\Rep^0(P),\boxdot_P,\ss^P\big) \arrow[ld,"\simeq","{(\fk F_\VOA,\fk V^\boxdot)}"'] \arrow[rd,"{(\fk F_\CN,\fk N^\boxtimes)}","\simeq"'] &   \\
\big(\RepUP,\boxdot_{U_P},\ss^{U_P}\big) \arrow[rr,"{(\fk F^{U_P}_\CWX,\fk W^{U_P})}","\simeq"'] &                         & \big(\Repf(\mc B_\Theta),\boxtimes_{\mc B_\Theta},\mathbb B^{\mc B_\Theta}\big)
\end{tikzcd}
\end{equation}
See Cor. \ref{lb117}.

\subsection{Structure the article}

In this paper, we present several closely related but slightly different comparison results, each serving a different purpose. The most useful one is Cor. \ref{lb117}. Accordingly, the reader may view the proof of Cor. \ref{lb117} as the main objective of this paper. Exp. \ref{lb111} and \ref{lb112} provide explicit examples where Cor. \ref{lb117} is applicable. A more formal presentation of these examples is provided in Exp. \ref{lb118} and Thm. \ref{lb116}.

Cor. \ref{lb117} addresses both the analytic problems of strong locality and strong integrability, as well as the algebraic problem of comparing extensions and braided $*$-functors. These two types of problems are treated in a unified way by Thm. \ref{lb72}, which roughly states that the smeared intertwining operators of $U_P$ are affiliated with the categorical extension of $\mc B_\Theta$. 

The notion of weak categorical extensions provides an abstract framework for the smeared intertwining operators of VOAs. The relationship between weak categorical extensions and categorical extensions of conformal nets is similar to that between an algebra of unbounded closed/closable operators and the von Neumann algebra generated by this algebra. In this paper, to prove Them. \ref{lb72}, we first establish an abstract version, Thm. \ref{lb71}, formulated in terms of weak categorical extensions.

In Sections 1 and 2, we review the basic properties of categorical extensions and weak categorical extensions. In Section 3, we describe the canonical braided $*$-functor from the representation category of a haploid commutative $C^*$-Frobenius algebra $Q$ in $\Rep(\mc A)$ (the representation category of a conformal net $\mc A$) to the representation category of the conformal net associated with $Q$, and we prove Thm. \ref{lb71} mentioned above. An analogous braided $*$-functor in the VOA setting is recalled in Section 4. In Section 5, we prove Thm. \ref{lb72} (and its enhancement, Theorem \ref{lb83}) based on Theorem \ref{lb71}. Finally, Section 6 is devoted to the proof of Corollary \ref{lb117}.

\subsection*{Acknowledgment}

I would like to thank Sebastiano Carpi and Andr\'e Henriques for many helpful discussions.

\section{Categorical extensions}\label{lb5}

\subsection{Preliminaries}\label{lb18}

In the notation of a (densely defined) unbounded linear operator $T:\mc H_1\rightarrow\mc H_2$, $\mc H_1$ and $\mc H_2$ are Hilbert spaces, and the domain $T$ is denoted by $\Dom(T)$ (which is a dense subspace of $\mc H_1$).

An interval in the unit circle $\Sbb^1$ means a nonempty nondense connected open subset of $\Sbb^1$. If $I\in\mc J$, we let $I'$ be the interior of $\Sbb^1\setminus I$.

Let $\mc J$ be the set of intervals in $\Sbb^1$. Let $\mc A$ be an irreducible local conformal covariant net on  $\Sbb^1$, or a \textbf{conformal net} for short. The vacuum Hilbert space for $\mc A$ is denoted by $\mc H_0$. The vacuum vector is noted by $\Omega$. All representations (i.e. modules) of $\mc A$ are understood to be separable and normal. Let $\Rep(\mc A)$ be the $C^*$-category of $\mc A$-modules whose hom spaces are
\begin{align*}
\Hom_{\mc A}(\mc H_i,\mc H_j)=\{&\text{bounded linear }T:\mc H_i\rightarrow H_j \text{ such that }\\
&T\pi_{i,I}(x)=\pi_{j,I}(x)T\text{ for all } I\in\mc J,x\in\mc A(I)\}
\end{align*}
Define the universal cover $\scr G=\wtd{\Diffp}(\Sbb^1)$ where $\Diff^+(\Sbb^1)$ is the group of the orientation preserving diffeomorphisms of $\Sbb^1$.  For each $I\in\mc J$, let $\scr G(I)\subset\scr G$ be the branch containing $1$ of the inverse image of $\{g\in\Diff^+(\Sbb^1):g\text{ has support in }I\}$ under the covering map $\scr G\rightarrow\Diff^+(\Sbb^1)$.

Let $\scr C$ be a full ($C^*$-)subcategory of $\Rep(\mc A)$ containing $\mc H_0$ and closed under taking submodules and finite (orthogonal) direct sums. Choose a tensor $*$-bifunctor $\boxdot$ on $\scr C$, together with unitary associators $(\mc H_i\boxdot\mc H_j)\boxdot\mc H_k\rightarrow\mc H_i\boxdot(\mc H_j\boxdot\mc H_k)$ and unitary unitors $\mc H_0\boxdot\mc H_i\rightarrow\mc H_i,\mc H_i\boxdot\mc H_0\rightarrow\mc H_i$ so that $\scr C$ becomes a $C^*$-tensor category. We identify the domain and the codomain of each of these maps so that the associators and the unitors are identity maps. (In this paper, we are mainly interested in the case where $\scr C=\Rep^V(\mc A_V)$, cf. Def. \ref{lb44}.)

An \textbf{arg-valued interval} is of the form $\wtd I=(I,\arg_I)$ where $I\in\mc J$ and $\arg_I$ is an arg function on $I$, i.e. a continuous function $\arg_I:I\rightarrow\Rbb$ such that $\arg(e^{\im t})-t\in 2\pi\Zbb$ where $e^{\im t}\in I$. The set of arg-valued intervals is denoted by $\wtd{\mc J}$. If $\wtd I=(I,\arg_I)$ and $\wtd J=(J,\arg_J)$ are arg-valued functions, we write 
\begin{align*}
\wtd I\subset\wtd J\qquad\Leftrightarrow\qquad I\subset J\text{ and }\arg_J|_I=\arg_I
\end{align*}
We say that $\wtd I$ and $\wtd J$ are \textbf{disjoint} if $I,J$ are disjoint; we say that $\wtd I$ is \textbf{anticlockwise} to $\wtd J$ (or equivalently, $\wtd J$ is \textbf{clockwise} to $\wtd I$), if $\arg_J<\arg_I<\arg_J+2\pi$ (in particular, $I$ and $J$ are disjoint). We define the 
\textbf{clockwise complement} of $\wtd I=(I,\arg_I)$ to be
\begin{align}
\wtd I'=(I',\arg_{I'})\qquad\text{where }\wtd I'\text{ is clockwise to }\wtd I
\end{align}
Define the \textbf{anticlockwise complement} of $\wtd I=(I,\arg_I)$ such that
\begin{align}
(\bpr\wtd I)'=\wtd I
\end{align}

According to \cite{Hen19}, every $\mc A$-module $\mc H_i$ is \textbf{conformal covariant}, which means that there is a (necessarily unique) strongly continuous unitary representation $U_i$ of $\scr G_{\mc A}$ on $\mc H_i$ satisfying that for each $I\in\mc J$ and each $g\in\scr G_{\mc A}(I)$, noting that $U(g)$ is defined according to the definition of conformal nets and is in $\mc A(I)$, we have
\begin{align}\label{eq9}
U_i(g)=\pi_i(U(g))
\end{align}
Here $\scr G_{\mc A}$ is the canonical central extension of $\scr G$ associated to $\mc A$:
\begin{gather*}
\scr G_{\mc A}=\{(g,V):g\in\scr G,V\text{is a unitary on $\mc H_0$ representing $U(g)$}\}\\
\scr G_{\mc A}(I)=\text{the inverse image of $\scr G(I)$ under $\scr G_{\mc A}\rightarrow\scr G$}
\end{gather*}
Since the projectively unitary representations of $\PSU$ is lifted in a unique way to a (strongly continuous) unitary representation of $\UPSU$ \cite{Bar54}, $\UPSU$ is naturally a subgroup of $\scr G_{\mc A}$.  See \cite[Sec. 2.1]{Gui21a} for details. 

\begin{rem}
Note that by \cite[Thm. A.1]{Car04}, $\mc A$ contains a Virasoro subnet $\Vir_c$ for some central charge $c\geq0$. It follows that we have a canonical equivalence of topological groups
\begin{align}
\scr G_{\mc A}\simeq\scr G_c\xlongequal{\mathrm{def}} \scr G_{\Vir_c}
\end{align}
(See \cite[Sec. 1]{Gui21c} for a detailed explanation.) We identify $\scr G_{\mc A}$ with $\Gc$ throughout this article.
\end{rem}

\begin{rem}\label{lb50}
The action of $\Diffp(\Sbb^1)$ on $\mc J$ can be lifted continuously to an action of $\scr G$ of $\Jtd$ (by viewing each $\wtd I\in\Jtd$ as an interval of the universal cover $\Rbb$ of $\Sbb^1$), and hence is lifted to an action
\begin{align*}
\scr G_c\curvearrowright\Jtd
\end{align*}
\end{rem}

For each $\mc A$-module $\mc H_i$ and each $I\in\mc J$, we let
\begin{align}
\mc H_i(I)=\Hom_{\mc A(I')}(\mc H_0,\mc H_i)\Omega  \label{eq13}
\end{align}
where $\Hom_{\mc A(I')}(\mc H_0,\mc H_i)$ is the set of bounded linear operators $T:\mc H_0\rightarrow \mc H_i$ satisfying $Tx=\pi_{i,I}(x)T$ for all $x\in\mc A(I')$. Thus $\mc H_0(I)=\mc A(I)\Omega$ by Haag duality.

Recall that two closable operators $A,B$ on a Hilbert space $\mc H$ are said to \textbf{commute strongly} if $\{A\}''$ commutes with $\{B\}''$, where $\{A\}''$ is the smallest von Neumann algebra that $\ovl A$ (the closure of $A$) is affiliated with. (In other words, if we let $\ovl A=UH$ be the polar decomposition where $H$ is self-adjoint and $U$ is a partial isometry, then $\{A\}''$ is the von Neumann algebra generated by $U$ and $(1+H)^{-1}$.) $\{B\}''$ is understood in a similar way. When $A$ is closed and $B$ is bounded (and everywhere defined), then $A$ commutes strongly with $B$ iff $BA\subset AB$ and $B^*A\subset AB^*$, iff there is a core $\scr D$ for $A$ such that $B\cdot A|_{\scr D}\subset A\cdot B$ and $B^*\cdot A|_\Dom\subset A\cdot B^*$. In particular, if $A,B$ are both bounded, then they commute strongly iff they \textbf{commute adjointly} (i.e. $AB=BA$ and $AB^*=B^*A$). See (for example) the Appendix of \cite{Gui26}  for more details. 

\begin{df}
Let $\mc P_0,\mc Q_0, \mc R_0,\mc S_0$ be pre-Hilbert spaces with completions $\mc P,\mc Q,\mc R,\mc S$ respectively. Let  $A:\mc P\rightarrow\mc R,B:\mc Q\rightarrow\mc S,C:\mc P\rightarrow\mc Q,D:\mc R\rightarrow\mc S$ be closable operators whose domains are subspaces of $\mc P_0,\mc Q_0, \mc P_0,\mc R_0$ respectively, and whose ranges are inside $\mc R_0,\mc S_0,\mc Q_0,\mc S_0$ respectively. We say that the diagram of closable operators 	\textbf{commutes strongly}
\begin{equation}\label{eq1}
\begin{tikzcd}
\mc P_0 \arrow[r,"C"] \arrow[d,"A"'] & \mc Q_0 \arrow[d,"B"] \\
\mc R_0 \arrow[r,"D"]           & \mc S_0        
\end{tikzcd}
\end{equation}
if the following holds: Let $\mc H=\mc P\oplus\mc Q\oplus\mc R\oplus\mc S$. Define unbounded closable operators $R,S$ on $\mc H$ with domains $\Dom(R)=\Dom(A)\oplus\Dom(B)\oplus\mc R\oplus \mc S$, $\Dom(S)=\Dom(C)\oplus\mc Q\oplus\Dom(D)\oplus \mc S$, such that
	\begin{gather*}
	R(\xi\oplus\eta\oplus\chi\oplus\varsigma)=0\oplus 0\oplus A\xi\oplus B\eta\qquad(\forall \xi\in\Dom(A),\eta\in\Dom(B),\chi\in\mc R,\varsigma\in \mc S),\\
	S(\xi\oplus\eta\oplus\chi\oplus\varsigma)=0\oplus C\xi\oplus 0\oplus D\chi   \qquad(\forall \xi\in\Dom(C),\eta\in\mc Q,\chi\in \Dom(D),\varsigma\in\mc S).
	\end{gather*}
Then (the closures of) $R$ and $S$ commute strongly.

When the four pre-Hilbert spaces are complete and the four closable operators are bounded, and when \eqref{eq1} commutes strongly, we also say that \eqref{eq1} \textbf{commutes adjointly}, since we have $DA=BC$ and $D^*B=AC^*$.
\end{df}

\begin{lm}\label{lb35}
Let $D$ be a self-adjoint positive (closed) operator on a Hilbert space $\mc H$. Let $\mc H^\infty=\bigcap_{n\in\Nbb}\Dom(D^n)$. Let $A_1,\dots,A_m$ and $B_1,\dots,B_n$ be closable operators on $\mc H$ with common domain $\mc H^\infty$. Suppose that for every $1\leq i\leq m$ and $1\leq j\leq n$ we have that $A_i\mc H^\infty\subset \mc H^\infty$ and $B_j\mc H^\infty\subset\mc H^\infty$, and that there exists $\eps>0$ such that $e^{\im tD}A_i e^{-\im tD}$ commutes strongly with $B_j$ for all $t\in (-\eps,\eps)$. Let $A$ be a polynomial of $A_1,\dots,A_n$, let $B$ be a polynomial of $B_1,\dots,B_n$, and assume that $A$ and $B$ (with domain $\mc H^\infty$) are closable. Then $A$ commutes strongly with $B$.
\end{lm}

\begin{proof}
This is a standard technique in the literature. See for example \cite[Lem. 4.17]{Gui21a}.
\end{proof}

Strongly commuting operators are adjointly commuting:
\begin{lm}\label{lb14}
Let $S,T$ be strongly commuting closed operators on a Hilbert space $\mc K$. Let $\xi\in\Dom(ST)\cap\Dom(S)$. Then $\xi\in\Dom(TS)$, and $ST\xi=TS\xi$.
\end{lm}

It follows that if $\Dom_0$ is a common invariant domain of strongly commuting $S,T$, then $ST|_{\Dom_0}=TS|_{\Dom_0}$.

\begin{proof}
This is well-known and can be proved easily. See for example \cite[Prop. A.5]{Gui26}.
\end{proof}

\subsection{The tensorators associated to categorical extensions}

Let us recall the definition of categorical extensions of $\mc A$ \cite{Gui21a}.

\begin{df}\label{lb31}
A \textbf{closed and vector labeled categorical extension} (or a \textbf{categorical extension}  for short\footnote{There are non-closed and non vector-labeled categorical extensions in \cite{Gui21a}. In this paper, we only consider closed and vector-labeled ones.}) of $\mc A$  in $\scr C$ is a quintuple $\scr E=(\mc A,\scr C,\boxdot,\ss,\mc H)$ associating to each $\wtd I\in\wtd{\mc J}$, $\mc H_i,\mc H_k\in\Obj(\scr C)$, and $\xi\in\mc H_i(I)$, bounded linear operators
\begin{align*}
L(\xi,\wtd I)\in\Hom_{\mc A(I')}(\mc H_k,\mc H_i\boxdot\mc H_k)\qquad R(\xi,\wtd I)\in\Hom_{\mc A(I')}(\mc H_k,\mc H_k\boxdot\mc H_i)
\end{align*}
called the \textbf{L operators} and the \textbf{R operators}, satisfying the following conditions:
\begin{enumerate}[label=(\alph*)]
\item (Isotony) If $\wtd I_1\subset\wtd I_2\in\Jtd$, and $\xi\in\mc H_i(I_1)$, then $L(\xi,\wtd I_1)=L(\xi,\wtd I_2)$, $R(\xi,\wtd I_1)=R(\xi,\wtd I_2)$ when acting on any  $\mc H_k\in\Obj(\scr C)$.
\item (Naturality) If $\mc H_i,\mc H_j,\mc H_{j'}\in\Obj(\scr C)$, $\wtd I\in\Jtd$, $G\in\Hom_{\mc A}(\mc H_j,\mc H_{j'})$,  $\xi\in\mc H_i(I)$, and $\eta\in\mc H_j$, then
\begin{align}
	(\idt_i\boxdot G)L(\xi,\wtd I)\eta=L(\xi,\wtd I)G\eta,\qquad (G\boxdot \idt_i)R(\xi,\wtd I)\eta=R(\xi,\wtd I)G\eta.
\end{align}
\item (State-field correspondence) For any $\mc H_i\in\Obj(\scr C)$, under the identifications $\mc H_i=\mc H_i\boxdot\mc H_0=\mc H_0\boxdot\mc H_i$ defined by the unitors, the relation
\begin{align}
	L(\xi,\wtd I)\Omega=R(\xi,\wtd I)\Omega=\xi\label{eq14}
\end{align}
holds for any $\wtd I\in\Jtd,\xi\in\mc H_i(I)$. It follows immediately that when acting on $\mc H_0$, $L(\xi,\wtd I)$ equals $R(\xi,\wtd I)$ and is independent of $\arg_I$.
\item (Density of fusion products) If $\mc H_i,\mc H_k\in\Obj(\scr C),\wtd I\in\Jtd$, then the set $L(\mc H_i(I),\wtd I)\mc H_k$ spans a dense subspace of $\mc H_i\boxdot\mc H_k$, and $R(\mc H_i(I),\wtd I)\mc H_k$ spans a dense subspace of $\mc H_k\boxdot\mc H_i$.\footnote{Indeed, they are equal to the full space $\mc H_i\boxdot\mc H_k$ and $\mc H_k\boxdot\mc H_i$ respectively by the fact that $\mc A(I)$ is a type III factor. See \cite[Lem. 6.1]{Gui21a}.}
\item (Locality) For any $\mc H_k\in\Obj(\scr C)$, disjoint $\wtd I,\wtd J\in\Jtd$ with $\wtd I$ anticlockwise to $\wtd J$, and any $\xi\in\mc H_i(I),\eta\in\mc H_j(J)$, the following diagram commutes adjointly.
\begin{equation}
\begin{tikzcd}
\quad \mc H_k\quad \arrow[rr,"{R(\eta,\wtd J)}"] \arrow[d, "{L(\xi,\wtd I)}"'] &&\quad \mc H_k\boxdot\mc H_j\quad \arrow[d, "{L(\xi,\wtd I)}"]\\
\mc H_i\boxdot\mc H_k\arrow[rr,"{R(\eta,\wtd J)}"] &&\mc H_i\boxdot\mc H_k\boxdot\mc H_j
\end{tikzcd}
\end{equation}
\item (Braiding) The operation $\ss$ (called the \textbf{braiding operator/operation}) associates to $\mc H_i,\mc H_j$ a unitary linear map $\ss_{i,j}:\mc H_i\boxdot\mc H_j\rightarrow\mc H_j\boxdot \mc H_i$ such that  
\begin{align}
\ss_{i,j} L(\xi,\wtd I)\eta=R(\xi,\wtd I)\eta\label{eq8}
\end{align}
whenever $\wtd I\in\Jtd,\xi\in\mc H_i(I)$, $\eta\in\mc H_j$.
\end{enumerate}
When we want to stress that the left and right operators $L,R$ depend on the choice of $\scr C$ and $\boxdot$, we write them as $L^\boxdot,R^\boxdot$. 

We say that $\scr E$ is \textbf{conformal covariant} if for every $g\in\Gc,\wtd I\in\Jtd,\mc H_i\in\Obj(\Rep(\mc A)),\xi\in\mc H_i(I)$, there exists a (necessarily unique) element denoted by $g\xi g^{-1}\in\mc H_i(gI)$ satisfying
\begin{align}
L(g\xi g^{-1},g\wtd I)=gL(\xi,\wtd I)g^{-1}\qquad R(g\xi g^{-1},g\wtd I)=gR(\xi,\wtd I)g^{-1}
\end{align}
when acting on any $\mc H_j\in\Obj(\Rep(\mc A))$. (Recall Rem. \ref{lb50} for the meaning of $g\wtd I$.) We say that $\scr E$ is \textbf{M\"obius covariant} if the above condition holds for all $g\in\UPSU$. (In this case, we write $g\xi g^{-1}$ as $g\xi$ if $g\in\UPSU$.)   \hfill\qedsymbol
\end{df}

\begin{rem}
Though we will not use this fact in this paper, we note that the braiding operator $\ss$ is automatically a homomorphism of $\mc A$-modules and makes $(\scr C,\boxdot,\ss)$ a braided $C^*$-tensor category. See \cite[Thm. 3.9]{Gui21a}.
\end{rem}

From the fact that $L(\xi,\wtd I),R(\xi,\wtd I)$ intertwine the actions of $\mc A(I')$, it is easy to see that if $x\in\mc A(I)$, then
\begin{align}\label{eq3}
L(x\Omega,\wtd I)|_{\mc H_k}=R(x\Omega,\wtd I)|_{\mc H_k}=\pi_{k,I}(x)
\end{align}

\begin{thm}\label{lb12}
There is a canonical (closed and vector labeled) conformal covariant categorical extension
\begin{align*}
\scr E_{\mathrm{Connes}}=(\mc A,\Rep(\mc A),\boxtimes_{\mc A},\mbb B^{\mc A},\mc H)
\end{align*}
defined by Connes fusion product. $\scr E_{\mathrm{Connes}}$ is called the \textbf{Connes categorical extension}. We
\begin{align*}
\text{abbreviate $\boxtimes_{\mc A}$ to $\boxtimes$ and $\mbb B^{\mc A}$ to $\mbb B$}
\end{align*}
when no confusion arises. Then
\begin{align*}
(\Rep(\mc A),\boxtimes_{\mc A},\mbb B^{\mc A})\equiv (\Rep(\mc A),\boxtimes,\mbb B)
\end{align*}
is a braided $C^*$-tensor category, called the \textbf{Connes braided $C^*$-tensor category}.
\end{thm}

\begin{proof}
This is \cite[Thm. 3.4, 3.5]{Gui21a}. See \cite{Gui21a} for details, and \cite[Sec. A, B]{Gui21b} for a sketch of the construction.
\end{proof}

\begin{thm}\label{lb1}
Let $\scr E=(\mc A,\scr C,\boxdot,\ss,\mc H)$ be a categorical extension. Then $\scr E$ can be embedded canonically into the Connes categorical extension $\scr E_{\mathrm{Connes}}=(\mc A,\Rep(\mc A),\boxtimes,\mbb B,\mc H)$. More precisely, this means the following:

There is a natural $\Psi$ (called the \textbf{tensorator associated to $\scr E$}) associating to each $\mc H_i,\mc H_j\in\Obj(\scr C)$ a unitary isomorphism of $\mc A$-modules
\begin{align}
\Psi=\Psi_{i,j}:\mc H_i\boxdot\mc H_j\rightarrow\mc H_i\boxtimes\mc H_j
\end{align}
(where ``natural" means  $\Psi_{i',j'}(F\boxdot G)=(F\boxtimes G)\Psi_{i,j}$ for all $F\in\Hom_{\mc A}(\mc H_i,\mc H_{i'})$ and $G\in\Hom_{\mc A}(\mc H_j,\mc H_{j'})$) such that 
\begin{align}\label{eq90}
(\id,\Psi):(\scr C,\boxdot,\ss)\rightarrow(\Rep(\mc A),\boxtimes,\mbb B)
\end{align}
is a braided $*$-functor, i.e., for each $\mc H_i,\mc H_j,\mc H_k\in\Obj(\scr C)$ the following three statements are true:
\begin{enumerate}[label=(\alph*)]
\item The following diagram commutes.
\begin{equation}
\begin{tikzcd}[column sep=large]
\mc H_i\boxdot\mc H_j\boxdot\mc H_k \arrow[r,"\Psi_{i\boxdot k,j}"] \arrow[d,"\Psi_{i,k\boxdot j}"'] &  \arrow[d,"\Psi_{i,k}\boxtimes\idt_j"](\mc H_i\boxdot\mc H_k)\boxtimes\mc H_j \\
\mc H_i\boxtimes(\mc H_k\boxdot\mc H_j) \arrow[r,"\idt_i\boxtimes \Psi_{k,j}"]           & \mc H_i\boxtimes\mc H_k\boxtimes\mc H_j          
\end{tikzcd}
\end{equation}
\item The following two maps equal the identity map $\idt_i:\mc H_i\rightarrow\mc H_i$.
\begin{subequations}\label{eq61}
	\begin{gather}
	\mc H_i\simeq \mc H_0\boxdot\mc H_i\xrightarrow{\Psi_{0,i}} \mc H_0\boxtimes\mc H_i\simeq\mc H_i,\\
	\mc H_i\simeq \mc H_i\boxdot\mc H_0\xrightarrow{\Psi_{i,0}} \mc H_i\boxtimes\mc H_0\simeq\mc H_i. 
	\end{gather}
\end{subequations}
\item The following diagram commutes.
\begin{equation*}
\begin{tikzcd}
\mc H_i\boxdot\mc H_j \arrow[d,"\Psi_{i,j}"'] \arrow[r,"\ss_{i,j}"] & \mc H_j\boxdot\mc H_i \arrow[d,"\Psi_{j,i}"] \\
\mc H_i\boxtimes\mc H_j \arrow[r,"\mbb B_{i,j}"]           & \mc H_j\boxtimes\mc H_i          
\end{tikzcd}
\end{equation*}
\end{enumerate}
Moreover, for any $\wtd I\in\wtd{\mc J},\mc H_i,\mc H_j\in\Obj(\scr C),\xi\in\mc H_i(I)$ we have
\begin{subequations}\label{eq12}
\begin{gather}
\Psi_{i,j}L^{\boxdot}(\xi,\wtd I)|_{\mc H_j}=L^{\boxtimes}(\xi,\wtd I)|_{\mc H_j}\\
\Psi_{j,i}R^{\boxdot}(\xi,\wtd I)|_{\mc H_j}=R^{\boxtimes}(\xi,\wtd I)|_{\mc H_j}
\end{gather}
\end{subequations}
\end{thm}

\begin{proof}
This is \cite[Thm. 3.10]{Gui21a}.
\end{proof}

\begin{rem}\label{lb38}
The braided $*$-functor \eqref{eq90} clearly restricts to an equivalence of braided $C^*$-tensor categories $(\scr C,\boxdot,\ss)\xlongrightarrow{\simeq}(\wht{\scr C},\boxtimes,\mathbb B)$ where $\wht{\scr C}$ is the \textbf{repletion of $\scr C$}, i.e., the full $C^*$-subcategory of all $\mc H_i\in\Obj(\Rep(\mc A))$ unitarily equivalent to some object in $\scr C$. (In particular, $\wht{\scr C}$ is closed under $\boxtimes$.)
\end{rem}

\begin{rem}\label{lb99}
It is illuminating to write \eqref{eq12} as
\begin{align}
\Psi\circ\scr E=\scr E_\Connes|_{\scr C}\subset\scr E_{\mathrm{Connes}}
\end{align}
Then Thm. \ref{lb1} can be summarized by the elegant statement: for every categorical extension $\scr E$ there is a natural unitary $\Psi$ making \eqref{eq90} a braided $*$-functor such that $\Psi\circ\scr E=\scr E_{\mathrm{Connes}}|_{\scr C}$. 

It is easy to see that the converse is also true: Suppose that we have a braided $*$-functor \eqref{eq90} where $\Psi$ is natural and unitary, then  $\Psi^*\circ\scr E_{\mathrm{Connes}}|_{\scr C}$ is a categorical extension. (Namely, \eqref{eq12} defines a categorical extension with L and R operators $L^\boxdot,R^\boxdot$.)
\end{rem}

\begin{rem}\label{lb20}
Since $\scr E_\Connes$ is conformal covariant, by Thm. \ref{lb1}, we see that any (closed and vector-labeled) categorical extension is conformal covariant.
\end{rem}

\subsection{Left and right operators}\label{lb22}

\begin{df}\label{lb29}
For each $\mc H_i\in\Obj(\Rep(\mc A))$ and $I\in\mc J$, we let $\mc H_i^\pr(I)$ be the set of all $\xi\in\mc H_i$ such that the following densely defined linear map is closable (preclosed):
\begin{align}
x\Omega\in\mc A(I')\Omega\qquad\mapsto\qquad x\xi\in\mc H_i
\end{align}
Clearly $\mc H_i(I)\subset\mc H_i^\pr(I)$.
\end{df}

\begin{thm}\label{lb7}
Let $\scr E=(\mc A,\scr C,\boxdot,\ss,\mc H)$ be a categorical extension. Then there exist unique operations $\scr L^\boxdot$ and $\scr R^\boxdot$ (or simply $\scr L,\scr R$) associating to each $\mc H_i\in\Obj(\scr C),\wtd I\in\Jtd,\xi\in\mc H_i^\pr(I)$, and each $\mc H_k\in\Obj(\scr C)$, closable operators
\begin{gather*}
\scr L(\xi,\wtd I)|_{\mc H_k}:\mc H_k\rightarrow\mc H_i\boxdot\mc H_k \qquad \scr R(\xi,\wtd I)|_{\mc H_k}:\mc H_k\rightarrow\mc H_k\boxdot\mc H_i
\end{gather*}
with common core $\mc H_i(I')$ such that the following conditions are satisfied for all $\mc H_i,\mc H_j,\mc H_k,\mc H_{k'}\in\Obj(\scr C)$:
\begin{enumerate}[label=(\alph*)]
\item If $\xi\in\mc H_i(I)$, then $\scr L(\xi,\wtd I)=L(\xi,\wtd I)$ and $\scr R(\xi,\wtd I)=R(\xi,\wtd I)$ (when acting on any object of $\scr C$).
\item (State-field correspondence) If $\wtd I\in\Jtd$ and $\xi\in\mc H_i^\pr(I)$, then $\scr L(\xi,\wtd I)\Omega=\scr R(\xi,\wtd I)\Omega=\xi$.
\item (Isotony)  If $\wtd I_1\subset\wtd I_2\in\Jtd$, and $\xi\in\mc H_i^\pr(I_1)$, then $\scr L(\xi,\wtd I_1)\supset \scr L(\xi,\wtd I_2)$, $\scr R(\xi,\wtd I_1)\supset \scr R(\xi,\wtd I_2)$ when acting on   $\mc H_k$.
	
\item (Naturality) If $G\in\Hom_{\mc A}(\mc H_k,\mc H_{k'})$, then for any $\wtd I\in\Jtd,\xi\in\mc H_i^\pr(I)$, the following diagrams of closed operators commute strongly.
\begin{equation*}
\begin{tikzcd}
\mc H_k \arrow[d,"{\scr L(\xi,\wtd I)}"'] \arrow[r,"G"] & \mc H_{k'} \arrow[d,"{\scr L(\xi,\wtd I)}"] \\
\mc H_i\boxdot\mc H_k \arrow[r,"\idt_i\boxdot G"]           & \mc H_i\boxdot\mc H_{k'}          
\end{tikzcd}
\qquad
\begin{tikzcd}
\mc H_k \arrow[r,"{\scr R(\xi,\wtd I)}"] \arrow[d,"G"'] & \mc H_k\boxdot\mc H_i \arrow[d,"G\boxdot\idt_i"] \\
\mc H_{k'} \arrow[r,"{\scr R(\xi,\wtd I)}"]           & \mc H_{k'}\boxdot\mc H_i        
\end{tikzcd}
\end{equation*}

\item (Locality) For any  disjoint $\wtd I,\wtd J\in\Jtd$ with $\wtd J$ clockwise to $\wtd I$, and any $\xi\in\mc H_i^\pr(I),\eta\in\mc H_j^\pr(J)$, the following diagram  commutes strongly.
\begin{equation}\label{eq2}
\begin{tikzcd}[column sep=huge]
\quad \mc H_k\quad \arrow[r,"{\scr R(\eta,\wtd J)}"] \arrow[d, "{\scr L(\xi,\wtd I)}"'] &\quad \mc H_k\boxdot\mc H_j\quad \arrow[d, "{\scr L(\xi,\wtd I)}"]\\
\mc H_i\boxdot\mc H_k\arrow[r,"{\scr R(\eta,\wtd J)}"] &\mc H_i\boxdot\mc H_k\boxdot\mc H_j
\end{tikzcd}
\end{equation}
	
\item (Braiding) For any $\wtd I\in\Jtd,\xi\in\mc H_i^\pr(I)$, we have
	\begin{align}
	\ss_{i,j} \scr L(\xi,\wtd I)|_{\mc H_j}=\scr R(\xi,\wtd I)|_{\mc H_j}.
	\end{align}
	
\item (M\"obius covariance) For any $g\in\UPSU,\wtd I\in\wtd{\mc J},\xi\in\mc H_i^\pr(I)$, we have $g\xi\in\mc H_i^\pr(g I)$, and
	\begin{align}
	\scr L(g\xi,g\wtd I)=g\scr L(\xi,\wtd I)g^{-1},\qquad \scr R(g\xi,g\wtd I)=g\scr R(\xi,\wtd I)g^{-1}
	\end{align}
	when acting on $\mc H_j$.
\end{enumerate}
\end{thm}

\begin{proof}
It was proved in \cite[Sec. 2.2]{Gui26} that $\scr L^\boxtimes$ and $\scr R^\boxtimes$ exist for the Connes categorical extension. Pulling back these operations using the unitary operation $\Psi$ in Thm. \ref{lb1}, i.e., defining 
\begin{align}\label{eq10}
\scr L^\boxdot(\xi,\wtd I)|_{\mc H_k}=\Psi_{i,k}^{-1}\scr L^\boxtimes(\xi,\wtd I)|_{\mc H_k}\qquad  \scr R^\boxdot(\xi,\wtd I)|_{\mc H_k}=\Psi_{k,i}^{-1}\scr R^\boxtimes(\xi,\wtd I)|_{\mc H_k}
\end{align}
we obtain the $\scr L^\boxdot$ and $\scr R^\boxdot$ operations for $\scr E$. (It is easy to see that they satisfy all the axioms in the theorem. For example, the locality can be proved in the same way as in \cite[Cor. 2.36]{Gui26}.)

To prove the uniqueness, we only need the state-field correspondence and the locality (among all the axioms): Choose any $\xi\in\mc H_i^\pr(I)$ and $\eta\in\mc H_j(I')$. Recall that $\wtd I'$ is the clockwise complement of $\wtd I$. By state-field correspondence and \eqref{eq2} (setting $\mc H_k=\mc H_0$), we get
\begin{align}\label{eq4}
\scr L(\xi,\wtd I)\eta=R(\eta,\wtd I')\xi
\end{align}
Since $\mc H(I')$ is a core for $\scr L(\xi,\wtd I)|_{\mc H_j}$, we know that $\scr L(\xi,\wtd I)$ is uniquely determined. The same can be said about $\scr R$.
\end{proof}

\begin{rem}\label{lb15}
In Thm. \ref{lb7}, the following \textbf{intertwining property} is satisfied: For any $\mc H_i\in\Obj(\scr C),\wtd I\in\Jtd,\xi\in\mc H_i(I),x\in\mc A(I')$, the following diagrams commute strongly:
\begin{equation*}
\begin{tikzcd}[column sep=large]
\mc H_k \arrow[d,"{\scr L(\xi,\wtd I)}"'] \arrow[r,"{\pi_{k,I'}(x)}"] & \mc H_k\arrow[d,"{\scr L(\xi,\wtd I)}"] \\
\mc H_i\boxdot\mc H_k \arrow[r,"{\pi_{i\boxdot k,I'}(x)}"]           & \mc H_i\boxdot\mc H_k        
\end{tikzcd}
\qquad
\begin{tikzcd}[column sep=large]
\mc H_k\arrow[r,"{\scr R(\xi,\wtd I)}"] \arrow[d,"{\pi_{k,I'}(x)}"'] &\mc H_k\boxdot\mc H_i \arrow[d,"{\pi_{k\boxdot i,I'}(x)}"] \\
\mc H_k \arrow[r,"{\scr R(\xi,\wtd I)}"]           & \mc H_k\boxdot\mc H_i    
\end{tikzcd}
\end{equation*}
This is due to condition (a) and the locality in Thm. \ref{lb7}, together with \eqref{eq3}.
\end{rem}

\begin{df}\label{lb2}
Let $\scr E=(\mc A,\scr C,\boxdot,\ss,\mc H)$ be a categorical extension. Let $\mc H_i\in\Obj(\scr C)$. A \textbf{left operator}  of $\scr E$ with charge space $\mc H_i$ and localized in $\wtd I$ is an operation $\fk L(\xk,\wtd I)$ where $\xk$ is a formal element, $\wtd I\in\Jtd$, and for any $\mc H_k\in\Obj(\scr C)$ we have a closed operator $\fk L(\xk,\wtd I)|_{\mc H_k}:\mc H_k\rightarrow\mc H_i\boxdot\mc H_k$ satisfying
\begin{subequations}\label{eq17}
\begin{align}
\fk L(\xk,\wtd I)|_{\mc H_k}=\scr L^\boxdot(\xi,\wtd I)|_{\mc H_k}
\end{align}
for some $\xi\in\mc H_i^\pr(I)$ independent of $\mc H_k$.

Let $\mc H_j\in\Obj(\scr C)$. Similarly, a \textbf{right operator} with charge space $\mc H_j$ is an operation $\fk R(\yk,\wtd J)$ where $\yk$ is a formal element, $\wtd J\in\Jtd$, and for any $\mc H_k\in\Obj(\scr C)$ we have a closed operator $\fk R(\yk,\wtd J)|_{\mc H_k}:\mc H_k\rightarrow\mc H_k\boxdot\mc H_j$ with core $\mc H_k(J')$ satisfying \begin{align}
\fk R(\yk,\wtd J)|_{\mc H_k}=\scr R^\boxdot(\eta,\wtd J)|_{\mc H_k}
\end{align}
for some $\eta\in\mc H_j^\pr(J)$ independent of $\mc H_k$.
\end{subequations}
\end{df}

\begin{rem}
The $\xi$ resp. $\eta$ making \eqref{eq17} true must be unique. Indeed, if we apply \eqref{eq17} to the vacuum vector $\Omega$, by the state-field correspondence in Thm. \ref{lb7}, we have 
\begin{align}
\xi=\fk L(\xi,\wtd I)\Omega\qquad \eta=\fk R(\xi,\wtd J)\Omega
\end{align}
\end{rem}

\begin{rem}\label{lb23}
Thus, roughly speaking, left and right operators are almost the same as those of the forms $\scr L(\xi,\wtd I)$ and $\scr R(\eta,\wtd J)$. The only difference is that general left and right operators are not vector-labeled. Thus, we call operators of the forms $\scr L(\xi,\wtd I)$ and $\scr R(\eta,\wtd J)$ respectively \textbf{vector-labeled left and right operators} (or simply $\scr L$ and $\scr R$ operators) of $\scr E$.
\end{rem}

The following noteworthy example was used in the proof of \cite[Prop. 3.12]{Gui26}. 

\begin{eg}\label{lb85}
Let $\wtd I\in\Jtd$. Then the left operators of $\scr E=(\mc A,\scr C,\boxdot,\ss,\mc H)$ with charge space $\mc H_0$ localized in $\wtd I$ are also the right operators of the same type, and vice versa. They are precisely of the form
\begin{align}\label{eq66}
\fk L(\xk,\wtd I)|_{\mc H_i}=\fk R(\xk,\wtd I)|_{\mc H_i}=\pi_{i,I}(X)
\end{align}
for all $\mc H_i\in\Obj(\scr C)$ where $X$ is a closed operator on $\mc H$ with core $\mc H(I')$, and $X$ is affiliated with $\mc A(I)$.
\end{eg}

\begin{rem}\label{lb84}
Recall that $\pi_{i,I}(X)$ is defined as follows: Let $X=UH$ be the polar decomposition where $U$ is the partial isometry and $H$ is positive. Then $\pi_{i,I}(X)=\pi_{i,I}(U)\pi_{i,I}(H)$ where $\pi_{i,I}(H)$ is the unique positive operator on $\mc H_i$ such that $(1+\pi_{i,I}(H))^{-1}=\pi_{i,I}((1+H)^{-1})$.
\end{rem}

\begin{proof}[Proof of Exp. \ref{lb85}]
Suppose that $\fk L(\xk,\wtd I)=\scr L(\xi,\wtd I)$ is a left operator where $\xi\in\mc H_0^\pr(I)$. Then by \eqref{eq3} and the locality of $\scr E$, for each $y\Omega\in\mc A(I')\Omega=\mc H_0(I')$ we have $\scr L(\xi,\wtd I)y\Omega=y\scr L(\xi,\wtd I)\Omega=y\xi$. This implies $y\scr L(\xi,\wtd I)|_{\mc H_0(I')}\subset\scr L(\xi,\wtd I)y$ for all $y\in\mc A(I')$, and hence that $X=\scr L(\xi,\wtd I)|_{\mc H_0}$ is affiliated with $\mc A(I)$. By locality again, for every $\eta\in\mc H_i(I')$, the following diagram commutes strongly:
\begin{equation*}
\begin{tikzcd}[column sep=huge]
\mc H_0 \arrow[r,"X"] \arrow[d,"{R(\eta,\wtd I')}"'] & \mc H_0 \arrow[d,"{R(\eta,\wtd I')}"] \\
\mc H_i  \arrow[r,"{\scr L(\xi,\wtd I)|_{\mc H_i}}"]           & \mc H_i       
\end{tikzcd}
\end{equation*}
Since the same is true when $\scr L(\xi,\wtd I)|_{\mc H_i}$ is replaced by $\pi_{i,I}(X)$, by the density of fusion product in $\scr E$ (or by choosing $\eta$ such that $R(\eta,\wtd I')$ is unitary, cf. \cite[Lem. 6.1]{Gui21a}), we conclude $\scr L(\xi,\wtd I)|_{\mc H_i}=\pi_{i,I}(X)$.

Conversely, let $X$ be as described in Exp. \ref{lb85}. Since $yX\subset Xy$ for all $y\in\mc A(I')$, it is clear that $\xi=X\Omega$ belongs to $\mc H_0^\pr(I)$ (since $y\Omega\in\mc H_0(I')\mapsto y\xi$ has closure $X$). Since $X$ and $\scr L(\xi,\wtd I)$ both have core $\mc H(I')$, and since $\scr L(\xi,\wtd I)y\Omega=y\xi=yX\Omega=Xy\Omega$, we conclude $X=\scr L(\xi,\wtd I)|_{\mc H_0}$. By the first paragraph, we conclude $\scr L(\xi,\wtd I)|_{\mc H_i}=\pi_{i,I}(X)$. This proves a half of \eqref{eq66}. The other half (about right operators) can be proved in a similar way.
\end{proof}

\section{Weak categorical extensions}

\subsection{Preliminaries}

Let $\varrho:\Rbb\rightarrow\UPSU$ be the one-parameter rotation group. For each $\mc H_i\in\Obj(\Rep(\mc A))$, let $\ovl L_0$ be the generator of the one-parameter unitary group $U_i\circ\varrho$, i.e. $U_i\circ\varrho(t)=e^{\im t\ovl L_0}$. Then $\ovl L_0\geq0$ by \cite[Thm. 3.8]{Wei06}. We let 
\begin{align}
\mc H_i^\infty=\bigcap_{n\in\Nbb}\Dom(\ovl{L_0}^n)\label{eq43}
\end{align}

\begin{df}
A densely defined linear operator $T:\mc H_i\rightarrow\mc H_j$ is called \textbf{smooth} if $\mc H_i^\infty\subset\Dom(T)$, if $\mc H_j^\infty\subset\Dom(T^*)$, if $T\mc H_i^\infty\subset\mc H_j^\infty$, and if $T^*\mc H_j^\infty\subset\mc H_i^\infty$. In particular, a smooth operator is closable since $T^*$ (whose domain is the set of all $\eta\in\mc H_j$ such that $\xi\in\mc H_i\mapsto\bk{T\xi,\eta}$ is bounded) is densely defined. 
\end{df}

\begin{rem}\label{lb3}
By the general fact $(A+B)^*\supset A^*+B^*$ and $(AB)^*\supset B^*A^*$ for densely defined linear operators, it is obvious that products of smooth operators are smooth, and that linear combinations of smooth operators are smooth. Elements in $\Hom_{\mc A}(\mc H_i,\mc H_j)$ are smooth since they intertwine the actions of $\varrho$ by \eqref{eq9}.
\end{rem}

\begin{df}\label{lb33}
For each $\mc H_i\in\Obj(\Rep(\mc A))$ and $I\in\mc J$, we let 
\begin{align*}
\mc H_i^\infty(I)=\{\xi\in\mc H_i(I):L^\boxtimes(\xi,\wtd I)|_{\mc H_j}:\mc H_j\rightarrow\mc H_i\boxtimes\mc H_j\text{ is smooth }\forall \mc H_j\in\Obj(\Rep(\mc A))\}
\end{align*}
Here $L^\boxtimes$ is the L operation of the Connes categorical extension $\scr E_{\mathrm{Connes}}$. Then by the braiding axiom \eqref{eq8},  $R^\boxtimes(\xi,\wtd I)$ is also smooth, since the braiding operator $\mbb B$ is smooth by Rem. \ref{lb3}. Define
\begin{align*}
\mc A^\infty(I)=\{x\in\mc A(I):\pi_{i,I}(x)\text{ is smooth for every }\mc H_i\in\Obj(\Rep(\mc A))\}
\end{align*}
Then by \eqref{eq3}, we clearly have
\begin{align*}
\mc H_0^\infty(I)=\mc A^\infty(I)\Omega
\end{align*}
\end{df}

\begin{rem}
Suppose that $\scr E=(\mc A,\scr C,\boxdot,\ss,\mc H)$ is a categorical extension, and $\mc H_i\in\Obj(\scr C),\xi\in\mc H_i^\infty(I)$. Then $L^\boxdot(\xi,\wtd I)$ and $R^\boxdot(\xi,\wtd I)$ are smooth when acting on any object of $\scr C$ because $L^\boxtimes(\xi,\wtd I)$ and $R^\boxtimes(\xi,\wtd I)$ are smooth, and because $\Psi\circ\scr E\subset\scr E_\Connes$ for some tensorator $\Psi$, cf. Thm. \ref{lb1}.
\end{rem}

\begin{pp}\label{lb4}
Let $\scr E=(\mc A,\scr C,\boxdot,\ss,\mc H)$ be a categorical extension. Let $\mc H_i\in\Obj(\scr C)$, $\wtd I\in\Jtd$. The following are true.
\begin{enumerate}[label=(\alph*)]
\item For each $\xi\in\mc H_i(I)$, there is a sequence $(\xi_n)_{n\in\Zbb_+}$ in $\mc H_i^\infty(I)$ converging to $\xi$ such that $\sup_{n\in\Zbb_+}\Vert L(\xi_n,\wtd I)\Vert\leq\Vert L(\xi,\wtd I)\Vert$, and that $L(\xi_n,\wtd I)|_{\mc H_k}$ converges $*$-strongly to $L(\xi,\wtd I)|_{\mc H_k}$ for any $\mc H_k\in\Obj(\scr C)$.
\item For each $\xi\in\mc H_i^\pr(I)$ and $\mc H_k\in\Obj(\scr C)$, the space $\mc H_k^\infty(I')$ is a core for $\scr L^\boxdot(\xi,\wtd I)|_{\mc H_k}$ and $\scr R^\boxdot(\xi,\wtd I)|_{\mc H_k}$. 
\end{enumerate}
\end{pp}

As a special case of part (a), $\mc A^\infty(I)$ is strongly-$*$ dense in $\mc A(I)$.

\begin{proof}
When $\scr E$ is the Connes categorical extension, this was proved in Prop. 2.21 and 2.23 of \cite{Gui26}. The general case can be proved by pulling back the left and right operators, cf. \eqref{eq10}.
\end{proof}

\begin{df}\label{lb32}
Let $T:\mc H_i\rightarrow\mc H_j$ be a closable smooth operator with dense domain $\Dom(T)=\mc H_i^\infty$. Then $T^*$ sends $\mc H_j^\infty$ into $\mc H_i^\infty$. A dense linear subspace $\Dom_0$ of $\mc H_i^\infty$ is called \textbf{quasi-rotation invariant (QRI)} if there exit $\delta>0$ and a dense linear subspace $\Dom_\delta\subset\Dom_0$ such that $\varrho(t)\Dom_\delta\subset\Dom_0$ for all $t\in(-\delta,\delta)$. We say that $T$ is \textbf{localizable} if every dense and QRI subspace of $\mc H_i^\infty$ is a core for $T$. In particular, if $T$ is localizable, then $\mc H_i^\infty$ is a core for $T$ since $\mc H_i^\infty$ is QRI. Unless otherwise stated, we assume that
\begin{align*}
\text{a smooth localizable $T:\mc H_i\rightarrow\mc H_j$ has domain $\Dom(T)=\mc H_i^\infty$}
\end{align*}
\end{df}

\begin{rem}\label{lb25}
Suppose that a closable operator $T:\mc H_i\rightarrow\mc H_j$ (with domain $\mc H_i^\infty$) is smooth and localizable. Let $\mc H_k\in\Obj(\Rep(\mc A))$ and $A\in\Hom_{\mc A}(\mc H_j,\mc H_k)$. Then $AT$ is also (closable and) smooth and localizable: It is smooth since both $A$ and $T$ are smooth (cf. Rem. \ref{lb3}); it is localizable since $T$ is localizable and $A$ is bounded.
\end{rem}

\begin{df}
Consider the diagram
\begin{equation}\label{eq11}
\begin{tikzcd}
\mc H_i^\infty \arrow[r,"C"] \arrow[d,"A"'] & \mc H_j^\infty \arrow[d,"B"] \\
\mc H_k^\infty \arrow[r,"D"]           & \mc H_l^\infty          
\end{tikzcd}
\end{equation}
where $A:\mc H_i\rightarrow\mc H_k$, $B:\mc H_j\rightarrow\mc H_l$, $C:\mc H_i\rightarrow\mc H_j$, $D:\mc H_k\rightarrow\mc H_l$ are smooth closable operators with domains $\mc H_i^\infty,\mc H_j^\infty,\mc H_i^\infty,\mc H_k^\infty$ respectively. We say that \eqref{eq11} \textbf{commutes adjointly} if $DA=BC$ on $\mc H_i^\infty$ and $CA^*=B^* D$ on $\mc H_k^\infty$.
\end{df}

\subsection{Weak categorical extensions and their closures}

Let $(\scr C,\boxdot)$ be as in Sec. \ref{lb5}. We recall the definition of weak categorical extensions introduced in \cite{Gui26}. The main example considered in this paper is given by Thm. \ref{lb47}.

\begin{df}\label{lb8}
Let $\fk H$ assign, to each $\wtd I\in\Jtd$ and $\mc H_i\in\Obj(\scr C)$, a set $\fk H_i(\wtd I)$  such that $\fk H_i(\wtd I_1)\subset\fk H_i(\wtd I_2)$  whenever $\wtd I_1\subset\wtd I_2$. A \textbf{weak categorical extension} $\scr E^w=(\mc A,\scr C,\boxdot,\ss,\fk H)$ of $\mc A$ associates to any $\mc H_i,\mc H_k\in\Obj(\scr C),\wtd I\in\Jtd,\fk a\in\fk H_i(\wtd I)$, smooth and localizable operators
\begin{gather*}	
\mc L(\fk a,\wtd I):\mc H_k\rightarrow\mc H_i\boxdot\mc H_k\\
\mc R(\fk a,\wtd I):\mc H_k\rightarrow\mc H_k\boxdot\mc H_i
\end{gather*}
with common domain $\mc H_k^\infty$ such that for any $\mc H_i,\mc H_j,\mc H_k,\mc H_{k'}\in\Obj(\scr C)$ and any $\wtd I,\wtd J,\wtd I_1,\wtd I_2\in\Jtd$, the following conditions are satisfied:
\begin{enumerate}[label=(\alph*)]
\item 	(Isotony) If  $\fk a\in\fk H_i(\wtd I_1)$ and $\wtd I_1\subset\wtd I_2$, then $\mc L(\fk a,\wtd I_1)=\mc L(\fk a,\wtd I_2)$, $\mc R(\fk a,\wtd I_1)=\mc R(\fk a,\wtd I_2)$ when acting on $\mc H_k^\infty$.
	
\item (Naturality) If  $G\in\Hom_{\mc A}(\mc H_k,\mc H_{k'})$,  the following diagrams commute\footnote{Note that they also commute adjointly and hence strongly since we have $(F\otimes G)^*=F^*\otimes G^*$ for any morphisms $F,G$ in a $C^*$-tensor category.}   for any $\fk a\in\fk H_i(\wtd I)$.
\begin{equation*}
\begin{tikzcd}[column sep=large]
\quad\mc H_k^\infty\quad \arrow[d,"{\mc L(\fk a,\wtd I)}"'] \arrow[r,"G"] & \quad\mc H_{k'}^\infty\quad \arrow[d,"{\mc L(\fk a,\wtd I)}"] \\
(\mc H_i\boxdot\mc H_k)^\infty \arrow[r,"\idt_i\boxdot G"]           & (\mc H_i\boxdot\mc H_{k'})^\infty          
\end{tikzcd}
\qquad
\begin{tikzcd}[column sep=large]
\mc H_k^\infty \arrow[r,"{\mc R(\fk a,\wtd I)}"] \arrow[d,"G"'] & (\mc H_k\boxdot\mc H_i)^\infty \arrow[d,"G\boxdot\idt_i"] \\
\mc H_{k'}^\infty \arrow[r,"{\mc R(\fk a,\wtd I)}"]           & (\mc H_{k'}\boxdot\mc H_i)^\infty        
\end{tikzcd}
\end{equation*}
	
\item (Neutrality) Under the identifications $\mc H_i=\mc H_i\boxdot\mc H_0=\mc H_0\boxdot\mc H_i$ through the unitors, for any $\fk a\in\fk H_i(\wtd I)$ we have
	\begin{align}
	\mc L(\fk a,\wtd I)|_{\mc H_0^\infty}=\mc R(\fk a,\wtd I)|_{\mc H_0^\infty}
	\end{align}
	
\item (Reeh-Schlieder property) Under the identification $\mc H_i=\mc H_i\boxdot\mc H_0$, the set $\mc L(\fk H_i(\wtd I),\wtd I)\Omega$ spans a dense subspace of $\mc H_i$.
	
\item (Density of fusion products) The set $\mc L(\fk H_i(\wtd I),\wtd I)\mc H_k^\infty$ spans a dense subspace of $\mc H_i\boxdot\mc H_k$, and $\mc R(\fk H_i(\wtd I),\wtd I)\mc H_k^\infty$ spans a dense subspace of $\mc H_k\boxdot\mc H_i$.
	
\item (Intertwining property) 	For any $\fk a\in\fk H_i(\wtd I)$ and $x\in\mc A^\infty(I')$, the following diagrams commute adjointly:
\begin{equation}\label{eq18}
\begin{tikzcd}[column sep=large]
\mc H_k^\infty \arrow[d,"{\mc L(\fk a,\wtd I)}"'] \arrow[r,"{\pi_{k,I'}(x)}"] & \mc H_k^\infty \arrow[d,"{\mc L(\fk a,\wtd I)}"] \\
(\mc H_i\boxdot\mc H_k)^\infty \arrow[r,"{\pi_{i\boxdot k,I'}(x)}"]           & (\mc H_i\boxdot\mc H_k)^\infty          
\end{tikzcd}
\qquad
\begin{tikzcd}[column sep=large]
\mc H_k^\infty \arrow[r,"{\mc R(\fk a,\wtd I)}"] \arrow[d,"{\pi_{k,I'}(x)}"'] & (\mc H_k\boxdot\mc H_i)^\infty \arrow[d,"{\pi_{k\boxdot i,I'}(x)}"] \\
\mc H_k^\infty \arrow[r,"{\mc R(\fk a,\wtd I)}"]           & (\mc H_k\boxdot\mc H_i)^\infty        
\end{tikzcd}
\end{equation}

\item (Weak locality) Assume that $\wtd J$ is clockwise to $\wtd I$. Then for  any $\fk a\in\fk H_i(\wtd I),\fk b\in\fk H_j(\wtd J)$, the following diagram   commutes adjointly.
\begin{equation}
\begin{tikzcd}[column sep=huge]
\quad \mc H_k^\infty\quad \arrow[r,"{\mc R(\fk b,\wtd J)}"] \arrow[d, "{\mc L(\fk a,\wtd I)}"'] &\quad (\mc H_k\boxdot\mc H_j)^\infty\quad \arrow[d, "{\mc L(\fk a,\wtd I)}"]\\
(\mc H_i\boxdot\mc H_k)^\infty\arrow[r,"{\mc R(\fk b,\wtd J)}"] &(\mc H_i\boxdot\mc H_k\boxdot\mc H_j)^\infty
\end{tikzcd}
\end{equation}
	
\item (Braiding) There is a unitary linear map $\ss_{i,k}:\mc H_i\boxdot\mc H_k\rightarrow\mc H_k\boxdot \mc H_i$ (the \textbf{braiding operator}) such that for any  $\fk a\in\fk H_i(\wtd I)$ and $\eta\in\mc H_k^\infty$,
	\begin{align}
	\ss_{i,k} \mc L(\fk a,\wtd I)\eta=\mc R(\fk a,\wtd I)\eta
	\end{align}
	
\item (Rotation covariance) 	For any $\fk a\in\fk H_i(\wtd I)$ and $g=\varrho(t)$ where $t\in\mbb R$, there exists an element $g\fk a$ inside $\fk H_i(g\wtd I)$, \footnote{(Recall Rem. \ref{lb50} for the meaning of $g\wtd I$.)} such that for any $\mc H_l\in\Obj(\scr C)$ and $\eta\in\mc H_l^\infty$, the following two equivalent equations are true.
\begin{subequations}
	\begin{gather}
	\mc L(g\fk a,g\wtd I)\eta=g\mc L(\fk a,\wtd I)g^{-1}\eta\\
	\mc R(g\fk a,g\wtd I)\eta=g\mc R(\fk a,\wtd I)g^{-1}\eta.
	\end{gather}
\end{subequations}
\end{enumerate}
If for each $\mc H_k\in\Obj(\scr C)$ and $g\in\UPSU$ we have $g\mc H_k^\infty\subset g\mc H_k^\infty$, and if the statements in (i) are true for any $g\in\UPSU$, we say that $\scr E^w$ is \textbf{M\"obius covariant}.
\end{df}

\begin{rem}\label{lb6}
In the above definition, for each $\fk a\in\fk H_i(\wtd I)$, the vector $\mc L(\fk a,\wtd I)\Omega$ (which equals $\mc R(\fk a,\wtd I)\Omega$ by neutrality) is an element of $\mc H_i^\pr(I)$. See \cite[Prop. 2.33]{Gui26}.
\end{rem}

We let $\ovl{\mc L(\fk a,\wtd I)}$ denote the operation associating to each $\mc H_k\in\Obj(\scr C)$ the closed operator $\ovl{\mc L(\fk a,\wtd I)}|_{\mc H_k}:\mc H_k\rightarrow\mc H_i\boxdot\mc H_k$ which is the closure of $\mc L(\fk a,\wtd I)|_{\mc H_k}$. $\ovl{\mc R(\fk a,\wtd I)}$ is understood in a similar way.

\begin{rem}
Since $\mc A^\infty(I)$ generates the von Neumann algebra $\mc A(I)$ (cf. Prop. \ref{lb4}), in the intertwining property in Def. \ref{lb8}, each of the two diagrams in \eqref{eq18} commutes strongly for all $x\in\mc A(I')$, not just for $x\in\mc A^\infty(I')$.
\end{rem}

\begin{thm}\label{lb16}
Let $\scr E^w=(\mc A,\scr C,\boxdot,\ss,\fk H)$ be a weak categorical extension. Then there is a unique (closed and vector labeled) categorical extension $\scr E=(\mc A,\scr C,\boxdot,\ss,\mc H)$ (called the \textbf{closure} of $\scr E^w$) satisfying that for every $\mc H_i,\mc H_k\in\Obj(\scr C),\wtd I\in\Jtd,\fk a\in\fk H_i(\wtd I)$, by setting $\xi=\mc L(\fk a,\wtd I)\Omega$ (which equals $\mc R(\fk a,\wtd I)\Omega$ and is in $\mc H_i^\pr(I)$, cf. Rem. \ref{lb6}), we have for each $\mc H_k\in\Obj(\scr C)$ that
\begin{align}
\ovl{\mc L(\fk a,\wtd I)}\big|_{\mc H_k}=\scr L(\xi,\wtd I)\big|_{\mc H_k}\qquad \ovl{\mc R(\fk a,\wtd I)}\big|_{\mc H_k}=\scr R(\xi,\wtd I)\big|_{\mc H_k}
\end{align}
\end{thm}

Thus, the closure of $\scr E^w$ is the unique categorical extension $\scr E$ such that the closure of every $\mc L$ resp. $\mc R$ operator is a left resp. right operator of $\scr E$ (Def. \ref{lb2}). The relationship between $\scr E^w$ and its closure is similar to that between a set of closed operators and the von Neumann algebra generated by them.

\begin{proof}
Uniqueness: Choose any $\wtd I$, and choose $\wtd J$ clockwise to $\wtd I$. Let $\mc H_i,\mc H_j\in\Obj(\scr C)$. Let $\eta\in\mc H_j(J)$. Choose any $\fk a\in\fk H_i(\wtd I)$, and let $\xi=\mc L(\fk a,\wtd I)\Omega$. By the locality in Thm. \ref{lb7} and the boundedness of $\scr R(\eta,\wtd J)=R(\eta,\wtd J)$, we have that $R(\eta,\wtd J)\Dom(\scr L(\xi,\wtd I)|_{\mc H_0})\subset\Dom(\scr L(\xi,\wtd I)|_{\mc H_j})$, and that
\begin{align*}
R(\eta,\wtd J)\mc L(\fk a,\wtd I)\Omega=R(\eta,\wtd J)\scr L(\xi,\wtd I)\Omega=\scr L(\xi,\wtd I)R(\eta,\wtd J)\Omega=\ovl{\mc L(\fk a,\wtd I)}\eta
\end{align*}
Therefore, $R(\eta,\wtd J)|_{\mc H_i}$ is uniquely determined by vectors of the form $\mc L(\fk a,\wtd I)\Omega$. By the Reeh-Schlieder property in Def. \ref{lb8}, such vectors span a dense subset of $\mc H_i$. So the R-operators are unique. Similarly, the L-operators are unique.

Existence: By \cite[Thm. 2.35]{Gui26}, there exists a natural unitary $\Psi$ satisfying all the descriptions in Thm. \ref{lb1}, except that \eqref{eq12} is replaced by
\begin{subequations}\label{eq15}
\begin{gather}
\Psi_{i,j}\ovl{\mc L(\fk a,\wtd I)}|_{\mc H_j}=\scr L^{\boxtimes}(\xi,\wtd I)|_{\mc H_j}\\
\Psi_{j,i}\ovl{\mc R(\fk a,\wtd I)}|_{\mc H_j}=\scr R^{\boxtimes}(\xi,\wtd I)|_{\mc H_j}
\end{gather}
for all $\fk a\in\fk H_i(\wtd I)$ and $\xi$ is set to be $\mc L(\fk a,\wtd I)\Omega$.
\end{subequations}
Then similar to the proof of Thm. \ref{lb7} (cf. also Rem. \ref{lb99}), the closure $\scr E$ can be defined by setting $L(\xi,\wtd I)|_{\mc H_j}=\Psi_{i,j}^{-1} L^\boxtimes(\xi,\wtd I)|_{\mc H_j}$ and $R(\xi,\wtd I)|_{\mc H_j}=\Psi_{j,i}^{-1} R^\boxtimes(\xi,\wtd I)|_{\mc H_j}$ for each $\mc H_i,\mc H_j\in\Obj(\scr C)$, $\wtd I\in\Jtd$, and $\xi\in\mc H_i(I)$. 
\end{proof}

\subsection{Weak left and right operators}

In the following, we fix a M\"obius covariant weak categorical extension $\scr E^w=(\mc A,\scr C,\boxdot,\ss,\fk H)$.

\begin{df}\label{lb9}
A \textbf{weak left operator} of $\scr E^w$ with charge space $\mc H_i\in\Obj(\scr C)$ is an operation $A(\fk x,\wtd I)$, where $\fk x$ is a formal element, $\wtd I\in\Jtd$, and for any $\mc H_k\in\Obj(\scr C)$, there is a smooth and localizable operator $A(\fk x,\wtd I):\mc H_k^\infty\rightarrow(\mc H_i\boxdot\mc H_k)^\infty$ such that the following conditions are satisfied:
\begin{enumerate}[label=(\alph*)]
\item If $\mc H_k,\mc H_{k'}\in\Obj(\scr C)$, $G\in\Hom_{\mc A}(\mc H_k,\mc H_{k'})$, then  the following diagram commutes.
\begin{equation}
\begin{tikzcd}[column sep=large]
\quad\mc H_k^\infty\quad \arrow[d,"{A(\xk,\wtd I)}"'] \arrow[r,"G"] & \quad\mc H_{k'}^\infty\quad \arrow[d,"{A(\xk,\wtd I)}"] \\
(\mc H_i\boxdot\mc H_k)^\infty \arrow[r,"\idt_i\boxdot G"]           & (\mc H_i\boxdot\mc H_{k'})^\infty          
\end{tikzcd}
\end{equation}
\item For any $\mc H_l,\mc H_k\in\Obj(\scr C)$,  $\wtd J\in\Jtd$  clockwise to $\wtd I$, and any $\fk b\in\fk H_l(\wtd J)$, the following diagram  commutes (non-necessarily adjointly).
\begin{equation}
\begin{tikzcd}[column sep=huge]
\quad \mc H_k^\infty\quad \arrow[r,"{\mc R(\fk b,\wtd J)}"] \arrow[d, "{A(\xk,\wtd I)}"'] &\quad (\mc H_k\boxdot\mc H_j)^\infty\quad \arrow[d, "{A(\xk,\wtd I)}"]\\
(\mc H_i\boxdot\mc H_k)^\infty\arrow[r,"{\mc R(\fk b,\wtd J)}"] &(\mc H_i\boxdot\mc H_k\boxdot\mc H_j)^\infty
\end{tikzcd}
\end{equation}
\end{enumerate}
\end{df}

\begin{df}\label{lb10}
A \textbf{weak right operator} of $\scr E^w$ with charge space $\mc H_j\in\Obj(\scr C)$ is an operation $B(\yk,\wtd J)$, where $\fk y$ is a formal element, $\wtd J\in\Jtd$, and for any $\mc H_k\in\Obj(\scr C)$, there is a smooth and localizable operator $B(\fk y,\wtd J):\mc H_k^\infty\rightarrow(\mc H_k\boxdot\mc H_j)^\infty$, such that the following conditions are satisfied:

\begin{enumerate}[label=(\alph*)]
\item If $\mc H_k,\mc H_{k'}\in\Obj(\scr C)$, $G\in\Hom_{\mc A}(\mc H_k,\mc H_{k'})$, then  the following diagram commutes.
\begin{equation}
\begin{tikzcd}[column sep=large]
\mc H_k^\infty \arrow[r,"{B(\yk,\wtd J)}"] \arrow[d,"G"'] & (\mc H_k\boxdot\mc H_j)^\infty \arrow[d,"G\boxdot\idt_j"] \\
\mc H_{k'}^\infty \arrow[r,"{B(\yk,\wtd J)}"]           & (\mc H_{k'}\boxdot\mc H_j)^\infty  
\end{tikzcd}  
\end{equation}
\item For any $\mc H_i,\mc H_k\in\Obj(\scr C)$,  $\wtd I\in\Jtd$  anticlockwise to $\wtd J$, and any $\fk a\in\fk H_i(\wtd I)$, the following diagram  commutes.
\begin{equation}
\begin{tikzcd}[column sep=huge]
\quad \mc H_k^\infty\quad \arrow[r,"{B(\yk,\wtd J)}"] \arrow[d, "{\mc L(\fk a,\wtd I)}"'] &\quad (\mc H_k\boxdot\mc H_j)^\infty\quad \arrow[d, "{\mc L(\fk a,\wtd I)}"]\\
(\mc H_i\boxdot\mc H_k)^\infty\arrow[r,"{B(\yk,\wtd J)}"] &(\mc H_i\boxdot\mc H_k\boxdot\mc H_j)^\infty
\end{tikzcd}
\end{equation}
\end{enumerate}
\end{df}

Recall Def. \ref{lb2} for the meaning of left and right operators of $\scr E$.

\begin{thm}\label{lb13}
Let $\scr E$ be the closure of $\scr E^w$. Assume that $\scr C$ is rigid, i.e., any object of $\scr C$ has a dual object in $\scr C$. If $A(\xk,\wtd I)$ is a weak left operator of $\scr E^w$ with charge space $\mc H_i$, then its closure is a left operator of $\scr E$, i.e., the operation $\ovl{ A(\fk x,\wtd I)}|_{\mc H_k}:\mc H_k\rightarrow\mc H_i\boxdot\mc H_k$ defines a left operator of $\scr E$. Similarly, the closure of a weak right operator of $\scr E^w$ is a right operator of $\scr E$.
\end{thm}

\begin{proof}
This is \cite[Thm. 2.38]{Gui26}, generalizing \cite[Thm. 8.1]{CKLW18} to the setting of categorical extensions.
\end{proof}

\section{$C^*$-Frobenius algebras and (weak) categorical extensions}\label{lb24}

Recall $\scr C$ described in Sec. \ref{lb5}. We fix a (closed and vector-labeled) categorical extension $\scr E=(\mc A,\scr C,\boxdot,\ss,\mc H)$ with braiding operation $\ss$ and with L and R operators $L^\boxdot,R^\boxdot$, abbreviated to $L,R$ when no confusion arises. (Recall from Rem. \ref{lb20} that $\scr E$ is automatically conformal covariant.) Then $(\scr C,\boxdot,\ss)$ is braided $C^*$-tensor category, and is canonically isomorphic to a braided $C^*$-tensor subcategory of $(\Rep(\mc A),\boxtimes,\mbb B)$, the Connes braided $C^*$-category for $\mc A$. (Recall Thm. \ref{lb1}.)

\subsection{Haploid commutative $C^*$-Frobenius algebras $Q$ and their representation categories}\label{lb106}

Recall the following definition: 

\begin{df}
A \textbf{$C^*$-Frobenius algebra} in $\scr C$ is a triple $Q=(\mc H_a,\mu,\iota)$ where $\mc H_a\in\Obj(\scr C)$, $\mu\in\Hom_{\mc A}(\mc H_a\boxdot\mc H_a,\mc H_a)$, $\iota\in\Hom_{\mc A}(\mc H_0,\mc H_a)$ satisfy the following conditions:
	\begin{itemize}
		\item (Unit) $\mu(\iota\boxdot\idt_a)=\idt_a=\mu(\idt_a\boxdot\iota)$.
		\item (Associativity+Frobenius relation) The following diagram commutes adjointly.
		\begin{equation}
			\begin{tikzcd}
				\mc H_a\boxdot\mc H_a\boxdot H_a\arrow[rr,"\idt_a\boxdot\mu"]\arrow[d,"\mu\boxdot\idt_a"'] && \mc H_a\boxdot\mc H_a\arrow[d,"\mu"]\\
				\quad\mc H_a\boxdot\mc H_a\quad\arrow[rr,"\mu"] && \quad\mc H_a\quad
			\end{tikzcd}	
		\end{equation} 
\end{itemize}
$Q$ is called
\begin{itemize}
\item \textbf{normalized} if $\iota^*\iota=\idt_{\mc H_0}$;
\item \textbf{commutative} if $\mu\circ\ss_{a,a}=\mu$;
\item \textbf{haploid} (or \textbf{irreducible}) if $\dim\Hom_{\mc A}(\mc H_0,\mc H_a)=1$. 
\end{itemize}
Unless otherwise stated, \uwave{all $C^*$-Frobenius algebras considered in this article are assumed to be normalized}.
\end{df}

In the following, we fix a haploid commutative $C^*$-Frobenius algebra $Q=(\mc H_a,\mu,\iota)$ in $\scr C$. 

\begin{df}\label{lb66}
A \textbf{dyslectic (unitary) $Q$-module} (in $\scr C$) denotes a pair $(\mc H_i,\mu^i)$ where $\mc H_i\in\Obj(\scr C)$, $\mu^i\in\Hom_{\mc A}(\mc H_a\boxdot\mc H_i,\mc H_i)$, and the following are satisfied:	
	\begin{itemize}
		\item (Unit) $\mu^i(\iota\boxdot\idt_i)=\idt_i$.
		\item (Associativity+Frobenius relation) The following diagram commutes adjointly.
		\begin{equation}
			\begin{tikzcd}
				\mc H_a\boxdot\mc H_a\boxdot H_i\arrow[rr,"\idt_a\boxdot\mu^i"]\arrow[d,"\mu\boxdot\idt_i"'] && \mc H_a\boxdot\mc H_i\arrow[d,"\mu^i"]\\
				\quad\mc H_a\boxdot\mc H_i\quad\arrow[rr,"\mu^i"] && \quad\mc H_i\quad
			\end{tikzcd}	
		\end{equation} 
\item (Dyslectic condition) $\mu^i\circ\ss_{i,a}=\mu^i\circ\ss_{a,i}^{-1}$.
	\end{itemize}	
If $\mc H_i$ and $\mc H_j$ are dyslectic $Q$-modules, a \textbf{(homo)morphism of dyslectic $Q$-modules} $\alpha:\mc H_i\rightarrow\mc H_j$ means that $\alpha\in\Hom_{\mc A}(\mc H_i,\mc H_j)$, and that
\begin{align*}
\mu^j(\idt_a\boxdot\alpha)=\alpha\mu^i
\end{align*}
We let
\begin{align*}
\Hom_Q(\mc H_i,\mc H_j)=\{\text{morphisms of dyslectic $Q$-modules }\mc H_i\rightarrow\mc H_j\}
\end{align*}
The class of all dyslectic $Q$-modules, together with the morphisms defined above, is a $C^*$-category (cf.  \cite[Sec. 6.1]{NY16} or \cite[Prop. 2.24]{Gui22}). We let 
\begin{align*}
\pmb{\Rep^0_{\scr C}(Q)}=\text{the $C^*$-category of dyslectic $Q$-modules in $\scr C$}
\end{align*}
\end{df}

\begin{df}
Let $\mc H_i,\mc H_j\in\Obj(\scr C)$. A \textbf{(unitary) fusion product of $\mc H_i,\mc H_j$ over $Q$} denotes a pair $(\mc H_i\boxdot_Q\mc H_j,\mu_{i,j})$ where $\mc H_i\boxdot_Q\mc H_j$ (abbreviated to $\mc H_{i\boxdot j}$ or simply $\mc H_{ij}$) together with $\mu^{i\boxdot j}\equiv\mu^{ij}\in(\mc H_a\boxdot\mc H_{ij},\mc H_{ij})$ is a dyslectic $Q$-module in $\scr C$. Moreover, we have
\begin{align}
\mu_{i,j}\in\Hom_{\mc A}(\mc H_i\boxdot\mc H_j,\mc H_i\boxdot_Q\mc H_j)
\end{align}
and the following conditions are satisfied:
\begin{itemize}
\item (Intertwining property) The actions of $Q$ on $\mc H_i,\mc H_j,\mc H_{ij}$ are invariant, i.e.,
\begin{align*}
\mu^{ij}(\idt_a\boxdot \mu_{i,j})=\mu_{i,j}(\mu^i\boxdot\idt_j)=\mu_{i,j}(\idt_i\boxdot\mu^j)(\ss_{a,i}\boxdot\idt_j)
\end{align*}
\item (Universal property) If $(\mc H_k,\mu^k)\in\Obj(\Rep_{\scr C}^0(Q))$, and if $\alpha\in\Hom_{\mc A}(\mc H_i\boxdot\mc H_j,\mc H_k)$ is a \textbf{type $k\choose i~j$ intertwining operator of $Q$} in the sense that
\begin{align*}
\mu^k(\idt_a\boxdot \alpha)=\alpha(\mu^i\boxdot\idt_j)=\alpha(\idt_i\boxdot\mu^j)(\ss_{a,i}\boxdot\idt_j)
\end{align*}
then there exists a unique $\wtd\alpha\in\Hom_Q(\mc H_i\boxdot_Q\mc H_j,\mc H_k)$ such that $\alpha=\wtd\alpha\mu_{i,j}$.
\item (Unitarity) The following diagram commutes adjointly:
\begin{equation}
\begin{tikzcd}
\mc H_i\boxdot\mc H_a\boxdot H_j\arrow[rr,"\idt_i\boxdot\mu^j"]\arrow[d,"(\mu^i\circ\ss_{i,a})\boxdot\idt_j"'] && \mc H_i\boxdot\mc H_j\arrow[d,"\mu_{i,j}"]\\
\mc H_i\boxdot\mc H_j\arrow[rr,"\mu_{i,j}"] && \mc H_i\boxdot_Q\mc H_j
\end{tikzcd}	
\end{equation} 
(Note that the commutativity follows from the intertwining property of $\mu^{ij}$.)
\end{itemize}
\end{df}

\begin{rem}\label{lb17}
Fusion products over $Q$ always exist and are unique up unitaries: if $(\mc H_i\boxdot_Q\mc H_k,\mu_{i,j})$ and $(\mc H_i\wht\boxdot_Q\mc H_k,\wht\mu_{i,j})$ are both fusion products of $\mc H_i,\mc H_j$ over $Q$ in $\scr C$, there there is a (necessarily unique) unitary $\Phi_{i,j}\in\Hom_Q(\mc H_i\boxdot_Q\mc H_k,\mc H_i\wht\boxdot_Q\mc H_k)$ (called the \textbf{linking map} between the two systems) such that
\begin{align}\label{eq64}
\wht\mu_{i,j}=\Phi_{i,j}\circ\mu_{i,j}
\end{align}
See \cite[Sec. 3.1]{Gui22} for details.
\end{rem}

\begin{thm}\label{lb19}
Suppose that for each $\mc H_i,\mc H_j\in\Obj(\Rep_{\scr C}^0(Q))$ a fusion product $\mc H_i\boxdot_Q\mc H_j=\mc H_{i,j}$ (with $\mu_{i,j}$) is assigned. (We call such an assignment $(\boxdot_Q,\mu_{\blt,\star})$ (or simply call $\boxdot$) a \textbf{system of fusion products in $\Rep^0_{\scr C}(Q)$}.) Then 
\begin{align*}
(\Rep_{\scr C}^0(Q),\boxdot_Q,\ss^Q)
\end{align*}
defined in the following way is a braided $C^*$-tensor category (called the \textbf{braided $C^*$-tensor category associated to $(\boxdot_Q,\mu_{\blt,\star})$}):
\begin{subequations}\label{eq50}
\begin{itemize}
\item If $F\in\Hom_Q(\mc H_i,\mc H_{\wtd i})$ and $G\in\Hom_Q(\mc H_j,\mc H_{\wtd j})$, then $F\boxdot_Q G$ is the unique element in $\Hom_Q(\mc H_i\boxdot_Q\mc H_j,\mc H_{\wtd i}\boxdot_Q\mc H_{\wtd j})$ satisfying
\begin{align}
\mu_{\wtd i,\wtd j}(F\boxdot G)=(F\boxdot_QG)\mu_{i,j} \label{eq19}
\end{align}
\item The (unitary) associator $\fk A_{i,j,k}\in\Hom_Q\big((\mc H_i\boxdot_Q\mc H_j)\boxdot_Q\mc H_k,\mc H_i\boxdot_Q(\mc H_j\boxdot_Q\mc H_k)\big)$ is determined by
\begin{align}
\mu_{i,jk}(\idt_i\boxdot\mu_{j,k})=\fk A_{i,j,k}\circ \mu_{ij,k}(\mu_{i,j}\boxdot\idt_k)
\end{align}
\item The (unitary) unitors $\fk l_{a,i}\in\Hom_Q(\mc H_a\boxdot_Q\mc H_i,\mc H_i)$ and $\fk r:_{i,a}:\Hom_Q(\mc H_i\boxdot_Q\mc H_a,\mc H_i)$ are determined by 
\begin{align}
\mu^i=\fk l_i \mu_{a,i}\qquad \mu^i\ss_{i,a}=\fk r_i\mu_{i,a}\label{eq21}
\end{align}
\item The (unitary) braiding operator $\ss^Q_{i,j}\in\Hom_Q(\mc H_i\boxdot_Q\mc H_j,\mc H_j\boxdot_Q\mc H_i)$ is determined by
\begin{align}
\mu_{j,i}\ss_{i,j}=\ss^Q_{i,j}\mu_{i,j} \label{eq20}
\end{align}
\end{itemize}
\end{subequations}
\end{thm}

\begin{proof}
This is well-known. See e.g. Thm. 3.13 and 3.23 of \cite{Gui22}.
\end{proof}

\begin{rem}\label{lb28}
Given a system of fusion products of dyslectic $Q$-modules, we identify $(\mc H_i\boxdot_Q\mc H_j)\boxdot_Q\mc H_k$ with $\mc H_i\boxdot_Q(\mc H_j\boxdot_Q\mc H_k)$ using the associator, and identify $\mc H_i,\mc H_a\boxdot_Q\mc H_i,\mc H_i\boxdot_Q\mc H_i$ using the unitors. Then we have
\begin{align}
\mu^i=\mu_{a,i}\qquad \mu^i\ss_{i,a}=\mu_{i,a}.
\end{align}
Moreover, given the above identifications, we have:
\end{rem}

\begin{thm}\label{lb30}
For any $\mc H_i,\mc H_j,\mc H_k\in\Obj(\Rep_{\scr C}^0(Q))$, the following diagram commutes adjointly:
\begin{equation}\label{eq58}
\begin{tikzcd}[column sep=huge]
\mc H_i\boxdot\mc H_k\boxdot\mc H_j \arrow[r,"{\idt_i\boxdot\mu_{k,j}}"] \arrow[d, "{\mu_{i,k}\boxdot\idt_j}"'] &\mc H_i\boxdot(\mc H_k\boxdot_Q\mc H_j)\arrow[d, "{\mu_{i,kj}}"]\\
(\mc H_i\boxdot_Q\mc H_k)\boxdot\mc H_j\arrow[r,"{\mu_{ik,j}}"] &\mc H_i\boxdot_Q\mc H_k\boxdot_Q\mc H_j
	\end{tikzcd}
\end{equation}
\end{thm}

\begin{proof}
This follows from (3.10) and Thm. 3.14 of \cite{Gui22}.
\end{proof}

\begin{rem}\label{lb39}
Suppose that we choose two systems of fusion products associating to each $\mc H_i,\mc H_j\in\Obj(\Rep_{\scr C}^0(Q))$ the fusion products $(\mc H_i\boxdot_Q\mc H_j,\mu_{i,j})$ and $(\mc H_i\wht\boxdot_Q\mc H_j,\wht\mu_{i,j})$ respectively. Then the unitary linking map $\Phi$ in Rem. \ref{lb17} is natural, and $(\id_{\Rep_{\scr C}^0(Q)},\Phi)$ is a braided $*$-functor implementing an isomorphism of the braided $C^*$-tensor categories 
\begin{align}
(\id_{\Rep_{\scr C}^0(Q)},\Phi):(\Rep^0_{\scr C}(Q),\boxdot_Q,\ss^Q)\xlongrightarrow{\simeq}(\Rep^0_{\scr C}(Q),\wht\boxdot_Q,\wht\ss^Q)
\end{align}
See Thm. 3.16 and 3.24 of \cite{Gui22} for the proof and the details.
\end{rem}

\subsection{The functor $\fk F_\CN:\Rep^0_{\scr C}(Q)\rightarrow\Rep_{\scr C}(\mc B_Q)$}

\begin{df}
A \textbf{(normalized and irreducible) conformal net extension} of $\mc A$ is a triple $(\mc B,\mc H_a,\iota)$ (or simply $\mc B$), where $(\mc B,\mc H_a)$ is an (irreducible local) conformal net, the Hilbert space $\mc H_a$ is associated with an $\mc A$-module structure $(\mc H_a,\pi_a)$,  $\iota\in\Hom_{\mc A}(\mc H_0,\mc H_a)$ satisfies $\iota^*\iota=\idt_{\mc H_0}$,\index{This is the condition of being normalized.} and the following conditions are satisfied:
\begin{enumerate}[label=(\arabic*)]
\item For each $I\in\mc J$ we have $\pi_{a,I}(\mc A(I))\subset\mc B(I)$.
\item If $\Omega_{\mc A}$ and $\Omega_{\mc B}$ denote the vacuum vectors of $\mc A,\mc B$ respectively, then $\Omega_{\mc B}=\iota\Omega_{\mc A}$.
\item The projective representation of $\Diffp(\Sbb^1)$ on $\mc H_a$ (defined by the conformal net $\mc B$) commutes with $\iota\iota^*$ and is pulled back by $\iota$ to the projective representation of $\Diffp(\Sbb^1)$ on $\mc H_0$.
\end{enumerate}
We say that $\mc B$ is a conformal net extension of $\mc A$ \textbf{in} \pmb{$\scr C$} if $\mc H_a\in\Obj(\scr C)$. We say that the extension has \textbf{finite index} if $[\mc B(I):\pi_{a,I}(\mc A(I))]$ is finite for some (and hence for all) $I\in\mc J$.
\end{df}

\begin{rem}
Let $(\mc B,\mc H_a,\iota)$ be a conformal net extension of $\mc A$. We often identify $\mc H_0$ with its image in $\mc H_a$ under $\iota$, identify $\mc A(I)$ with $\pi_{a,I}(\mc A(I))$, and identify the vacuum vectors of $\mc A$ and $\mc B$ via $\iota$ and denote both by $\Omega$. Then $\mc A$ becomes a conformal subnet of $\mc B$. In this case, we simply write the extension as $(\mc B, \mc H_a)$.
\end{rem}

\begin{rem}
Recall from Subsec. \ref{lb18} that $\mc A$ has a Virasoro subnet $\Vir_c$ determined by the central charge $c\in\Rbb_{\geq0}$, and that $\GA$ is identified with $\Gc=\scr G_{\Vir_c}$ in a canonical way. Now suppose that $(\mc B,\mc H_a)$ is a conformal net extension of $\mc A$. Then $\scr G_{\mc B}$ consists of $(g,V)$ where $g\in\scr G$ and $V$ is a unitary operator on $\mc H_a$ representing the projective action of $g$ on $\mc H_a$. By the definition of conformal net extension, $V|_{\mc H_0}$ represents the projective action of $g$ on $\mc H_0$. Thus, we have an isomorphism $\scr G_{\mc B}\xrightarrow{\simeq}\GA$ defined by $(g,V)\mapsto(g,V|_{\mc H_0})$. Therefore, we can make the identifications
\begin{align*}
\Gc=\GA=\scr G_{\mc B}
\end{align*}
\end{rem}

\begin{df}\label{lb100}
Let $\mc B$ be a conformal net extension of $\mc A$ in $\scr C$. If $(\mc H_i,\pi_i)$ is a $\mc B$-module whose restriction to $\mc A$ gives an object in $\scr C$, we say that $\mc H_i$ is a \textbf{$\mc B$-module in $\scr C$}. We let
\begin{align*}
\pmb{\Rep_{\scr C}(\mc B)}=\text{the $C^*$-category of $\mc B$-modules in $\scr C$}
\end{align*}
Note that if $\scr C$ is full and replete, then $\Rep_{\scr C}(\mc B)$ is clearly also full and replete.
\end{df}

\begin{thm}\label{lb26}
There is a one-to-one correspondence between a haploid commutative $C^*$-Frobenius algebra $Q=(\mc H_a,\mu,\iota)$ in $\scr C$ and a finite-index conformal net extension $(\mc B_Q,\mc H_a,\iota)$ of $\mc A$ in $\scr C$. $\mc B_Q$ is determined by the fact that for every $\wtd I\in\Jtd$,
\begin{align}
\mc B_Q(I)=\big\{\mu L^\boxdot(\xi,\wtd I)|_{\mc H_a}:\xi\in\mc H_a(I)  \big\}
\end{align}
The unitary representation of $\Gc$ on the $\mc A$-module $\mc H_a$ gives the projectively unitary representation of $\Diffp(\Sbb^1)$ associated to $\mc B_Q$.
\end{thm}

We call $\mc B_Q$ the \textbf{conformal net extension associated to} $Q$.

\begin{proof}
See Thm. 2.6, 2.12, 5.7 of \cite{Gui21c}.
\end{proof}

\begin{thm}\label{lb27}
There is a $*$-functor $\fk F_\CN$ inducing an isomorphism of $C^*$-categories
\begin{subequations}
\begin{gather}
\begin{gathered}
\fk F_\CN:\Rep^0_{\scr C}(Q)\xlongrightarrow{\simeq}\Rep_{\scr C}(\mc B_Q)\\
(\mc H_i,\mu^i)\in\Obj(\Rep_{\scr C}^0(Q))\quad\mapsto\quad (\mc H_i,\pi_i)\in\Obj(\Rep_{\scr C}(\mc B_Q))\\
F\in\Hom_Q(\mc H_i,\mc H_j)\quad\mapsto\quad F\in\Hom_{\mc B_Q}(\mc H_i,\mc H_j)
\end{gathered}
\end{gather}
where $\pi_i$ is determined the fact that for each $\wtd I\in\Jtd$ and $\xi\in\mc H_a(I)$,
\begin{align}
\pi_{i,I}\big(\mu L^\boxdot(\xi,\wtd I\big)|_{\mc H_a})=\mu^i L^\boxdot(\xi,\wtd I)|_{\mc H_i}\equiv \mu_{a,i} L^\boxdot(\xi,\wtd I)|_{\mc H_i}
\end{align}
\end{subequations} 
\end{thm}
Thus, $F$ and $\fk F_\CN(F)$ are the same map of Hilbert spaces $\mc H_i\rightarrow\mc H_j$.

\begin{proof}
See Main Theorem A in \cite[Sec. 2]{Gui21c}.
\end{proof}


\begin{df}\label{lb70}
For each haploid commutative $C^*$-Frobenius algebra $Q$ in $\scr C$ and a system of fusion products $(\boxdot_Q,\mu_{\blt,\star})$ in $\Rep^0_{\scr C}(Q)$ we define \textbf{a braided \pmb{$C^*$}-tensor structure on} \pmb{$\Rep_{\scr C}(\mc B_Q)$} to be the \textbf{image of \pmb{$\Rep^0_{\scr C}(Q)$} under \pmb{$(\fk F_\CN,\id)$}}. 

More precisely: Recall by Thm. \ref{lb19} that $(\Rep^0_{\scr C}(Q),\boxdot_Q,\ss^Q)$ is the braided $C^*$-tensor category associated to $(\boxdot_Q,\mu_{\blt,\star})$. The pushforward of this braided $C^*$-tensor structure to $\Rep_{\scr C}(\mc B_Q)$ gives the desired $C^*$-tensor category
\begin{align*}
\big(\Rep_{\scr C}(\mc B_Q),\boxdot_Q,\ss^Q\big)
\end{align*}
Thus, the braided $*$-functor $(\fk F_\CN,\id)$ (where $\fk F_\CN$ is described in Thm. \ref{lb27}) gives an isomorphism of braided $C^*$-tensor category
\begin{align*}
(\fk F_\CN,\id): \big(\Rep^0_{\scr C}(Q),\boxdot_Q,\ss^Q\big)\xlongrightarrow{\simeq}\big(\Rep_{\scr C}(\mc B_Q),\boxdot_Q,\ss^Q\big)
\end{align*}
where the natural map $\id$ means that for each objects $X=(\mc H_i,\mu^i)$ and $Y=(\mc H_j,\mu^j)$ of $\Rep_{\scr C}^0(Q)$ we have
\begin{align*}
\fk F_\CN(X)\boxdot_Q \fk F_\CN(Y)=\fk F_\CN(X\boxdot_Q Y)
\end{align*}
\end{df}

\begin{rem}
Briefly speaking, if we identify the $C^*$-tensor category $\Rep_{\scr C}^0(Q)$ with $\Rep_{\scr C}(\mc B_Q)$, then the braided $C^*$-tensor structure on $\Rep_{\scr C}(\mc B_Q)$ is the same as the one $\Rep_{\scr C}^0(Q)$ defined by the system of fusion products $\boxdot_Q$. By Thm. \ref{lb19}, a different system of fusion products defines a different but isomorphic $C^*$-tensor structure on $\Rep_{\scr C}(\mc B_Q)$.
\end{rem}

\subsection{Categorical extension $\scr E_Q$ associated to $Q$; the tensorator $\fk N^\boxdot:\mc H_i\boxdot_Q\mc H_j\rightarrow\mc H_i\boxtimes_{\mc B_Q}\mc H_j$}\label{lb40}

In this subsection and the next one, we fix a haploid commutative $C^*$-Frobenius algebra $Q=(\mc H_a,\mu,\iota)$ in $\scr C$. 

\begin{lm}
For each $\mc H_i\in\Obj(\Rep_{\scr C}(\mc B_Q))$ and $I\in\mc J$, we have
\begin{align*}
\Hom_{\mc A(I')}(\mc H_0,\mc H_i)\Omega=\Hom_{\mc B_Q(I')}(\mc H_a,\mc H_i)\Omega
\end{align*}
So $\mc H_i(I)$ can denote both sides unambiguously (recall \eqref{eq13}). 
\end{lm}

\begin{proof}
\cite[Lem. 5.8]{Gui21c}.
\end{proof}

\begin{thm}\label{lb21}
Let $(\boxdot_Q,\mu_{\blt,\star})$ be a system of fusion products in $\Rep^0_{\scr C}(Q)$ which determines a braided $C^*$-tensor category $(\Rep^0_{\scr C}(Q),\boxdot_Q,\ss^Q)$ as in Thm. \ref{lb19}. Let $(\Rep_{\scr C}(\mc B_Q),\boxdot_Q,\ss^Q)$ be its image under $(\fk F_\CN,\id)$, cf. Def. \ref{lb70}. Then there is a (closed vector-labeled) conformal covariant categorical extension
\begin{align*}
\scr E_Q=(\mc B_Q,\Rep_{\scr C}(\mc B_Q),\boxdot_Q,\ss^Q,\mc H)
\end{align*}
determined by the following fact: If we let $L^Q,R^Q$ denote the (bounded) L and R operators of $\scr E_Q$, then for each $\mc H_i,\mc H_k\in\Obj(\Rep_{\scr C}(\mc B_Q))$, $\wtd I\in\Jtd$, and $\xi\in\mc H_i(I)$, we have
\begin{gather}\label{eq63}
L^Q(\xi,\wtd I)|_{\mc H_k}=\mu_{i,k}\circ L^\boxdot(\xi,\wtd I)|_{\mc H_k}\qquad R^Q(\xi,\wtd I)|_{\mc H_k}=\mu_{k,i}\circ R^\boxdot(\xi,\wtd I)|_{\mc H_k}
\end{gather}
which are bounded linear maps $\mc H_k\rightarrow\mc H_i\boxdot_Q\mc H_k$ and $\mc H_k\rightarrow\mc H_k\boxdot_Q\mc H_i$ respectively.
\end{thm}

One can write \eqref{eq63} as
\begin{align}
\scr E_Q=\mu_{\blt,\star}\circ\scr E
\end{align}

\begin{proof}
See the proof of Main Theorem C in \cite[Sec. 5]{Gui21c}. Note that the conformal covariance of $\scr E_Q$ is automatic by Rem. \ref{lb20}.
\end{proof}

\begin{df}\label{lb69}
The $\scr E_Q$ in Thm. \ref{lb21} is called the \textbf{categorical extension associated to $Q$, $\scr E$, and $(\boxdot_Q,\mu_{\blt,\star})$}. Its vector-labeled left and right operators are written as $\scr L^Q(\xi,\wtd I)$ and $\scr R^Q(\xi,\wtd I)$ where $\xi\in\mc H_i^\pr(I)$ (cf. Rem. \ref{lb23}); when $\xi\in\mc H_i(I)$, we write them as $L^Q(\xi,\wtd I),R^Q(\xi,\wtd I)$.  
\end{df}

Recall that $L=L^\boxdot$ and $R=R^\boxdot$ are the L and R operators of the categorical extension $\scr E=(\mc A,\scr C,\boxdot,\ss,\mc H)$ fixed at the beginning of this chapter. In the next corollary, we let $L^{\boxtimes_{\mc B_Q}}$ and $R^{\boxtimes_{\mc B_Q}}$ be the L and R operators of the Connes categorical extension $\scr E_{\mc B_Q,\Connes}=(\mc B_Q,\Rep(\mc B_Q),\boxtimes_{\mc B_Q},\mbb B^{\mc B_Q},\mc H)$ for $\mc B_Q$ (cf. Thm. \ref{lb12}).

\begin{co}\label{lb41}
Let $(\boxdot_Q,\mu_{\blt,\star})$ be a system of fusion products in $\Rep^0_{\scr C}(Q)$ giving a braided $C^*$-tensor category $(\Rep_{\scr C}(\mc B_Q),\boxdot_Q,\ss^Q)$ (cf. Thm. \ref{lb19} and Def. \ref{lb70}). Let $\fk F_\CN:\Rep^0_{\scr C}(Q)\rightarrow\Rep_{\scr C}(\mc B_Q)$ be the $*$-functor in Thm. \ref{lb27}. Let 
\begin{align*}
\fk N^\boxdot:\mc H_\blt\boxdot_Q\mc H_\star\rightarrow\mc H_\blt\boxtimes_{\mc B_Q}\mc H_\star
\end{align*}
be the tensorator associated to $\scr E_Q=(\mc B_Q,\Rep_{\scr C}(\mc B_Q),\boxdot_Q,\ss^Q,\mc H)$ (cf. Thm. \ref{lb1}), i.e., for each $\mc H_i,\mc H_j\in\Obj(\Rep_{\scr C}(\mc B_Q))$, $\fk N^\boxdot_{i,j}\in\Hom_{\mc B_Q}(\mc H_i\boxdot_Q\mc H_j,\mc H_i\boxtimes_{\mc B_Q}\mc H_j)$ and $\fk N^\boxdot_{j,i}\in\Hom_{\mc B_Q}(\mc H_j\boxdot_Q\mc H_i,\mc H_j\boxtimes_{\mc B_Q}\mc H_i)$ are the unique unitary maps satisfying
\begin{subequations}\label{eq32}
\begin{gather}
\fk N^\boxdot_{i,j}\circ \mu_{i,j}L^{\boxdot}(\xi,\wtd I)|_{\mc H_j}=L^{\boxtimes_{\mc B_Q}}(\xi,\wtd I)|_{\mc H_j}\\
\fk N^\boxdot_{j,i}\circ \mu_{j,i}R^{\boxdot}(\xi,\wtd I)|_{\mc H_j}=R^{\boxtimes_{\mc B_Q}}(\xi,\wtd I)|_{\mc H_j}
\end{gather}
\end{subequations}
Let $\wht{\scr C}$ be the repletion of $\scr C$ (Rem. \ref{lb38}). Then $(\fk F_\CN,\fk N^\boxdot)$ is a braided $*$-functor, and
\begin{align}\label{eq104}
(\fk F_\CN,\fk N^\boxdot):(\Rep_{\scr C}^0(Q),\boxdot_Q,\ss^Q)\xlongrightarrow{\simeq} (\Rep_{\wht{\scr C}}(\mc B_Q),\boxtimes_{\mc B_Q},\mbb B^{\mc B_Q})
\end{align}
is an equivalence of braided $C^*$-tensor categories, which is also an isomorphism of braided $C^*$-tensor categories when $\scr C=\wht{\scr C}$.
\end{co}

\begin{proof}
This is immediate from Thm. \ref{lb27}, \ref{lb21}, and \ref{lb1}.
\end{proof}

\begin{rem}\label{lb68}
Let $\scr E_{\mc B_Q,\Connes}$ be the Connes categorical extension for $\mc B_Q$. Then one can write \eqref{eq32} as
\begin{align}
\fk N^\boxdot\circ\mu_{\blt,\star}\circ\scr E=\fk N^\boxdot\circ\scr E_Q=\scr E_{\mc B_Q,\Connes}\big|_{\Rep_{\scr C}(\mc B_Q)}
\end{align}
Suppose that $(\wht\boxdot_Q,\wht\mu_{\blt,\star})$ is another system of fusion products in $\Rep^0_{\scr C}(Q)$ which together with $Q,\scr E$ give a categorical extension $\wht{\scr E}_Q=(\mc B_Q,\Rep_{\scr C}(\mc B_Q),\wht\boxdot_Q,\wht\ss^Q,\mc H)$. Then, similarly, we have
\begin{align*}
\wht{\fk N}^\boxdot\circ\wht\mu_{\blt,\star}\circ\scr E=\wht{\fk N}^\boxdot\circ\wht{\scr E}_Q=\scr E_{\mc B_Q,\Connes}\big|_{\Rep_{\scr C}(\mc B_Q)}
\end{align*}
and hence $\fk N^\boxdot\circ\mu_{\blt,\star}=\wht{\fk N}^\boxdot\circ\wht\mu_{\blt,\star}$. Thus, since $\wht\mu_{\blt,\star}=\Phi\circ\mu_{\blt,\star}$ where $\Phi$ is the unitary linking map from $(\boxdot_Q,\mu_{\blt,\star})$ to $(\wht\boxdot_Q,\wht\mu_{\blt,\star})$ (Rem. \ref{lb17}), we have
\begin{subequations}
\begin{gather}
\wht{\fk N}^\boxdot\circ\Phi=\fk N^\boxdot \label{eq55}\\
\wht{\scr E}^Q=\Phi\circ\scr E^Q \label{eq33}
\end{gather}
\end{subequations}

In the case that $\scr C$ is replete (and hence \eqref{eq104} is an isomorphism), if we identify the $C^*$-categories $\Rep^0_{\scr C}(Q)$ and $\Rep_{\scr C}(\mc B_Q)$ via $\fk F_\CN$, we can choose $\wht\boxdot_Q=\boxtimes_{\mc B_Q}$ and $\wht\mu_{\blt,\star}=\fk N^\boxdot\circ\mu_{\blt,\star}$. Then $\Phi$ becomes $\fk N^\boxdot$, and hence $\wht{\fk N}^\boxdot=\id$. We conclude that there exists a system $(\wht\boxdot_Q,\wht\mu_{\blt,\star})$ of fusion products in $\Rep^0_{\scr C}(Q)$ such that $\wht{\fk N}^\boxdot=\id$ and $\wht\boxdot_Q=\boxtimes_{\mc B_Q}$. 
\end{rem}

\subsection{Weak categorical extensions associated to $Q$}

Fix a haploid commutative $C^*$-Frobenius algebra $Q=(\mc H_a,\mu,\iota)$ in $\scr C$. Recall Thm. \ref{lb16} for the meaning of the closure of a weak categorical extension. Recall that the categorical extension $\scr E=(\mc A,\scr C,\boxdot,\ss,\mc H)$ is fixed at the beginning of Sec. \ref{lb24}. The following theorem is the key result that ensures the comparison theorems for VOA and conformal net extensions.

\begin{thm}\label{lb71}
Let $(\boxdot_Q,\mu_{\blt,\star})$ be a system of fusion products in $\Rep^0_{\scr C}(Q)$ giving a braided $C^*$-tensor category $(\Rep_{\scr C}(\mc B_Q),\boxdot_Q,\ss^Q)$ (as described in Thm. \ref{lb19} and Def. \ref{lb70}). Let $\scr E^w=(\mc A,\scr C,\boxdot,\ss,\fk H)$ be a weak categorical extension whose closure equals $\scr E=(\mc A,\scr C,\boxdot,\ss,\mc H)$. Then there is a (necessarily unique) weak categorical extension
\begin{align}
\scr E^w_Q=(\mc B_Q,\Rep_{\scr C}(\mc B_Q),\boxdot_Q,\ss^Q,\fk H)
\end{align}
associating to each $\mc H_i,\mc H_j\in\Obj(\Rep_{\scr C}(\mc B_Q)),\wtd I\in\Jtd,\fk a\in\fk H_i(I)$ smooth and localizable operators
\begin{gather*}
\mc L^Q(\fk a,\wtd I):\mc H_k\rightarrow\mc H_i\boxdot_Q\mc H_k\\
\mc R^Q(\fk a,\wtd I):\mc H_k\rightarrow\mc H_k\boxdot_Q\mc H_i
\end{gather*}
(with domains $\mc H_k^\infty$) defined by
\begin{align}\label{eq16}
\mc L^Q(\fk a,\wtd I)|_{\mc H_k^\infty}=\mu_{i,k}\circ \mc L(\fk a,\wtd I)|_{\mc H_k^\infty}\qquad \mc R^Q(\fk a,\wtd I)|_{\mc H_k^\infty}=\mu_{k,i}\circ \mc R(\fk a,\wtd I)|_{\mc H_k^\infty}
\end{align}
Moreover, the closure of $\scr E_Q^w$ equals $\scr E_Q=\mu_{\blt,\star}\circ\scr E$, the categorical extension associated to $Q$, $\scr E$, and $(\boxdot_Q,\mu_{\blt,\star})$ (cf. Def. \ref{lb69}).
\end{thm}

Note that $\scr E^w$ and $\scr E^w_Q$ have the same label set $\fk H_i(I)$ for each $\mc H_i\in\Obj(\Rep_{\scr C}(\mc B_Q))$ and $I\in\mc J$.

In the following proof, symbols such as $\mc H_i^\infty(I)$ are defined as in Def. \ref{lb33} using $\mc A$, not using $\mc B_Q$. (There is no such ambiguity for the notation $\mc H_i^\infty$.)

\begin{proof}
Step 1. The $\mc L^Q$ and $\mc R^Q$ operations defined by \eqref{eq16} are smooth and localizable (cf. Rem. \ref{lb25}). In particular, they are closable. Let us prove that $\scr E_Q^w$ satisfies the axioms of a weak categorical extension (cf. Def. \ref{lb8}). Choose $\mc H_i,\mc H_k,\mc H_{k'}\in \Obj(\Rep_{\scr C}(\mc B_Q))$. We leave the verification of the intertwining property and the weak locality to the later steps.

Isotony: This is obvious.

Naturality: Let $\fk a\in\fk H_i(I)$ and $G\in\Hom_Q(\mc H_k,\mc H_{k'})$ We have a commutativity diagram
\begin{equation}
\begin{tikzcd}[column sep=large]
\quad\mc H_k^\infty\quad \arrow[d,"{\mc L(\fk a,\wtd I)}"'] \arrow[r,"G"] & \quad\mc H_{k'}^\infty\quad \arrow[d,"{\mc L(\fk a,\wtd I)}"]  \\
(\mc H_i\boxdot\mc H_k)^\infty \arrow[r,"\idt_i\boxdot G"]  \arrow[d,"\mu_{i,k}"']         & (\mc H_i\boxdot\mc H_{k'})^\infty  \arrow[d,"{\mu_{i,{k'}}}"]        \\
(\mc H_i\boxdot_Q\mc H_k)^\infty \arrow[r,"\idt_i\boxdot_Q G"]&(\mc H_i\boxdot_Q\mc H_{k'})^\infty
\end{tikzcd}
\end{equation}
where the first diagram commutes by the naturality of $\scr E^w$, and the second one commutes due to \eqref{eq19}. Their composition gives the naturality of the $\mc L^Q$ operators. Similarly, $\mc R^Q$ satisfies the naturality.

Braiding: For each $\eta\in\mc H_k^\infty$,  by \eqref{eq20} we have
\begin{align*}
\mc R^Q(\fk a,\wtd I)\eta=\mu_{k,i}\mc R(\fk a,\wtd I)\eta=\mu_{k,i}\ss_{i,k}\mc L(\fk a,\wtd I)\eta=\ss_{i,k}^Q\mu_{i,k}\mc L(\fk a,\wtd I)\eta=\ss_{i,k}^Q\mc L^Q(\fk a,\wtd I)\eta
\end{align*}

Neutrality: This is due to the above braiding and the fact that $\ss_{i,a}^Q$ is the identity map if we identify $\mc H_i\boxdot_Q\mc H_a$ and $\mc H_a\boxdot_Q\mc H_i$ with $\mc H_i$ through the unitors in \eqref{eq21}: 
\begin{align*}
\ss_{i,a}^Q\mu_{i,a}\xlongequal{\eqref{eq20}}\mu_{a,i}\ss_{i,a}\xlongequal{\eqref{eq21}}\mu^i\ss_{i,a}\xlongequal{\eqref{eq21}}\mu_{i,a}
\end{align*}

Reeh-Schlieder + density of fusion products: These follow from the corresponding properties of $\scr E^w$ and the fact that $\mu_{i,a},\mu_{i,k}$ are surjective. (It is well-known that $\mu_{i,k}\mu_{i,k}^*$ is a scalar, cf. for example \cite[Prop. 3.3]{Gui22}.)

Rotation covariance: This follows from the corresponding property of $\scr E^w$ and the fact that $\mu_{i,k}$ intertwines the actions of the rotation group.\\[-1ex]

Sep 2. We now check that $\scr E^w_Q$ satisfies the intertwining property. By Thm. \ref{lb26}, elements of $\mc B_Q(I')$ are precisely of the form 
\begin{align}
X=\mu L(\xi,\bpr\wtd I)|_{\mc H_a} \label{eq22}
\end{align}
where $\xi\in\mc H_a(I')$, and recall that $\bpr\wtd I$ is the anticlockwise complement of $\wtd I$. By Thm. \ref{lb27}, we have (recalling Rem. \ref{lb28})
\begin{gather*}
\pi_{k,I'}(\mu L(\xi,\bpr\wtd I)|_{\mc H_a})=\mu_{a,k}L(\xi,\bpr\wtd I)|_{\mc H_k}\\
\pi_{k\boxdot_Q i,I'}(\mu L(\xi,\bpr\wtd I)|_{\mc H_a})=\mu_{a,k\boxdot_Q i}L(\xi,\bpr\wtd I)|_{\mc H_{k\boxdot_Q i}}
\end{gather*}
Thus, checking the intertwining property for $\mc R^Q$ means proving that for every $\fk a\in\fk H_i(\wtd I)$ and any $X=\eqref{eq22}$, the following diagram of closable operators commutes strongly:
\begin{equation}\label{eq23}
\begin{tikzcd}[column sep=large]
\mc H_k\arrow[r,"{\mc R^Q(\fk a,\wtd I)}"] \arrow[d,"{\mu_{a,k}\circ L(\xi,\bpr\wtd I)}"'] & (\mc H_k\boxdot_Q\mc H_i) \arrow[d,"{\mu_{a,k\boxdot_Q i}\circ L(\xi,\bpr\wtd I)}"] \\
\mc H_k \arrow[r,"{\mc R^Q(\fk a,\wtd I)}"]           & (\mc H_k\boxdot_Q\mc H_i)        
\end{tikzcd}
\end{equation}
By Prop. \ref{lb4}, there exists a sequence $(\xi_n)_{n\in\Zbb_+}$ in $\mc H_a^\infty(I')$ such that $L(\xi_n,\bpr\wtd I)|_{\mc H_j}$ converges strongly-* to $L(\xi,\bpr\wtd I)|_{\mc H_j}$ for each $\mc H_j\in\Obj(\Rep_{\scr C}(\mc B_Q))$. It suffices to prove \eqref{eq23} when $\xi$ is replaced by each $\xi_n$. Thus, it suffices to prove that \eqref{eq23} commutes strongly for every $\xi\in\mc H_a^\infty(I')$. Note that in this case $L(\xi,\bpr\wtd I)$ is bounded and smooth, and the same is true for $\mu_{a,\star}\circ L(\xi,\bpr\wtd I)$. Thus, it suffices to prove that \eqref{eq23} commutes adjointly. 

Consider the diagrams
\begin{equation}\label{eq24}
\begin{tikzcd}[column sep=large]
\mc H_k^\infty \arrow[r,"{\mc R(\fk a,\wtd I)}"] \arrow[d,"{L(\xi,\bpr\wtd I)}"'] & (\mc H_k\boxdot\mc H_i)^\infty \arrow[r,"\mu_{k,i}"] \arrow[d,"{L(\xi,\bpr\wtd I)}"] & (\mc H_k\boxdot_Q\mc H_i)^\infty \arrow[d,"{L(\xi,\bpr\wtd I)}"] \\
(\mc H_a\boxdot\mc H_k)^\infty \arrow[r,"{\mc R(\fk a,\wtd I)}"] \arrow[d,"\mu_{a,k}"'] & (\mc H_a\boxdot\mc H_k\boxdot\mc H_i)^\infty \arrow[r,"\idt_a\boxdot\mu_{k,i}"] \arrow[d,"\mu_{a,k}\boxdot\idt_i"] & (\mc H_a\boxdot(\mc H_k\boxdot_Q\mc H_i))^\infty \arrow[d,"{\mu_{a,k\boxdot_Q i}}"] \\
\mc H_k^\infty \arrow[r,"{\mc R(\fk a,\wtd I)}"]           & (\mc H_k\boxdot\mc H_i)^\infty \arrow[r,"\mu_{k,i}"]           & (\mc H_k\boxdot_Q\mc H_i)^\infty          
\end{tikzcd}
\end{equation}
By Thm. \ref{lb16}, the closure of $\mc R(\fk a,\wtd I)$ is a right operator of $\scr E$, and hence commutes strongly with $L(\xi,\bpr\wtd I)$ by the locality in Thm. \ref{lb7}. Thus, the upper left cell of \eqref{eq24} commutes adjointly. The upper right and the lower left cells commute adjointly by the naturality axioms in Def. \ref{lb31} and \ref{lb8}, and by $(\mu_{a,k}\boxdot\idt_i)^*=\mu_{a,k}^*\boxdot\idt_i$ and $(\idt_a\boxdot\mu_{k,i})^*=\idt_a\boxdot\mu_{k,i}^*$. By Thm. \ref{lb30}, the lower right cell commutes adjointly. Thus, the largest rectangle (which is just \eqref{eq23}) commutes adjointly.

We have finished proving that the $\mc R^Q$ operators satisfy the strong intertwining property. Combining this fact with the braiding axiom proved in Step 1, we see that the $\mc L^Q$ operators also satisfy the strong intertwining property. \\[-1ex]

Step 3. We now show that the closure of $\mc L^Q(\fk a,\wtd I)$ resp. $\mc R^Q(\fk a,\wtd I)$ is a left resp. right operator of $\scr E_Q$. Then by Lem. \ref{lb14} and the locality (satisfied by the left and right operators of $\scr E_Q$) in Thm. \ref{lb7}, we immediately know that $\scr E_Q^w$ satisfies the axiom of weak locality in Def. \ref{lb8}, and hence that $\scr E_Q^w$ is a weak categorical extension. It also follows from Thm. \ref{lb16} that $\scr E_Q$ is the unique closure of $\scr E_Q^w$.

Let us study $\mc L^Q(\fk a,\wtd I)$; the treatment of $\mc R^Q(\fk a,\wtd I)$ is similar and hence omitted. Let $\xi=\mc L^Q(\fk a,\wtd I)\Omega$, which is in $\mc H_i^\infty$. Let us show that $\xi\in\mc H_i^{Q,\pr}(I)$ in the sense of Def. \ref{lb29}, i.e., the linear map
\begin{align}
Y\Omega\in\mc B_Q(I')\Omega\qquad\mapsto \qquad Y\xi\in\mc H_i \label{eq25}
\end{align}
is closable. In fact, since we have proved in Step 2 that $\mc L^Q(\fk a,\wtd I)$ satisfies the strong intertwining property, the closableness of \eqref{eq25} follows directly from \cite[Prop. 2.33]{Gui26}. We present the proof below since it is short: 

Choose any $Y\in\mc B_Q(I')$. By the intertwining property proved in Step 2, we have $Y\cdot \ovl{\mc L^Q(\fk a,\wtd I)}|_{\mc H_a}\subset \ovl{\mc L^Q(\fk a,\wtd I)} Y|_{\mc H_a}$. In particular, $Y\Omega$ is in the domain of $\ovl{\mc L^Q(\fk a,\wtd I)}$. This proves that $\eqref{eq25}\subset \ovl{\mc L^Q(\fk a,\wtd I)} |_{\mc H_a}$. So \eqref{eq25} is closable since $\ovl{\mc L^Q(\fk a,\wtd I)} |_{\mc H_a}$ is closed.

It remains to prove for each $\mc H_j\in\Obj(\Rep_{\scr C}(\mc B_Q))$ that
\begin{align}
\ovl{\mc L^Q(\fk a,\wtd I)}|_{\mc H_j}=\scr L^Q(\xi,\wtd I)|_{\mc H_j} \label{eq28}
\end{align}
For each $\mc H_k\in\Obj(\Rep_{\scr C}(\mc B_Q))$ and $\eta\in\mc H_j^\infty(I')$, consider 
\begin{equation}\label{eq26}
\begin{tikzcd}[column sep=large]
\mc H_k^\infty \arrow[r,"{R(\eta,\wtd I')}"] \arrow[d,"{\mc L(\fk a,\wtd I)}"'] & (\mc H_k\boxdot\mc H_j)^\infty \arrow[r,"\mu_{k,j}"] \arrow[d,"{\mc L(\fk a,\wtd I)}"] & (\mc H_k\boxdot_Q\mc H_j)^\infty \arrow[d,"{\mc L(\fk a,\wtd I)}"] \\
(\mc H_i\boxdot\mc H_k)^\infty \arrow[r,"{R(\eta,\wtd I')}"] \arrow[d,"\mu_{i,k}"'] & (\mc H_i\boxdot\mc H_k\boxdot\mc H_j)^\infty \arrow[r,"\idt_i\boxdot\mu_{k,j}"] \arrow[d,"\mu_{i,k}\boxdot\idt_j"] & (\mc H_i\boxdot(\mc H_k\boxdot_Q\mc H_j))^\infty \arrow[d,"{\mu_{i,k\boxdot_Q j}}"] \\
(\mc H_i\boxdot_Q\mc H_k)^\infty \arrow[r,"{R(\eta,\wtd I')}"]           & ((\mc H_i\boxdot_Q\mc H_k)\boxdot\mc H_j)^\infty \arrow[r,"\mu_{i\boxdot_Qk,j}"]           & (\mc H_k\boxdot_Q\mc H_j)^\infty          
\end{tikzcd}
\end{equation}
where each of the four cells commutes adjointly by the same reasons for \eqref{eq24}. Thus, the largest rectangle commutes adjointly, i.e., 
\begin{equation}\label{eq27}
\begin{tikzcd}[column sep=large]
\mc H_k^\infty \arrow[r,"{R^Q(\eta,\wtd I')}"] \arrow[d,"{\mc L^Q(\fk a,\wtd I)}"'] & (\mc H_k\boxdot_Q\mc H_j)^\infty\arrow[d,"{\mc L^Q(\fk a,\wtd I)}"] \\
(\mc H_i\boxdot_Q\mc H_k)^\infty \arrow[r,"{R^Q(\eta,\wtd I')}"] &(\mc H_k\boxdot_Q\mc H_j)^\infty
\end{tikzcd}
\end{equation}
commutes adjointly. Setting $\mc H_k=\mc H_a$, we get $\mc L^Q(\fk a,\wtd I)R^Q(\eta,\wtd I')\Omega=R^Q(\eta,\wtd I')\mc L^Q(\fk a,\wtd I)\Omega$, i.e., 
\begin{align*}
\mc L^Q(\fk a,\wtd I)\eta=R^Q(\eta,\wtd I')\xi
\end{align*}
By the locality of left and right operators in $\scr E_Q$ (cf. Thm. \ref{lb7}), we have $R^Q(\eta,\wtd I')\scr L^Q(\xi,\wtd I)\subset \scr L^Q(\xi,\wtd I)R^Q(\eta,\wtd I')$. Thus $\eta=R^Q(\eta,\wtd I')\Omega$ is in the domain of $\scr L^Q(\xi,\wtd I)$, and 
\begin{align}
\scr L^Q(\xi,\wtd I)\eta=R^Q(\eta,\wtd I')\xi\label{eq29}
\end{align}
This proves that $\scr L^Q(\xi,\wtd I)|_{\mc H_j}$ equals $\ovl{\mc L^Q(\fk a,\wtd I)}|_{\mc H_j}$ when restricted to $\mc H_j^\infty(I')$, a dense subspace of the domains of the two operators. 

Note that $\mc H_j^\infty(I')$ is QRI (recall Def. \ref{lb32}). (Choose any $J\in\mc J$ compactly supported in $I'$, then $\mc H_j^\infty(J)$ is a dense subspace of $\mc H_j^\infty(I')$, and $\varrho(t)\mc H_j^\infty(J)\subset\mc H_j^\infty(I')$ for any $t$ such that $\varrho(t)J\subset I$.) Thus, since $\mc L^Q(\fk a,\wtd I)|_{\mc H_j^\infty}$ is smooth and localizable (by Rem. \ref{lb25}), $\mc H_j^\infty(I')$ is a core for $\mc L^Q(\fk a,\wtd I)|_{\mc H_j^\infty}$. This proves that ``$\subset$" holds in \eqref{eq28}.

To finish proving \eqref{eq28}, it remains to prove that $\mc H_j^\infty(I')$ is a core for $\scr L^Q(\xi,\wtd I)|_{\mc H_j}$. Note that we cannot directly use Prop. \ref{lb4}-(b), since $\mc H_j^\infty(I')$ is defined using the smooth L and R operators of $\mc A$ but not those of $\mc B_Q$. However, we can use Prop. \ref{lb4}-(a), which says in part that for every $\eta\in\mc H_j(I')$ there is a sequence $(\eta_n)$ in $\mc H_j^\infty(I')$ converging to $\eta$ such that $R(\eta_n,\wtd I)$ converges strongly to $R(\eta,\wtd I)$ when acting on every object of $\scr C$. It follows that $R^Q(\eta_n,\wtd I)|_{\mc H_i}=\mu_{i,j}R(\eta_n,\wtd I)|_{\mc H_i}$ converges strongly to $R^Q(\eta,\wtd I)|_{\mc H_i}=\mu_{i,j}R(\eta,\wtd I)|_{\mc H_i}$. Thus, by \eqref{eq29} (with $\eta$ replaced by $\eta_n$) we have
\begin{align*}
\lim_n\scr L^Q(\xi,\wtd I)\eta_n=\lim_nR^Q(\eta_n,\wtd I')\xi=R^Q(\eta,\wtd I')\xi
\end{align*}
The existence of the limit on the LHS finishes the proof that $\mc H_j^\infty(I')$ is a core for $\scr L^Q(\xi,\wtd I)|_{\mc H_j(I')}$. Therefore,  $\mc H_j^\infty(I')$ is a core for $\scr L^Q(\xi,\wtd I)|_{\mc H_j}$ since $\mc H_j(I')$ is so (cf. Thm. \ref{lb7}).
\end{proof}

\section{(Weak) categorical extensions associated to unitary VOAs}

Throughout this paper, we assume that \uwave{any VOA $V$ is of CFT-type} (i.e., $V$ has $L_0$-grading $V=\bigoplus_{n\in\Nbb}V(n)$ where the weight-$0$ subspace $V(0)$ is spanned by the \textbf{vacuum vector $\Omega$}). Moreover, we assume that \uwave{$V(n)$ is finite-dimensional for each $n\in\Zbb$.}  All $V$-modules are understood to be semisimple, i.e., they are \textit{finite} direct sums of irreducible (ordinary) $V$-modules. In particular, for any $V$-module $W$, the operator $L_0$ is diagonalizable. We assume that each $L_0$-eigenspace of $W$ is finite-dimensional. (This is automatic when $V$ is $C_2$-cofinite, cf. \cite{Buhl02}.)

Recall that if a VOA $V$ is self-dual (e.g. when $V$ is unitary), $C_2$-cofinite, and rational, then the category $\Rep(V)$ of (semisimple ordinary) $V$-modules is a modular tensor category by \cite{Hua08a,Hua08b}. See \cite{DL14,CKLW18} for details about unitary VOAs.

\subsection{Preliminaries}

The goal of this subsection is to explain the concepts in Conditions I and II (cf. Def. \ref{lb34}).

\begin{df}\label{lb64}
We say that $V$ is \textbf{complete unitarity} if the following conditions hold:
\begin{itemize}
\item $V$ is CFT-type, $C_2$-cofinite and rational.
\item $V$ is a unitary VOA. Moreover, every irreducible $V$-module (and hence every $V$-module) is \textbf{unitarizable}, i.e., admits a unitary structure. (See \cite{DL14,Gui19a} for the definition of unitary $V$-modules.)
\item The canonical non-degenerate sesquilinear forms on the spaces of unitary intertwining operators (defined in \cite[Sec. 6]{Gui19b}) are positive definite.
\end{itemize}
\end{df}

The most important property about complete unitarity is the following fact proved in \cite[Thm. 7.9]{Gui19b} based on Huang-Lepowsky's theory of vertex tensor categories \cite{HL95a,HL95b,HL95c,Hua95,Hua05a,Hua05b,Hua08a,Hua08b}.

\begin{thm}\label{lb61}
Let $V$ be completely unitary. Let
\begin{align*}
\RepV=\text{the $C^*$-category of unitary $V$-modules}
\end{align*}
whose hom space is
\begin{align*}
\Hom_V(W_i,W_j)=\{&\text{linear map }T:W_i\rightarrow W_j\text{ such that }\\
&TY_i(v)_n=Y_j(v)_nT\text{ for all }v\in V,n\in\Zbb\}
\end{align*}
and whose $*$-structure is the adjoint. Then $\RepV$, together with the tensor and braiding structures
\begin{align*}
\boxdot_V=\boxdot\qquad \ss^V=\ss
\end{align*}
defined by Huang-Lepowsky's vertex tensor category theory and the canonical inner products in Def. \ref{lb64}, is a unitary modular tensor category. (In particular, $\RepV$ is a braided $C^*$-tensor category, called the \textbf{Huang-Lepowsky braided $C^*$-tensor category} of $V$.)
\end{thm}
Note that for each $W_i,W_j\in\Obj(\RepV)$, every $T\in\Hom_V(W_i,W_j)$ is bounded and $T^*\in\Hom_V(W_j,W_i)$ can be defined. (To see this, restrict $T$ to the (finitely many) irreducible components of $W_i$ and $W_j$.)

Unless otherwise stated, we identify $(W_i\boxdot W_j)\boxdot W_k$ with $W_i\boxdot (W_j\boxdot W_k)$ via the associator, and identify $V\boxdot W_i$ and $W_i\boxdot V$ with $W_i$ via the unitors.

\begin{df}
Assume that $V$ is completely unitary. Given a set of unitary $V$-modules $\mc F^V$, we say that $\RepV$ is \textbf{$\boxdot$-generated by $\mc F^V$} if each irreducible unitary $V$-module is (unitarily) equivalent to a submodule of $W_1\boxdot\cdots\boxdot W_n$ for some $W_1,\dots,W_n\in\mc F^V$.
\end{df}

For each (semisimple) unitary $V$-module $W_i$, we let
\begin{align*}
Y_{W_i}(v,z)=Y_i(v,z)=\sum_{n\in\Zbb} Y_i(v)_nz^{-n-1}
\end{align*}
denote the vertex operation of $W_i$. We refer the readers to \cite{CKLW18} for the meaning of energy-bounds. 

The Virasoro operators $\{L_n\}$ act on each unitary VOA module $W_i$ and satisfies $\bk{L_nw_1|w_2}=\bk{w_1|L_{-n} w_2}$. In particular, the closure of ${\ovl L_0}$ is a self-adjoint positive operator. Note also that $L_0$ gives the grading $W_i=\bigoplus_{s\in\Rbb_{\geq0}}W_i(s)$ where each weight-$s$ eigenspace $W_i(s)$ is finite-dimensional. Eigenvectors of $L_0$ are called \textbf{homogeneous vectors}.

\begin{df}\label{lb53}
If $W_i$ is a unitary $V$-module, we let $\mc H_i$ be the Hilbert space completion of the inner product space $W_i$. We let
\begin{align}
\mc H_i^\infty=\bigcap_{n\in\Nbb}\Dom(\ovl{L_0}^n)
\end{align}
\end{df}

More generally, a \textbf{type $W_k\choose W_iW_j$ intertwining operator $\mc Y$} of $V$ (where $W_i,W_j,W_k$ are $V$-modules) is defined as in \cite{FHL93}, and is written as
\begin{align*}
\mc Y(w_i,z)=\sum_{s\in\Rbb}\mc Y(w_i)_s z^{-s-1}
\end{align*}
where $w_i\in W_i$, and each $\mc Y(w_i)_s$ is a linear map $W_j\rightarrow W_k$. 

\begin{cv}
If $\mc Y$ is a type $W_k\choose W_iW_j$ intertwining operator and $T\in\Hom_V(W_j,W_l)$, then $T\mc Y$ denotes the type $W_l\choose W_iW_j$ intertwining operator defined by
\begin{align*}
T\mc Y(w^{(i)},z)w^{(j)}=\sum_{s\in\Rbb}T\circ \mc Y(w^{(i)})_s w^{(j)}z^{-n-1}
\end{align*}
for all $w^{(i)}\in W_i,w^{(j)}\in W_j$.
\end{cv}

\begin{df}
Given a type $W_k\choose W_iW_j$ intertwining operator $\mc Y$, we call $W_i,W_j,W_k$ respectively the \textbf{charge space}, the \textbf{source space}, and the \textbf{target space} of $\mc Y$. We say that $\mc Y$ is \textbf{unitary} resp. \textbf{irreducible} if $W_i,W_j,W_k$ are unitary resp. irreducible $V$-modules.
\end{df}

The condition of energy bounds for unitary intertwining operators was defined in \cite[Sec. 3.1]{Gui19a} in the same spirit as that of \cite{CKLW18}: 
\begin{df}
Let $\mc Y$ be a type $W_k\choose W_iW_j$ \textit{unitary} intertwining operator. Let $w^{(i)}\in W_i$ be homogeneous. We say that $\mc Y(w^{(i)},z)$ is \textbf{energy-bounded} if there exist $M,a,b\geq0$ such that for every $s\in\Rbb,w^{(j)}\in W_j$ we have
\begin{align*}
\Vert \mc Y(w^{(i)})_sw^{(j)}\Vert\leq M(1+|s|)^b\Vert(1+L_0)^a w^{(j)}\Vert
\end{align*}
We say that the intertwining operator $\mc Y$ is \textbf{energy-bounded} if $\mc Y(w^{(i)},z)$ is energy-bounded for every homogeneous vector $w^{(i)}$ of the charge space.
\end{df}

\begin{df}
We say that $V$ is \textbf{energy-bounded} if the vertex operator $Y$ for the vacuum module $V$ is energy-bounded. We say that $V$ is \textbf{strongly energy-bounded} if for every unitary $V$-module $W_i$ (equivalently, for every unitary irreducible $V$-module $W_i$), the vertex operator $Y_{W_i}=Y_i$, as a unitary intertwining operator of type $W_i\choose V W_i$, is energy-bounded.
\end{df}

\begin{df}
Let $\wtd I=(I,\arg_I)\in\Jtd$. A (smooth) \textbf{arg-valued function supported in $\wtd I$} is a pair $\wtd f=(f,\arg_I)$ where $f\in C_c^\infty(I)$. We let
\begin{align*}
C_c^\infty(\wtd I)=\{\text{smooth arg-valued functions supported in }\wtd I\}
\end{align*}
If $\mc Y$ is an energy-bounded type $W_k\choose W_iW_j$ intertwining operator of $V$, for each homogeneous $w\in W_i$ and $\wtd f\in C_c^\infty(\wtd I)$, we define the \textbf{smeared intertwining operator}
\begin{gather}
\begin{gathered}
\mc Y(w,\wtd f)=\int_{\theta\in\arg_I(I)}\mc Y(w^{(i)},e^{\im\theta})f(e^{\im\theta})\frac{e^{\im\theta}}{2\pi}d\theta\\
\Dom\big(\mc Y(w,\wtd f)\big)=\mc H_j^\infty
\end{gathered}
\end{gather}  
viewed a (densely-defined) unbounded smooth linear maps $\mc H_j\rightarrow\mc H_k$ with dense domain $\mc H_j^\infty$. (Cf. \cite[Sec. 4.4]{Gui21a} and \cite[Ch. 3]{Gui19a} for details.) For the vertex operator $Y_i$ for a unitary $V$-module $W_i$, we write
\begin{align*}
Y_i(v,f)=Y_i(v,\wtd f)
\end{align*}
since $Y_i(v,\wtd f)$ is independent of the arg-values.
\end{df}

\begin{rem}\label{lb54}
Let $\mc Y_1(w_1,\wtd f_1),\dots,\mc Y_n(w_n,\wtd f_n)$ be smeared intertwining operators. Assume that $W_j$ is the source space of $\mc Y_n$ and $W_k$ is the target space of $\mc Y_1$. We understand the product
\begin{align}
\mc Y_1(w_1,\wtd f_1)\cdots\mc Y_n(w_n,\wtd f_n)\label{eq30}
\end{align}
as a densely defined unbounded linear map $\mc H_j\rightarrow\mc H_k$ with domain $\mc H_j^\infty$. It is the product defined in the usual way if the target space of $\mc Y_i$ equals the source space of $\mc Y_{i+1}$ (if $1\leq i<n$); otherwise, it is zero. Moreover,
\begin{align*}
\text{the product \eqref{eq30} is smooth and localizable}
\end{align*}
(recall Def. \ref{lb32}). The smoothness is obvious because each factor $\mc Y_i(w_i,\wtd f_i)$ is smooth. The localizablity of \eqref{eq30} follows from that of $\ovl {L_0}^n$ (\cite[Lem. 7.2]{CKLW18}); see \cite[Prop. 3.2]{Gui26} for more explanation.   
\end{rem}

\begin{df}\label{lb101}
Let $\mc Y$ be a type $W_k\choose W_i W_j$ intertwining operator. We say that $\mc Y$ satisfies the \textbf{strong intertwining property} if for every homogeneous $w\in W_i,v\in V$, every $\wtd I\in\Jtd$, and every $\wtd f\in C_c^\infty(\wtd I),g\in C_c^\infty(I')$, the following diagram of closable smooth operators commutes strongly:
\begin{equation}\label{eq42}
\begin{tikzcd}[column sep=large]
\mc H_j^\infty \arrow[d,"{\mc Y(w,\wtd f)}"'] \arrow[r,"{Y_j(v,g)}"] & \mc H_j^\infty \arrow[d,"{\mc Y(w,\wtd f)}"] \\
\mc H_k^\infty \arrow[r,"{Y_k(v,g)}"]           & \mc H_k^\infty          
\end{tikzcd}
\end{equation}
\end{df}

\begin{df}
We say that $V$ is \textbf{strongly local} (\cite{CKLW18}) if the vacuum vertex operator $Y(\cdot,z)$ (viewed as a type $V\choose VV$ intertwining operator) satisfies the strong intertwining property.
\end{df}

\begin{rem}\label{lb56}
Let $W_i$ be a unitary module of a unitary VOA $V$. If we let $U_i$ denote the unitary representation of $\UPSU$ on $\mc H_i$ integrated from $L_0,L_{\pm1}$, then for each $g\in\UPSU$ we have $U_i(g)\mc H_i^\infty=\mc H_i^\infty$. Therefore, $U_i(g)$ is smooth.
\end{rem}

\begin{proof}
This is due to the fact that $\UPSU$ is generated by the $e^{\im t l_0},e^{\im t (l_1+l_{-1})},e^{t(l_1-l_{-1})}$ (where $l_0,l_{\pm 1}$ are the standard generators of the Lie algebra of $\UPSU$) and that $e^{\im t \ovl{L_0}},e^{\im t (\ovl{L_1+L_{-1}})},e^{t(\ovl{L_1-L_{-1}})}$ preserve $\mc H_i^\infty$ by \cite[Prop. 2.1]{TL99}. In fact, $U_i(g)$ is smooth for every $g\in\Diffp(\Sbb^1)$; see \cite[Sec. 3.2]{Gui26} for more explanation.
\end{proof}

\begin{pp}\label{lb51}
Let $V$ be a unitary VOA. Let $\mc Y$ be an energy-bounded type $W_k\choose W_i W_j$ intertwining operator. Let $w^{(i)}\in W_i$ be quasi-primary (resp. homogeneous) with $L_0w^{(i)}=dw^{(i)}$ ($d\in\Rbb$). Then for every $\wtd I\in\Jtd$ and $\wtd f\in C_c^\infty(\wtd I)$, and for every $g\in\UPSU$ (resp. every $g=e^{\im t\ovl{L_0}}$ where $t\in\Rbb$), there is $\wtd f_{g,d}\in C_c^\infty(g\wtd I)$ depending only on $\wtd f,g,d$ such that
\begin{align}
U_k(g)\mc Y(w^{(i)},\wtd f)U_j(g)^*=\mc Y(w^{(i)},\wtd f_{g,d})
\end{align}
where $U_j$ and $U_k$ are the unitary representations of $\UPSU$ on $\mc H_j,\mc H_k$ integrated from $L_0,L_{\pm1}$.
\end{pp}

\begin{proof}
See \cite[Prop. 3.4]{Gui26} for details and the explicit formula of $\wtd f_{g,d}$.
\end{proof}

\subsection{Conditions I and II}

\begin{df}\label{lb34}
We say that $V$ satisfies \textbf{Condition I} if the following hold:
\begin{enumerate}[label=(\alph*)]
\item $V$ is completely unitary.
\item Every irreducible unitary intertwining operator of $V$ is energy-bounded and satisfies the strong intertwining property. 
\end{enumerate}
We say that $V$ satisfies \textbf{Condition II} if the following hold: 
\begin{enumerate}[label=(\arabic*)]
\item $V$ is completely unitary, strongly energy-bounded, and strongly local.
\item A set $\mc F^V$ of unitary $V$-modules $\boxdot$-generating $\RepV$ is chosen. 
\item If $\mc Y$ is a unitary intertwining operator of $V$ whose charge space belongs to $\mc F^V$ and whose source space and target space are irreducible, then $\mc Y$ is energy-bounded and satisfies the strong intertwining property. 
\end{enumerate}
\end{df}

The irreducibilities  assumed in I and II are redundant: see Rem. \ref{lb46}. Also, unlike in \cite{Gui26}, we do not assume that the objects in $\mc F^V$ are irreducible.  While this relaxation does not lead to any essential differences, it simplifies certain aspects of our discussion in this paper.

\begin{rem}\label{lb36}
Note that if $V$ satisfies Condition I-(b), then $V$ is automatically strongly energy-bounded and strongly local. Thus, 
\begin{align*}
\text{Condition I}\qquad\Longrightarrow\qquad\text{Condition II}
\end{align*}
Moreover, it is clear that
\begin{gather*}
\text{Condition I}\quad\Longleftrightarrow\quad\text{Condition II}\\
\text{if every irreducible $V$-module is isomorphic to an object in $\mc F^V$}
\end{gather*}
\end{rem}

\begin{rem}\label{lb46}
Let $W_i$ be a unitary $V$-module with (finite orthogonal) irreducible decomposition $W_i=\bigoplus_a W_{i,a}$. Suppose that for every $a$, any \textit{irreducible} unitary intertwining operator with charge space $W_{i,a}$ is energy-bounded resp. satisfies the strong intertwining property. Then any unitary intertwining operator with charge space $W_i$ is energy-bounded resp. satisfies the strong intertwining property.

Therefore, if $V$ satisfies Condition II, then every unitary intertwining operator of $V$ with charge space in $\mc F^V$ is energy-bounded and satisfies the strong intertwining property; if $V$ satisfies Condition I, then every unitary intertwining operator of $V$ is energy-bounded and satisfies the strong intertwining property.
\end{rem}

\begin{proof}
The claim about energy bounds is obvious. The claim about the strong intertwining property follows from Lem. \ref{lb35}. See \cite[Rem. 3.14]{Gui26} for details.
\end{proof}


\begin{rem}
We have
\begin{align*}
\text{Condition II in Def. \ref{lb34}}\qquad\Longleftrightarrow\qquad\text{Condition B in \cite[Sec. 3.4]{Gui26}}
\end{align*}
\end{rem}

\begin{proof}
It is clear that ``$\Rightarrow$" is true. Condition B seems weaker than Condition II in the following aspects. First, in Condition B we assumed that $V$ is $E$-strongly local where $E$ is a set of generating quasi-primary vectors. But this automatically implies that $V$ is energy-bounded by \cite[Thm. 8.1]{CKLW18}. Second, in Condition B, the ``strong unitarity" was assumed instead of the ``complete unitarity" in Condition II. But Condition B implies the complete unitarity by \cite[Thm. 3.15]{Gui26}. Third, the assumptions in Condition B of the energy bounds and the strong intertwining property on the intertwining operators  are seemly weaker than those in Condition II; but they are in fact equivalent under the other assumptions of Condition B due to \cite[Prop. 3.28]{Gui26}.
\end{proof}

\begin{pp}\label{lb79}
Suppose that $V$ satisfies Condition II, and that every irreducible unitary intertwining operator of $V$ is energy-bounded. Then $V$ satisfies Condition I.
\end{pp}

\begin{proof}
\cite[Cor. 3.29]{Gui26}.
\end{proof}

The following theorem gives many examples of VOAs satisfying Condition I or II. We refer the readers to \cite[Sec. 5.4]{CKLW18} for the basic properties about \textbf{unitary subalgebras} and \textbf{(unitary) coset subalgebras}. If $V$ is a unitary subalgebra of a unitary VOA $U$, we let $V^c$ be the coset of $V$ in $U$, and let $V^{cc}$ be the coset of $V^c$ in $U$.

\begin{thm}\label{lb102}
The following are true.
\begin{enumerate}[label=(\arabic*)]
\item $V,V'$ are unitary VOAs satisfying Condition I (resp. II) if and only if the tensor product unitary VOA $V\otimes V'$ satisfies Condition I (resp. II).
\item Assume that $U$ is a unitary VOA satisfying Condition I (resp. II). Assume that $V$ is a unitary subalgebra of $U$ such that both $V$ and $V^c$ are $C_2$-cofinite and rational. Then $V^c$ satisfied Condition I (resp. II).
\end{enumerate}
\end{thm}

\begin{proof}
(1): ``$\Rightarrow$" follows from \cite[Thm. 3.37]{Gui26}. ``$\Leftarrow$" is a special case of the following part (2) (by setting $U=V\otimes V'$).

(2): Assume that $U$ satisfies Condition II. By \cite[Thm. 1.1]{CMSY24}, any (CFT-type) unitary VOA extension of $V$ is $C_2$-cofinite and rational. Therefore, both $V^c=V^{ccc}$ and $V^{cc}$ are $C_2$-cofinite and rational. Therefore, by \cite[Thm. 3.34]{Gui26}, $V^c$ satisfies  Condition II.

The claim about Condition I also follows by slightly adapting the proof of \cite[Thm. 3.34]{Gui26}: If $U$ satisfies Condition I, then in view of Rem. \ref{lb36}, the set $\mc F^U$ in the proof of Thm. 3.34  can be chosen to contain all irreducible unitary $U$-modules up to equivalence. Then the $\mc F^{V^c}$ defined as in that proof also contains all irreducible unitary $V^c$-modules up to equivalence. Therefore, $V^c$ satisfies Condition I.
\end{proof}

\subsection{The CKLW net $\mc A_V$ and the CWX functor $\fk F^V_\CWX:\RepV\rightarrow\Rep(\mc A_V)$}

\begin{df}
Suppose that $V$ is a unitary VOA. Let
\begin{align*}
\mc H_0=\text{the Hilbert space completion of $V$}
\end{align*}
Let $\mc A_V$ be the unique conformal net acting on $\mc H_0$ whose vacuum vector $\Omega$ equals that of $V$, whose projective representation of $\Diffp(\Sbb^1)$ is integrated from that of $\{L_n\}$, and which satisfies that for every $I\in\mc J$, the von Neumann algebra $\mc A_V(I)$ is generated by all $\ovl{Y(v,f)}$ where $v\in V$ is homogeneous and $f\in C_c^\infty(I)$. (Cf. \cite{CKLW18}.) We call $\mc A_V$ the \textbf{CKLW net} associated to $V$. 
\end{df}

\begin{df}\label{lb37}
Let $V$ be a strongly local VOA. We say that a $V$-module $W_i$ is \textbf{strongly integrable} (\cite{CWX}) if the following are true:
\begin{itemize}
\item $W_i$ is unitary. Moreover, the vertex operator $Y_i(\cdot,z)$, as a type $W_i\choose V W_i$ intertwining operator, is energy-bounded.
\item There is a (necessarily unique) $\mc A_V$-module $(\mc H_i,\pi_i)$ such that for any homogeneous $v\in V$ and any $I\in\mc J,f\in C_c^\infty(I)$,
\begin{align}
\pi_{i,I}(\ovl{Y(v,f)})=\ovl{Y_i(v,f)}\label{eq31}
\end{align}
\end{itemize}
\end{df}
Recall Rem. \ref{lb84} for the meaning of $\pi_{i,I}(\ovl{Y(v,f)})$.

\begin{rem}\label{lb52}
Suppose that $V$ is strongly local with central charge $c$, and let $W_i$ be a strongly integrable $V$-module. Let $\mc A_V$ be the CKLW net. Then the unitary representation of $\Gc\simeq\scr G_{\mc A_V}$ on $\mc H_i$ (as described in Subsec. \ref{lb18}) is integrated from the unitary representations of the Virasoro subalgebra of $V$.  

In particular, the action of $\UPSU$ on $\mc H_i$ inherited from that of $\scr G_c$ is integrated from $L_0,L_{\pm1}$, and the rotation group on $\mc H_i$ acts as $e^{\im t\ovl{L_0}}$. Therefore, the definition of $\mc H_i^\infty$ as in Def. \ref{lb53} agrees with that in \eqref{eq43}.
\end{rem}

\begin{proof}
See \cite[Prop. 3.10]{Gui26} or \cite[Prop. 4.9]{Gui21a}.
\end{proof}

\begin{thm}[\cite{CWX}, \cite{Gui19b} Thm. 4.3]\label{lb42}
Let $V$ be completely unitary and strongly local. Assume that all unitary $V$-modules are strongly integrable. Then the $*$-functor 
\begin{gather*}
\fk F_\CWX^V:\RepV\rightarrow\Rep(\mc A_V)\\
(W_i,Y_i)\mapsto (\mc H_i,\pi_i)\\
T\in\Hom_V(W_i,W_j)\mapsto T\in\Hom_{\mc A_V}(\mc H_i, \mc H_j)
\end{gather*}
(where $(\mc H_i,\pi_i)$ is defined as in Def. \ref{lb37}), called the \textbf{CWX functor}, is a fully faithful $*$-functor. Thus, $\fk F_\CWX^V$ implements an isomorphism of $C^*$-tensor categories from $\RepV$ to 
\begin{align}
\scr C=\fk F_\CWX^V(\RepV)
\end{align}
a full replete $C^*$-subcategory of $\Rep(\mc A_V)$ containing the identity object $\mc H_0$ and closed under taking submodules and finite direct sums.
\end{thm}

Thus we have
\begin{align}
(\mc H_i,\pi_i)=\fk F_\CWX^V(W_i,Y_i)\qquad\text{or simply}\qquad \mc H_i=\fk F_\CWX^V(W_i)
\end{align}
Note that the repleteness of $\scr C$ is obvious: if an $\mc A_V$-module is unitary equivalent (via a unitary $U$) to some $\fk F_\CWX^V(W_i,Y_i)$, then it is equal to $\fk F_\CWX^V(U^{-1}W_i,U^{-1}Y_i U)$. Also, it makes sense to set $\fk F_\CWX^V(T)=T$ since $T\in\Hom_V(W_i,W_j)$ must be bounded (since it is clearly so when restricted to each irreducible component of $W_i$).

\begin{rem}
The inverse functor
\begin{align*}
(\fk F_\CWX^V)^{-1}:\scr C\rightarrow\RepV
\end{align*}
can be written down explicitly. Indeed, if $(W_i,Y_i)$ and $(W_j,Y_j)$ are both sent by $\fk F_\CWX^V$ to $(\mc H_i,\pi_i)$, then by the full-faithfulness of $\fk F_\CWX^V$, there is a unique $T\in\Hom_V(W_i,W_j)$ that extends to $\idt_{\mc H_i}$. So $T$ is the identity map, which implies $W_i\subset W_j$ and $Y_j|_{W_i}=Y_i$. Switching $W_i$ and $W_j$, we obtain $W_i=W_j$ and $Y_i=Y_j$. Thus, we can define $(\fk F_\CWX^V)^{-1}$ by sending each $\fk F_\CWX^V(W_i,Y_i)$ to $(W_i,Y_i)$. The definition of $(\fk F_\CWX^V)^{-1}$ on Hom spaces is obvious. 
\end{rem}

\begin{df}\label{lb44}
In the remaining part of this article, we always let 
\begin{align}\label{eq105}
\scr C=\Rep^V(\mc A_V):=\fk F^V_\CWX(\RepV)
\end{align}
(which is a full and replete $C^*$-subcategory of $\Rep(\mc A_V)$), and let $(\scr C,\boxdot,\ss)$ be the image of the Huang-Lepowsky braided $C^*$-tensor category $(\RepV,\boxdot,\ss)$ under the braided $*$-functor $(\fk F^V_\CWX,\id)$, which is
\begin{align}
\boxed{~(\Rep^V(\mc A_V),\boxdot,\ss)\equiv (\Rep^V(\mc A_V),\boxdot_V,\ss^V):=(\fk F^V_\CWX,\id)(\RepV,\boxdot,\ss)~}
\end{align}
In other words, we have an isomorphism of braided $C^*$-tensor categories:
\begin{align}
(\fk F^V_\CWX,\id):(\RepV,\boxdot,\ss)\xlongrightarrow{\simeq} (\Rep^V(\mc A_V),\boxdot,\ss)
\end{align}
where $\id$ is the identity tensorator. Thus, since $\mc H_i=\fk F_\CWX^V(W_i)$ and $\mc H_j=\fk F_\CWX^V(W_j)$, we have
\begin{align}
\fk F_\CWX^V(W_i\boxdot W_j)=\mc H_i\boxdot\mc H_j
\end{align}
\end{df}

In particular, we see that $\mc H_i\boxdot\mc H_j$, as a Hilbert space, is the completion of the inner product space $W_i\boxdot W_j$.

The strong intertwining property is closely related to the strong integrability due to the following elementary fact:

\begin{lm}\label{lb43}
Let $V$ be an energy-bounded VOA. Let $W_i$ be a unitary $V$-module whose vertex operator $Y_i$ is energy-bounded. Then the following are equivalent:
\begin{enumerate}[label=(\arabic*)]
\item $W_i$ is strongly integrable.
\item For each $I\in\mc J$ there is a set $\fk A_I$ of partial isometries $\mc H_0\rightarrow\mc H_i$ such that $\bigvee_{T\in\fk A_I}\Rng(T)$ is dense in $\mc H_i$, and that for each $T\in\fk A_I$, each homogeneous $v\in V$, and each $f\in C_c^\infty(I)$, we have
\begin{align}
T\ovl{Y(v,f)}\subset \ovl{Y_i(v,f)}T\qquad T^*\ovl{Y_i(v,f)}\subset \ovl{Y(v,f)}T^*
\end{align}
\item For each $I\in\mc J$ there is a set $\fk B_I$ of closable operators $\mc H_0\rightarrow\mc H_i$ (with dense domains in $\mc H_0$) such that $\bigvee_{T\in\fk B_I}\Rng(T)$ is dense in $\mc H_i$, and that for each $T\in\fk B_I$, each homogeneous $v\in V$, and each $f\in C_c^\infty(I)$, the following diagram of closable operators commutes strongly:
\begin{equation}
\begin{tikzcd}[column sep=large]
\mc H_0 \arrow[d,"T"'] \arrow[r,"{\ovl{Y(v,f)}}"] & \mc H_0 \arrow[d,"T"] \\
\mc H_i \arrow[r,"{\ovl{Y_i(v,f)}}"]           & \mc H_i        
\end{tikzcd}
\end{equation}
\end{enumerate}
\end{lm}

Here, $\Rng(T)$ is the range of $T$, i.e. $T(\Dom(T))$.

\begin{proof}
(1)$\Rightarrow$(2): Let $\pi_i$ be the representation of $\mc A_V$ on $\mc H_i$ as in Def. \ref{lb37}. By the basic property of normal representations of von Neumann algebras, there is a set $\fk A_I$ of partial isometries whose ranges spanning a dense subspace of $\mc H_i$ and satisfying $Tx=\pi_{i,I}(x)T$ and $T^*\pi_{i,I}(x)=xT^*$ for all $x\in\mc A_V(I)$. (Indeed, since $\mc A_V(I)$ is a type III factor, $\fk A_I$ can be chosen to have only one element which is a unitary.) So (2) follows.

(2)$\Rightarrow$(3): Obvious.

(3)$\Rightarrow$(1): By replacing each $T\in\fk B_I$ with the partial isometry in the polar decomposition of $\ovl T$, it suffices to assume that each $T\in\fk B_I$ is bounded (with domain $\mc H_0$). Then by a standard argument using Zorn's lemma (together will another application of polar decomposition), one can find a set $\fk A_I$ satisfying the statements in (2), and satisfying moreover that $T_1T_1^*\cdot T_2T_2^*=0$ if $T_1,T_2\in\fk A_I$ are distinct, and that $\sum_{T\in\fk A_I}TT^*=\idt_{\mc H_i}$. Then the pullback of the $\mc A_V(I)$-representation $(\bigoplus_{\fk A_I}\mc H_0,\bigoplus_{\fk A_I}\pi_{0,I})$ to $\mc H_i$  via the isometry $\xi\in\mc H_i\mapsto \bigoplus_{T\in\fk A_I}T^*\xi$ gives a (normal) representation $\pi_{i,I}$ of $\mc A_V(I)$ on $\mc H_i$. The collection $(\pi_{i,I})_{I\in\mc J}$ gives the desired representation of $\mc A_V$ on $\mc H_i$ making $W_i$ strongly integrable. (See \cite[Prop. 4.7]{Gui19b} for more details.)
\end{proof}

\begin{rem}
From (2) and (3) of Lem. \ref{lb43}, it is clear that if $W_i$ is an energy-bounded module of a unitary VOA $V$, and if $W_i=\bigoplus_a W_{i,a}$ is an (orthogonal) decomposition of $W_i$ into unitary submodules, then $W_i$ is strongly integrable if and only if every $W_{i,a}$ is strongly integrable.
\end{rem}

\begin{thm}\label{lb48}
Suppose that $V$ satisfies Condition II. Then all unitary $V$-modules are strongly integrable. Thus Thm. \ref{lb42} applies to $V$. 
\end{thm}

\begin{proof}
One checks that every unitary (and energy-bounded) $V$-module satisfies Lem. \ref{lb43}-(3) where $\fk B_I$ can be constructed using products of partial isometries in the polar decompositions of smeared intertwining operators localized in $I'$ (with any $\arg_{I'}$). See \cite[Thm. 3.16]{Gui26} for details. (Alternatively, one can define $\fk B_I$ using (the closures of) products of smeared intertwining operators localized in $I'$, and then invoke Lem. \ref{lb35} to check that $\fk B_I$ satisfies the assumption in Lem. \ref{lb43}-(3).)
\end{proof}

\begin{rem}\label{lb49}
Suppose that $V$ satisfies Condition II. Then in Condition II-(3), $\mc F^V$ can be extended to $\mc F^V\cup\{V\}$. In other words, every intertwining operator $\mc Y$ with charge space $V$ satisfies the strong intertwining property. Indeed, we can write $\mc Y=TY_{W_1}$ where $W_1,W_2$ are unitary $V$-modules, and $T\in\Hom_V(W_1,W_2)$. By Thm. \ref{lb48}, $W_1$ is strongly integrable. Therefore, by \cite[Lem. 3.9]{Gui26}, $Y_{W_1}$ satisfies the strong intertwining property. Therefore, by Lem. \ref{lb35}, $\mc Y$ satisfies the strong intertwining property.
\end{rem}

\begin{co}\label{lb57}
Assume that $V$ satisfies Condition II. For each unitary type $W_k\choose W_iW_j$ intertwining operator $\mc Y$ where $W_i\in\mc F^V\cup\{V\}$, each homogeneous $w^{(i)}\in W_i$, each $\wtd I\in\Jtd,\wtd f\in C_c^\infty(\wtd I)$, and  each $y\in\mc A(I')$, the following diagram commutes strongly:
\begin{equation}\label{eq41}
\begin{tikzcd}[column sep=large]
\mc H_j \arrow[d,"{\ovl{\mc Y(w^{(i)},\wtd f)}}"'] \arrow[r,"{\pi_{j,I'}(y)}"] & \mc H_j \arrow[d,"{\ovl{\mc Y(w^{(i)},\wtd f)}}"] \\
\mc H_k \arrow[r,"{\pi_{k,I'}(y)}"]           & \mc H_k        
\end{tikzcd}
\end{equation}
\end{co}

\begin{proof}
By Thm. \ref{lb48}, all unitary $V$-modules are strongly integrable. Thus, by the strong commutativity of \eqref{eq42} (and noting Rem. \ref{lb49}), we see that \eqref{eq41} commutes strongly if $y$ is replaced by $\ovl{Y(v,g)}$ for each homogeneous $v\in V$ and $g\in C_c^\infty(I')$. Since all such closed operators generate $\mc A(I')$, the claim of the corollary follows easily.
\end{proof}

\subsection{Categorical extension $\scr E_V$ associated to $V$; the Wassermann tensorator $\fk W^V:\mc H_i\boxdot_V\mc H_j\rightarrow\mc H_i\boxtimes_{\mc A_V}\mc H_j$}\label{lb65}

Let $V$ be a completely unitary VOA. We now recall the intertwining operators $\Gamma^V$ and $\Delta^V$, which are called $\mc L_i$ and $\mc R_i$ in \cite[Sec. 4.2]{Gui21a} and $\mc L^V,\mc R^V$ in \cite[Sec. 3.5]{Gui26}. All details can be found in \cite[Ch. 4]{Gui21a}.

Recall that $\boxdot_V=\boxdot$ is the tensor $*$-bifunctor and $\ss^V=\ss$ is the braiding in $\RepV$. $\Gamma^V$ associates to each $W_i,W_j\in\Obj(\RepV)$ a type $W_i\boxdot W_j\choose W_i~W_j$ intertwining $\Gamma^V(\cdot,z)$ satisfying that for every $W_k\in\Obj(\RepV)$ and every type $W_k\choose W_iW_j$ intertwining operator $\mc Y$ there is a unique $T\in\Hom_V(W_i\boxdot W_j,W_k)$ such that
\begin{align}\label{eq53}
\mc Y(w^{(i)},z)w^{(j)}=T\circ \Gamma^V(w^{(i)},z)w^{(j)}
\end{align}
for all $w^{(i)}\in W_i,w^{(j)}\in W_j$. Clearly, such $\Gamma^V$ is unique up to multiplication by a unitary endomorphism on the left. The actual expression of $\Gamma^V$ is not important. What is more important is its relationship with the braided $C^*$-tensor structure of $\RepV$, as described below.

$\Delta^V$ associates to each $W_i,W_j$ a type $W_j\boxdot W_i\choose W_i~W_j$ intertwining $\Delta^V(\cdot,z)$ satisfying
\begin{align}\label{eq34}
\Delta^V(w^{(j)},z)w^{(i)}=\ss_{i,j}\circ \Gamma^V(w^{(i)},z)w^{(j)}
\end{align}
for all $w^{(i)}\in W_i,w^{(j)}\in W_j$. Under the identification $(W_i\boxdot W_j)\boxdot W_k=W_i\boxdot(W_j\boxdot W_k)$ via the unitary associator, $\Gamma^V$ and $\Delta^V$ satisfy the braiding relation 
\begin{subequations}\label{eq35}
\begin{gather}
\Gamma^V(w^{(i)},z)\cdot\Delta^V(w^{(j)},\zeta)w^{(k)}=\Delta^V(w^{(j)},\zeta)\cdot\Gamma^V(w^{(i)},z)w^{(k)}\\
(\Gamma^V)^\dagger(w^{(i)},z)\cdot\Delta^V(w^{(j)},\zeta)w^{(i,k)}=\Delta^V(w^{(j)},\zeta)\cdot(\Gamma^V)^\dagger(w^{(i)},z)w^{(i,k)}
\end{gather}
\end{subequations}
for each $w^{(i)}\in W_i,w^{(j)}\in W_j,w^{(k)}\in W_k,w^{(i,k)}\in W_i\boxdot W_k$ and each $z,\zeta\in\Sbb^1$ equipped with $\arg z,\arg\zeta$ such that $\arg z-2\pi<\arg\zeta<\arg z$. Here, 
\begin{align*}
(\Gamma^V)^\dagger(w^{(i)},z)=\Gamma^V(e^{\ovl zL_1}(e^{-\im\pi}\ovl {z^{-2}})^{L_0}w^{(i)},\ovl {z^{-1}})^\dagger
\end{align*} 
is the adjoint intertwining operator of $\Gamma^V$. See \cite[Sec. 4.3]{Gui21a} for details.

If $F\in\Hom_V(W_i,W_k)$ and $G\in\Hom_V(W_j,W_l)$, then for each $w^{(i)}\in W_i,w^{(j)}\in W_j$ we have
\begin{subequations}\label{eq36}
\begin{gather}
(F\boxdot G)\Gamma^V(w^{(i)},z)w^{(j)}=\Gamma^V(Fw^{(i)},z)Gw^{(j)}\\
(G\boxdot F)\Delta^V(w^{(i)},z)w^{(j)}=\Delta^V(Fw^{(i)},z)Gw^{(j)}
\end{gather}
\end{subequations}

Under the identifications $V\boxdot W_i=W_i=W_i\boxdot V$ via the unitors, we have
\begin{subequations}\label{eq37}
\begin{gather}
\Gamma^V(v,z)w^{(i)}=\Delta^V(v,z)w^{(i)}=Y_i(v,z)w^{(i)}\label{eq37a}\\
\Gamma^V(w^{(i)},z)v=\Delta^V(w^{(i)},z)v
\end{gather}
\end{subequations}
for all $w^{(i)}\in W_i,v\in V$.

\begin{rem}\label{lb58}
We have the following smeared version of \eqref{eq35}: If $\wtd J\in\Jtd$ is clockwise to $\wtd I\in\Jtd$, and if $f\in C_c^\infty(\wtd I),g\in C_c^\infty(\wtd J)$, then the following diagram commutes adjointly (cf. \cite[Thm. 4.8]{Gui21a}):
\begin{equation}\label{eq38}
\begin{tikzcd}[column sep=huge]
\quad \mc H_k^\infty\quad \arrow[r,"{\Delta^V(w^{(j)},g)}"] \arrow[d, "{\Gamma^V(w^{(i)},f)}"'] &\quad (\mc H_k\boxdot\mc H_j)^\infty\quad \arrow[d, "{\Gamma^V(w^{(i)},f)}"]\\
(\mc H_i\boxdot\mc H_k)^\infty\arrow[r,"{\Delta^V(w^{(j)},g)}"] &(\mc H_i\boxdot\mc H_k\boxdot\mc H_j)^\infty
\end{tikzcd}
\end{equation}
provided that the four intertwining operators appeared in \eqref{eq38} are energy-bounded. Similarly, the smeared version of \eqref{eq34}, \eqref{eq36}, and \eqref{eq37} hold provided that the intertwining operators involved are energy-bounded: one simply replaces $z$ by $\wtd f\in C_c^\infty(I)$. 
\end{rem}

Recall the following definition:

\begin{df}
Let $W_i$ be a $V$-module. A vector $w\in W_i$ is called \textbf{quasi-primary} if $w$ is homogeneous (i.e. it is an eigenvector of $L_0$), and if $L_1w=0$. We say that a subset $M\subset W_i$ \textbf{generates} $W_i$ if the smallest subspace containing $M$ and invariant under $Y(v)_n$ (for all $v\in V,n\in\Zbb$) is $W_i$. 
\end{df}

\begin{eg}
The set of quasi-primary vectors of $W_i$ generate $W_i$. In fact, this is true when $W_i$ is irreducible, since the lowest weight vectors of $W_i$ are quasi-primary. So this is also true in general since $W_i$ is semi-simple.
\end{eg}

\begin{lm}\label{lb45}
Let $V$ be completely unitary. Let $W_i,W_j$ be unitary $V$-modules. Assume that the type $W_i\boxdot W_j\choose W_i~W_j$ intertwining operator $w^{(i)}\in W_i\mapsto \Gamma^V(\cdot,z)|_{W_j}$ is energy-bounded. Let $M_i\subset W_i$ be a subset of homogeneous vectors generating $W_i$. Then for each $\wtd I\in\Jtd$, the subspace
\begin{align}\label{eq39}
\Span\{\Gamma^V(w^{(i)},\wtd f)w^{(j)}:w^{(i)}\in M_i,w^{(j)}\in W_j\text{ is homogeneous, }\wtd f\in C_c^\infty(\wtd I)\}
\end{align}
is dense in $\mc H_i\boxdot\mc H_j$ (the Hilbert space completion of $W_i\boxdot W_j$, cf. Def. \ref{lb44}).
\end{lm}

In fact, we do not need complete unitarity in full power. Instead, one can only assume that $V$ is unitary, $C_2$-cofinite, rational, and that the type $W_i\boxdot W_j\choose W_i~W_j$ intertwining operator involved in Lem. \ref{lb45} is (unitary and) energy-bounded.

\begin{proof}
When $W_i$ is irreducible, and especially when $\Gamma^V$ satisfies the strong intertwining property, this lemma was proved in \cite[Prop. 4.12]{Gui21a}. However, we shall argue that the strong intertwining property is not necessary.

Our starting point is the standard fact that, for any fixed $z\in\Cbb\setminus\{0\}$ with chosen $\arg z$,  any vector $w\in W_i\boxdot W_j$ satisfying
\begin{align}
\bk{\Gamma^V(w^{(i)},z)w^{(j)}|w}=0\qquad\text{for all }w^{(i)}\in W_i,w^{(j)}\in W_j
\end{align}
is zero. (See \cite[Prop. A.3]{Gui19a} for details.) Therefore, any $w\in W_i\boxdot W_j$ satisfying
\begin{align}
\bk{\Gamma^V(w^{(i)},z)w^{(j)}|w}=0\qquad\text{for all }w^{(i)}\in M_i,w^{(j)}\in W_j
\end{align}
is zero. (To see this, it suffices to show that for each fixed $w\in W_i\boxdot W_j$, the subspace $\mc T$ of all $w^{(i)}\in W_i$ satisfying $\bk{\Gamma^V(w^{(i)},z)w^{(j)}|w}=0$ for all $w^{(j)}\in W_j$ is invariant under all $Y(v)_n$. Indeed, the Jacobi-identity for intertwining operators implies that the formal Laurent series $\bk{\Gamma^V(Y(v,\zeta-z)w^{(i)},z)w^{(j)}|w}$ (of $\zeta-z$) and $\bk{\Gamma^V(w^{(i)},z)Y(v,\zeta)w^{(j)}|w}$ (of $\zeta$) are the expansions of the same rational function. Thus, if $w^{(i)}\in\mc T$, then the latter series is zero, and hence the former is also zero, which implies $Y(v)_nw^{(i)}\in\mc T$.) It follows (by letting the test function $\wtd f$ converge to the $\delta$-function at any given point of $\Sbb^1$) that for any fixed $\wtd I$, any $w\in W_i\boxdot W_j$ orthogonal to \eqref{eq39} is zero.

Now, to show that \eqref{eq39} is dense, we let $\mc W$ be the subspace of all $\xi\in\mc H_i\boxdot\mc H_j$ orthogonal to \eqref{eq39}. If we can show that $\mc W$ is rotation invariant, i.e., invariant under $e^{\im t\ovl{L_0}}$ for all $t\in\Rbb$, then for every $\xi\in\mc W$, the image of $\xi$ under any spectral projection of $\ovl{L_0}$ (which must be an element of $W_i\boxdot W_j$) is inside $\mc W$, and hence is $0$ by the above paragraph. Then we conclude $\mc W=0$.

We now show that $\mc W$ is rotation invariant by a standard Reeh-Schlieder type argument. Choose any $\xi\in\mc W$. By Prop. \ref{lb51}, for every $\eta\in\eqref{eq39}$, there is $\delta>0$ such that $\bk{e^{\im t\ovl{L_0}}\xi|\eta}=0$ for all $t\in(-\delta,\delta)$. Since $\ovl{L_0}\geq0$,  the function $f(z)=\bk{e^{\im z\ovl{L_0}}\xi|\eta}$ is holomorphic on $\mbb H=\{z\in\Cbb:\Imag z>0\}$ and continuous on $\ovl{\Hbb}$. By the Schwarz reflection principle, $f$ can be extended to a holomorphic function on $\Cbb\setminus(\Rbb\setminus(-\delta,\delta))$, and hence $f=0$ on that domain since $f=0$ on $(-\delta,\delta)$. Thus $f=0$ on $\ovl{\Hbb}$ by the continuity, and hence $\bk{e^{\im t\ovl{L_0}}\xi|\eta}=0$ for all $t\in\Rbb$. 
\end{proof}

\begin{lm}\label{lb55}
Let $V$ be energy-bounded. Let $I\in\mc J$. Then
\begin{align}
\begin{aligned}\label{eq40}
\mc W_I=\Span\{&Y(v_1,f_1)\cdots Y(v_n,f_n)\Omega:n\in\Zbb_+,\text{and }v_1,\dots,v_n\in V\text{ are quasi-primary, }\\
&I_1,\dots,I_n\in\mc J\text{ are disjoint subintervals of $I$, }f_i\in C_c^\infty(I_i)\text{ for all }i\}
\end{aligned}
  \end{align}
and
\begin{align}\label{eq44}
\mc V_I=\Span\{Y(v,f)\Omega:v\in V\text{ is homogeneous, }f\in C_c^\infty(I) \}
\end{align}
are dense in $\mc H_0$.
\end{lm}

\begin{proof}
We prove the density of \eqref{eq40}; the second one is similar. Our proof is similar to that of \cite[Thm. 8.1]{CKLW18}: As in the proof of Thm. \ref{lb45}, by the Schwarz reflection principle  and Prop. \ref{lb51}, it suffices to prove that any $v\in V$ orthogonal to \eqref{eq40} is zero.  In fact, instead of assuming $v\in V$, it suffices to assume $v\in \mc H_0^\infty$.

So let $v\in \mc H_0^\infty$ be orthogonal to \eqref{eq40}. Choose any $n\in\Zbb_+$ and any quasi-primary $v_1,\dots,v_n\in V$. Another application of the Schwarz reflection principle (applied to $\bk{e^{\im z\ovl{L_0}}Y(v_{n-1},f_{n-1})^*\cdots Y(v_1,f_1)^*v|Y(v_n,f_n)\Omega}$) shows that $v$ is orthogonal to $Y(v_1,f_1)\cdots Y(v_n,f_n)\Omega$ where $f_1,\dots,f_{n-1}$ are supported in mutually disjoint subintervals of $I$, and $f_n$ is supported in any element of $\mc J$. By linearity, $v$ is orthogonal to $Y(v_1,f_1)\cdots Y(v_{n-1},f_{n-1})Y(v_n)_{k_n}\Omega$ where $f_1,\dots,f_{n-1}$ are supported in mutually disjoint subintervals of $I$ and $k_n\in\Zbb$. Applying the same argument repeatedly, we see that $v$ is orthogonal to $Y(v_1)_{k_1}\cdots Y(v_n)_{k_n}\Omega$ for all $k_1,\dots,k_n\in\Zbb$. Since this is true for all $n$, and since the vectors of the form $Y(v_1)_{k_1}\cdots Y(v_n)_{k_n}\Omega$ span $V$ (cf. \cite[Prop. 6.6]{CKLW18}), we conclude that $v$ is orthogonal to $V$, and hence $v=0$.
\end{proof}

\begin{df}\label{lb82}
Given arg-valued smooth functions $\wtd f,\wtd g$, we say that $\wtd f$ is \textbf{anticlockwise} to $\wtd g$ (equivalently, $\wtd g$ is \textbf{clockwise} to $\wtd f$) if there exist $\wtd I,\wtd J\in\Jtd$ with $\wtd J$ clockwise to $\wtd I$ such that $\wtd f\in C_c^\infty(\wtd I),\wtd g\in C_c^\infty (\wtd J)$. 
\end{df}

Recall from Def. \ref{lb44} that the braided $C^*$-tensor category $(\scr C,\boxdot,\ss)$ is
\begin{align*}
(\Rep^V(\mc A_V),\boxdot,\ss):=(\fk F^V_\CWX,\id)(\RepV,\boxdot,\ss)
\end{align*}

\begin{thm}\label{lb47}
Suppose that $V$ satisfies Condition II in Def. \ref{lb34}. Let $\mc F^V$ be as in Condition II. For each $W_i\in\Obj(\RepV)$ and $\wtd I\in\Jtd$, let
\begin{subequations}\label{eq45}
\begin{gather}
\fk H_i(\wtd I)_{i_1,\dots,i_n}=\Hom_V(W_{i_1}\boxdot\cdots\boxdot W_{i_n},W_i)\times W_{i_1}\times\cdots\times W_{i_n}\times C_c^\infty(\wtd I)^n\label{eq45a}\\
\fk H_i(\wtd I)=\bigsqcup_{
\begin{subarray}{c}
n\in\Zbb_+\\
W_{i_1},\dots,W_{i_n}\in \mc F^V\cup\{V\}
\end{subarray}
} \fk H_i(\wtd I)_{i_1,\dots,i_n}
\end{gather}
\end{subequations}
For each $\fk a\in\fk H_i(\wtd I)_{i_1,\dots,i_n}$ where
\begin{align}
\fk a=(T,w^{(i_1)},\dots,w^{(i_n)},\wtd f_1,\dots,\wtd f_n)
\end{align}
we set for each $W_j\in\Obj(\RepV)$ that
\begin{subequations}\label{eq46}
\begin{gather}
\mc L(\fk a,\wtd I)\big|_{\mc H_j^\infty}=(T\boxdot\idt_j)\circ \Gamma^V(w^{(i_1)},\wtd f_1)\cdots \Gamma^V(w^{(i_n)},\wtd f_n)\big|_{\mc H_j^\infty}\label{eq46a}\\
\mc R(\fk a,\wtd I)\big|_{\mc H_j^\infty}=(\idt_j\boxdot T)\circ \Delta^V(w^{(i_n)},\wtd f_n)\cdots \Delta^V(w^{(i_1)},\wtd f_1)\big|_{\mc H_j^\infty}
\end{gather}
if $w^{(i_1)},\dots,w^{(i_n)}$ are quasi-primary and $\wtd f_{\nu+1}$ is clockwise to  $\wtd f_\nu$ for each $\nu=1,\dots,n-1$; otherwise, we set $\mc L(\fk a,\wtd I)=0$ and $\mc R(\fk a,\wtd I)=0$. Then these data give a M\"obius covariant weak categorical extension
\begin{gather}\label{eq47}
\scr E^w=(\mc A_V,\Rep^V(\mc A_V),\boxdot_V,\ss^V,\fk H)\equiv(\mc A_V,\Rep^V(\mc A_V),\boxdot,\ss,\fk H)
\end{gather}
\end{subequations}
\end{thm}

\begin{df}\label{lb62}
The closure of $\scr E^w$ in Thm. \ref{lb47} is denoted by
\begin{align}
\scr E_V=(\mc A_V,\Rep^V(\mc A_V),\boxdot,\ss,\mc H)
\end{align}
and is called the \textbf{(closed) categorical extension associated to $V$}. The tensorator
\begin{align*}
\fk W^V=\fk W^V_{i,j}:\mc H_i\boxdot_V\mc H_j\rightarrow\mc H_i\boxtimes_{\mc A_V}\mc H_j
\end{align*}
associated to $\scr E_V$ (cf. Thm. \ref{lb1}) is called the \textbf{Wassermann tensorator}. So it is determined by
\begin{align}
\fk W^V\circ\scr E_V=\scr E_{\mc A_V,\Connes}\big|_{\Rep^V(\mc A_V)}
\end{align}
where $\scr E_{\mc A_V,\Connes}$ is the Connes categorical extension for $\mc A_V$.
\end{df}

$\scr E_V$ and $\fk W^V$ are independent of the choice of $\mc F^V$. More precisely:

\begin{thm}\label{lb60}
Let $V$ satisfy Condition II. Suppose that we have another set $\wtd{\mc F}^V$ of unitary $V$-modules satisfying the assumptions in Condition II (cf. Def. \ref{lb34}). Then the categorical extension $\scr E_V$ defined by $\mc F^V$ is equal to the one defined by $\wtd{\mc F}^V$. Therefore, the Wassermann tensorator $\fk W^V$ defined by $\mc F^V$ is equal to the one defined by $\wtd{\mc F}^V$.
\end{thm}

\begin{proof}[\textbf{Proof of Thm. \ref{lb47}}]
This theorem is similar to \cite[Thm. 3.22]{Gui26}, except that the construction of $\scr E^w$ is slightly simplified. We need to check all the axioms in Def. \ref{lb8}. To begin with, note that the $\mc L$ and $\mc R$ operators are smooth and localizable by Rem. \ref{lb54} (and \ref{lb25}). Isonoty is obvious. Naturality: By \eqref{eq36}. Neutrality: By \eqref{eq37}.

By the density Lem. \ref{lb45}, the rotation covariance of smeared intertwining operators (cf. Prop. \ref{lb51}), and the fact that smeared intertwining operators are localizable (Rem. \ref{lb54}), one shows by induction on $n$ that for any $W_{i_1},\dots,W_{i_n}\in\mc F^V$,  
\begin{align*}
\Span\{&\Gamma^V(w^{(i_1)},\wtd f_1)\cdots \Gamma^V(w^{(i_n)},\wtd f_n)\mc H_j^\infty:w^{(i_1)}\in W_{i_1},\dots,w^{(i_n)}\in W_{i_n}\\
&\wtd f_1,\dots,\wtd f_n\in C^\infty_c(\wtd I),\wtd f_{\nu+1}\text{ is clockwise to }\wtd f_\nu\}
\end{align*}
is a QRI subspace of $(\mc H_{i_1}\boxdot\cdots\boxdot\mc H_{i_n}\boxdot\mc H_i)^\infty$. (In particular, it is dense.) A similar property holds if $\Gamma^V$ is replaced by $\Delta^V$. This implies the density of fusion products.

Choose $\wtd J,\wtd K\subset \wtd I$ such that $\wtd K$ is clockwise to $\wtd J$. We use the notation in Lem. \ref{lb55}. Then by the rotation covariance, $\mc W_K$ is a QRI subspace of $\mc H_0^\infty$. Therefore, since products of smeared intertwining operators are localizable (Rem. \ref{lb54}), 
\begin{align*}
\Span\{&\Gamma^V(w^{(i_1)},\wtd f_1)\cdots \Gamma^V(w^{(i_n)},\wtd f_n)\cdot \mc W_K:w^{(i_1)}\in W_{i_1},\dots,w^{(i_n)}\in W_{i_n}\\
&\wtd f_1,\dots,\wtd f_n\in C^\infty_c(\wtd J),\wtd f_{\nu+1}\text{ is clockwise to }\wtd f_\nu\}
\end{align*}
is a dense subspace of $\mc H_{i_1}\boxdot\cdots\boxdot\mc H_{i_n}$. A similar property holds when $\Gamma^V$ is replaced by $\Delta^V$. Therefore, since the last few of $W_{i_1},\dots,W_{i_n}$ in \eqref{eq45a} can be chosen to be $V$, we conclude the Reeh-Schlieder property.

Intertwining property: By Cor. \ref{lb57} and Lem. \ref{lb35}.

Weak locality: This follows from Rem. \ref{lb58} and \eqref{eq37}. See the proof of \cite[Thm. 2.44]{Gui26} (especially the commutative diagrams therein) for details.

Braiding: This follows from the naturality of $\ss$ (which implies $\ss\circ(T\boxdot\idt)=(\idt\boxdot T)\circ\ss$) and the following Lem. \ref{lb59}.

M\"obius covariance: By Rem. \ref{lb56} and Prop. \ref{lb51}.
\end{proof}

\begin{proof}[\textbf{Proof of Thm. \ref{lb60}}]
Let $\mc G^V=\mc F^V\cup\wtd{\mc F}^V$. Let $\scr E^w_{\mc F^V}$ and $\scr E^w_{\mc G^V}$ be the weak categorical extensions defined by $\mc F^V$ and $\mc G^V$ respectively as in \eqref{eq47}, and let $\scr E_{\mc F^V}$ and $\scr E_{\mc G^V}$ denote their closures. Then $\scr E^w_{\mc F^V}\subset \scr E^w_{\mc G^V}$, i.e., every $\mc L$ resp. $\mc R$ operator of $\scr E^w_{\mc F^V}$ is a $\mc L$ resp. $\mc R$ operator of $\scr E^w_{\mc G^V}$, and hence its closure is a left resp. right operator of $\scr E_{\mc G^V}$. Therefore, $\scr E_{\mc G^V}$ is a (and hence \textit{the}) closure of $\scr E^w_{\mc F^V}$, cf. Thm. \ref{lb16}. So $\scr E_{\mc F^V}=\scr E_{\mc G^V}$. Similarly, $\scr E_{\wtd{\mc F}^V}=\scr E_{\mc G^V}$.
\end{proof}

\begin{lm}\label{lb59}
Let $V$ and $\mc F^V$ be as in Thm. \ref{lb47}. Let $W_{i_1},\dots,W_{i_n}\in\mc F^V$. Let $w^{(i_1)}\in W_{i_1},\dots,w^{(i_n)}\in W_{i_n}$ be homogeneous. Let $\wtd I\in\Jtd$ and $\wtd f_1,\dots,\wtd f_n\in C_c^\infty(I)$ such that $\wtd f_k$ is anticlockwise to $\wtd f_{k+1}$ for each $k=1,\dots,n-1$. Then for each $W_j\in\Obj(\RepV)$ we have
\begin{align}
\Delta^V(w^{(i_n)},\wtd f_n)\cdots \Delta^V(w^{(i_1)},\wtd f_1)\big|_{\mc H_j^\infty}=\ss_{i_1\boxdot\cdots\boxdot i_n,j}\Gamma^V(w^{(i_1)},\wtd f_1)\cdots \Gamma^V(w^{(i_n)},\wtd f_n)\big|_{\mc H_j^\infty}
\end{align}
\end{lm}

\begin{proof}
This follows from \eqref{eq36}, Rem. \ref{lb58}, and the Hexagon axiom for $\ss$. See (the proof of) \cite[Prop. 2.43]{Gui26} for details.
\end{proof}

When $V$ satisfies Condition I, its categorical extension can be generated by a simpler (though not M\"obius covariant) weak categorical extension:

\begin{thm}
Suppose that $V$ satisfies Condition I in Def. \ref{lb34}. For each $W_i\in\Obj(\RepV)$ and $\wtd I\in\Jtd$, let
\begin{gather}
\fk H_i(\wtd I)=W_i\times V\times C_c^\infty(I)\times C_c^\infty(I)
\end{gather}
For each $\fk a\in\fk H_i(\wtd I)$ where
\begin{align*}
\fk a=(w^{(i)},v,f,g)
\end{align*}
we set for each $W_j\in\Obj(\RepV)$ that
\begin{align}\label{eq48}
\mc L(\fk a,\wtd I)|_{\mc H_j^\infty}=\Gamma^V(w^{(i)},f)Y_j(v,g)\big|_{\mc H_j^\infty}\qquad \mc R(\fk a,\wtd I)|_{\mc H_j^\infty}=\Delta^V(w^{(i)},f)Y_j(v,g)\big|_{\mc H_j^\infty}
\end{align}
if $w^{(i)},v$ are homogeneous; otherwise,  we set $\mc L(\fk a,\wtd I)=0$ and $\mc R(\fk a,\wtd I)=0$. Then these data give a weak categorical extension $\scr E^w$. Moreover, the closure of $\scr E^w$ is equal to $\scr E_V$, the categorical extension associated to $V$ (cf. Def. \ref{lb62}).
\end{thm}

\begin{proof}
That \eqref{eq48} defines a weak categorical extension $\scr E^w$ can be proved in the same way as Thm. \ref{lb60}, except that one uses the density of \eqref{eq44} instead of \eqref{eq40} to conclude the Reeh-Schlieder property.

Let $\wtd{\scr E}^w$ be the M\"obius covariant weak categorical extension defined in Thm. \ref{lb47} by any $\mc F^V$. (For example, we let $\mc F^V$ contain every irreducible module up to unitary isomorphisms.) By definition, its closure is $\scr E_V$. A similar algebraic computation shows that the $\mc L(\fk a,\wtd I)$ and $\mc R(\fk a,\wtd I)$ in \eqref{eq48} are respectively weak left and right operators of $\wtd{\scr E}^w$ (cf. Def. \ref{lb9} and \ref{lb10}). Thus, by Thm. \ref{lb13}, $\mc L(\fk a,\wtd I)$ and $\mc R(\fk a,\wtd I)$ left and right operators of $\scr E_V$. So $\scr E_V$ is a (and hence the) closure of $\scr E^w$. This finishes the proof.
\end{proof}
\begin{rem}
Alternatively, one can show that $\scr E_V$ is the closure of $\scr E^w$ in the following way without using Thm. \ref{lb13}. It is easy to check that $\scr E^w\cup\wtd{\scr E}^w$ is a weak categorical extension. So its closure is also a closure (and hence the closure) of both $\scr E^w$ and $\wtd{\scr E}^w$. Thus,  $\scr E^w$ and $\wtd{\scr E}^w$ have the same closure.
\end{rem}

\section{(Weak) categorical extensions associated to unitary VOA extensions}

\subsection{The unitary VOA extension $U_P$ associated to the Haploid commutative $C^*$-Frobenius algebra $P$}\label{lb107}

We fix a unitary VOA $V$ with vertex operation $Y^V=Y$.

\begin{df}\label{lb63}
A \textbf{(normalized CFT-type) unitary VOA extension} is a triple $(U,Y^U,\iota)$ (or simply $U$), where $(U,Y^U)$ is a (CFT-type) unitary VOA, the inner product space $U$ is equipped with a unitary $V$-module structure
\begin{align*}
U=W_a\text{ where }(W_a,Y_a)\in\Obj(\RepV)
\end{align*}
Moreover, $\iota\in\Hom_V(V,W_a)=\Hom_V(V,U)$ satisfies $\iota^*\iota=\idt_{V}$, and the following conditions are satisfies:
\begin{enumerate}[label=(\arabic*)]
\item For each $v\in V$, we have $Y_a(v)_n=Y^U(\iota v)_n$ for all $n\in\Zbb_+$, i.e.
\begin{align*}
Y_a(v,z)=Y^U(\iota v,z)
\end{align*}
\item $\iota$ sends the vacuum vector of $V$ to that of $U$.
\item $\iota$ sends the conformal vector of $V$ to that of $U$.
\end{enumerate}
\end{df}

\begin{rem}
In Def. \ref{lb63}, for each $v_1,v_2\in V$ we have
\begin{align*}
Y^U(\iota v_1,z)\iota v_2=Y_a(v_1,z)\iota v_2=\iota Y(v_1,z)v_2
\end{align*}
Therefore, identifying $V$ with $\iota(V)$ via $V$, a unitary VOA extension is equivalently a unitary VOA $U$ whose underlying inner product space contains $V$, and whose vertex operation restricts to that of $V$, whose vacuum vectors and conformal vectors are equal to those of $V$.
\end{rem}

\begin{thm}\label{lb75}
Assume that $V$ is completely unitary. Then there is a one-to-one correspondence between a haploid commutative $C^*$-Frobenius algebra $P=(W_a,\mu,\iota)$ in $\RepV$ and a unitary VOA extension $(U_P,Y^{U_P},\iota)$ of $V$. $Y^{U_P}$ is determined by
\begin{align}\label{eq49}
Y^{U_P}(u_1,z)u_2=\mu\circ\Gamma^V(u_1,z)u_2
\end{align}
for all $u_1,u_2\in U_P$.
\end{thm}
Recall that $\mu\in\Hom_V(W_a\boxtimes_V W_a,W_a)$ and $\Gamma^V$ is a type $W_a\boxtimes_V W_a\choose W_a~W_a$ intertwining operator of $V$.

\begin{proof}
The (injective) map $P\mapsto U_P$, where $U_P$ is a (not necessarily unitary) VOA extension of $V$ is due to \cite[Thm. 3.2]{HKL15}. That $U_P$ is of CFT-type is due to Li's classification of vacuum-like vectors \cite[Prop. 3.4]{Li94}; see Thm. 4.5 or Prop. 11.2 of \cite{Gui24} for a detailed explanation. That $U_P$ is a unitary extension, as well as the surjectivity of $P\mapsto U_P$, is \cite[Thm. 2.21]{Gui22}.
\end{proof}

\begin{rem}
In fact, every haploid algebra in $\RepV$ has a unique unitary (i.e. $C^*$) structure. It follows that every (CFT-type) VOA extension of $V$ is unitary  in a unique way.  See \cite[Thm. 4.7]{CGGH23} for details.
\end{rem}

\subsection{The functor $\fk F_\VOA:\Rep^0(P)\rightarrow\RepUP$ and the tensorator $\fk V^\boxdot:W_i\boxdot_P W_j\rightarrow W_i\boxdot_{U_P}W_j$}\label{lb108}

Let $V$ be completely unitary. Recall that
\begin{align*}
(\RepV,\boxdot,\ss)\equiv(\RepV,\boxdot_V,\ss^V)
\end{align*}
is the Huang-Lepowsky braided $C^*$-tensor category (Thm. \ref{lb61}). Let $P$ be a haploid $C^*$-Frobenius algebra in $\RepV$ and let $U_P$ be its associated unitary VOA extension. Then $U_P$ is completely unitary \cite[Thm. 3.30]{Gui22}. Then we also have the Huang-Lepowsky braided $C^*$-tensor category of unitary $U_P$-modules
\begin{align}
(\RepUP,\boxdot_{U_P},\ss^{U_P})
\end{align}
As in Subsec. \ref{lb65}, for each $W_i,W_j\in\RepUP$ we have $U_P$-intertwining operators $\Gamma^{U_P}$ and $\Delta^{U_P}$ of $U_P$ of types $W_i\boxdot_{U_P} W_j\choose W_i~W_j$ and $W_j\boxdot_{U_P} W_i\choose W_i~W_j$ respectively related by
\begin{align*}
\Delta^{U_P}(w^{(i)},z)w^{(j)}=\ss_{i,j}^{U_P}\Gamma^{U_P}(w^{(i)},z)w^{(j)}
\end{align*}
for all $w^{(i)}\in W_i,w^{(j)}\in W_j$. 

Let
\begin{align*}
\Rep^0(P)=\Rep^0_{\RepV}(P)
\end{align*}
be the braided $C^*$-category of dyslectic unitary $P$-modules as in Def. \ref{lb66}. Recall from Thm. \ref{lb19} that any system ($\boxdot_P,\mu_{\blt,\star}$) of fusion products of dyslectic $P$-modules, where
\begin{align*}
\mu_{i,j}\in\Hom_V(W_i\boxdot W_j,W_i\boxdot_PW_j)
\end{align*}
for each dyslectic modules $W_i,W_j$, determines a braided $C^*$-tensor structure on $\Rep^0(P)$.

\begin{thm}\label{lb67}
Let $V$ be completely unitary. Let $P$ be a haploid commutative $C^*$-Frobenius algebra in $\RepV$. Then $U_P$ is completely unitary.

Moreover, choose a system of fusion products $(\boxdot_P,\mu_{\blt,\star})$ in $\Rep^0(P)$ giving a braided $C^*$-tensor category
\begin{align*}
(\Rep^0(P),\boxdot_P,\ss^P)
\end{align*}
as in Thm. \ref{lb19} (satisfying the four conditions as in \eqref{eq50}). Then we have a braided $*$-functor $(\fk F_\VOA,\fk V^\boxdot)$ impletementing an isomorphism of braided $C^*$-tensor categories
\begin{align*}
(\fk F_\VOA,\fk V^\boxdot):\big(\Rep^0(P),\boxdot_P,\ss^P \big)\xlongrightarrow{\simeq}\big(\RepUP,\boxdot_{U_P},\ss^{U_P}\big)
\end{align*}
Here,
\begin{gather}\label{eq54}
\begin{gathered}
\fk F_\VOA:\Rep^0(P)\rightarrow\RepUP\\
(W_i,\mu^i)\in\Obj(\Rep^0(P))\quad\mapsto\quad (W_i,Y^{U_P}_i)\in\Obj(\RepUP)\\
F\in\Hom_P(W_i,W_j)\quad\mapsto\quad F\in\Hom_{U_P}(W_i,W_j)
\end{gathered}
\end{gather}
where $Y_i^{U_P}$ is related to $\mu^i\in\Hom_V(U_P\boxdot W_i,W_i)$ by the fact that for all $u\in U_P,w^{(i)}\in W_i$,
\begin{align}
Y_i^{U_P}(u,z)w^{(i)}=\mu^i\circ\Gamma^V(u,z)w^{(i)}
\end{align}
and the natural map $\fk V^{\boxdot}$ associates to each objects $W_i,W_j$ of $\Rep^0(P)$ a unitary $\fk V^\boxdot_{i,j}\in\Hom_{U_P}(W_i\boxdot_P W_j, W_i\boxdot_{U_P}W_j)$ determined by
\begin{gather}\label{eq51}
\fk V^\boxdot_{i,j}\circ\mu_{i,j}\Gamma^V(w^{(i)},z) w^{(j)}=\Gamma^{U_P}(w^{(i)},z)w^{(j)}
\end{gather}
\end{thm}

Note that in this theorem we also have
\begin{gather}\label{eq52}
\fk V^\boxdot_{j,i}\circ\mu_{j,i}\Delta^V(w^{(i)},z) w^{(j)}=\Delta^{U_P}(w^{(i)},z)w^{(j)}
\end{gather}
since
\begin{align*}
&\Delta^{U_P}(w^{(i)},z)w^{(j)}=\ss^{U_P}_{i,j}\Gamma^{U_P}(w^{(i)},z)w^{(j)}=\ss^{U_P}_{i,j}\fk V^\boxdot_{i,j}\mu_{i,j}\Gamma^V(w^{(i)},z) w^{(j)}\\
=&\fk V^\boxdot_{j,i}\ss^P_{i,j}\mu_{i,j}\Gamma^V(w^{(i)},z) w^{(j)}=\fk V^\boxdot_{j,i}\mu_{j,i}\ss_{i,j}\Gamma^V(w^{(i)},z) w^{(j)}=\fk V^\boxdot_{j,i}\mu_{j,i}\Delta^V(w^{(i)},z) w^{(j)}
\end{align*}
where $\ss^{U_P}_{i,j}\fk V^\boxdot_{i,j}=\fk V^\boxdot_{j,i}\ss^P_{i,j}$ is because $\fk V^\boxdot$ intertwines the braidings $\ss^P_{i,j}$ and $\ss^{U_P}_{i,j}$ (as required by braided $*$-functors), and $\ss^P_{i,j}\mu_{i,j}=\mu_{j,i}\ss_{i,j}$ is due to \eqref{eq20}. Thus, in this way, \eqref{eq51} and \eqref{eq52} are parallel to \eqref{eq32}.

\begin{rem}
In Thm. \ref{lb67}, $\Hom_{U_P}(W_i\boxdot_P W_j, W_i\boxdot_{U_P}W_j)$ is the abbreviation of $\Hom_{U_P}(\fk F_\VOA(W_i\boxdot_P W_j), \fk F_\VOA(W_i)\boxdot_{U_P}\fk F_\VOA(W_j))$. Alternatively, $\Hom_{U_P}(W_i\boxdot_P W_j, W_i\boxdot_{U_P}W_j)$ can be understood by viewing $W_i,W_j,W_i\boxdot_P W_j$ as an objects 
\begin{align}
\big(\RepUP,\boxdot_P,\ss^P\big):=(\fk F_\VOA,\id)\big(\Rep^0(P),\boxdot_P,\ss^P \big)
\end{align}
This is similar to our perspective in Cor. \ref{lb41}.
\end{rem}


\begin{proof}[\textbf{Proof of Thm. \ref{lb67}}]
By \cite[Thm. 3.30]{Gui22}, $V$ is completely unitary. By Sec. 2.5 and especially Thm. 2.30 of \cite{Gui22}, $\fk F_\VOA$ implements an isomorphism of $C^*$-categories. To simplify the following discussion, we identify the $C^*$-categories $\Rep^0(P)$ and $\RepUP$ via $\fk F_\VOA$.  By \eqref{eq53}, for each $W_i,W_j\in\Obj(\RepUP)$, there is a unique morphism $\nu_{i,j}\in\Hom_V(W_i\boxdot W_j,W_i\boxdot_{U_P}W_j)$ such that
\begin{align*}
\nu_{i,j}\circ\Gamma^V(w^{(i)},z) w^{(j)}=\Gamma^{U_P}(w^{(i)},z)w^{(j)}
\end{align*}
for all $w^{(i)}\in W_i,w^{(j)}\in W_j$. By \cite[Thm. 3.29]{Gui22}, $(\boxdot_{U_P},\nu_{\blt,\star})$ is a system of fusion products in $\Rep^0(P)$. By \cite[Thm. 3.30]{Gui22}, the braided $C^*$-tensor structure on $\RepUP$ associated to $(\boxdot_{U_P},\nu_{\blt,\star})$ (in the sense of Thm. \ref{lb19}) is equal to $(\RepUP,\boxdot_{U_P},\ss^{U_P})$, the Huang-Lepowsky braided $C^*$-tensor category. Let
\begin{align*}
\fk V_{\blt,\star}^\boxdot:W_\blt\boxdot_P W_\star\rightarrow W_\blt\boxdot_{U_P}W_\star
\end{align*}
be the (unitary) linking map between the two systems, cf. Rem. \ref{lb17}. Namely, $\nu_{\blt,\star}=\fk V^\boxdot\circ\mu_{\blt,\star}$. So \eqref{eq51} is satisfied. By Rem. \ref{lb39}, $(\id,\fk V^\boxdot)$ implements an isomorphism of braided $C^*$-tensor categories $\big(\Rep^0(P),\boxdot_P,\ss^P \big)\xlongrightarrow{\simeq}\big(\RepUP,\boxdot_{U_P},\ss^{U_P}\big)$.
\end{proof}

\begin{rem}
The non-unitary version of Thm. \ref{lb67} is due to \cite{HKL15} (without addressing the isomorphism of braided tensor structures) and \cite{CKM24}.
\end{rem}

The following two remarks are  parallel to Rem. \ref{lb68}.

\begin{rem}\label{lb73}
The following fact proved in \cite[Thm. 3.29]{Gui22} and used in the proof of Thm. \ref{lb67} is worth noting: Identifying the $C^*$-categories $\Rep^0(P)$ and $\RepUP$ via $\fk F_\VOA$, there is a system of fusion products in $\Rep^0(P)$ such that $\boxdot_P=\boxdot_{U_P}$ and $\fk V^\boxdot=\id$, i.e., the one $(\boxdot_{U_P},\mu_{\blt,\star})$ satisfying
\begin{align}\label{eq60}
\mu_{i,j}\Gamma^V(w^{(i)},z)w^{(j)}=\Gamma^{U_P}(w^{(i)},z)w^{(j)}
\end{align}
for all $W_i,W_j\in\Obj^0(P)$.
\end{rem}

\begin{rem}
Suppose that $(\wht\boxdot_P,\wht\mu_{\blt,\star})$ is another system of fusion products in $\Rep^0(P)$, and that the natural unitary $\wht{\fk V}^\boxdot:W_\blt\wht\boxdot_P W_\star\rightarrow W_\blt\boxdot_{U_P}W_\star$ is defined for this system similar to $\fk V^\boxdot$, i.e.
\begin{align*}
\wht{\fk V}^\boxdot_{i,j}\circ\wht\mu_{i,j}\Gamma^V(w^{(i)},z) w^{(j)}=\Gamma^{U_P}(w^{(i)},z)w^{(j)}
\end{align*}
for all $W_i,W_j\in\Obj(\Rep^0(P))$. From the proof of Thm. \ref{lb67}, it is clear that
\begin{align}\label{eq65}
\wht{\fk V}_{i,j}^\boxdot\circ\Phi_{i,j}=\fk V_{i,j}^\boxdot
\end{align}
where $\Phi:W_\blt\boxdot_P W_\star\rightarrow W_\blt\wht\boxdot_P W_\star$ is the connecting map between the two systems of fusion products (Rem. \ref{lb17}). 
\end{rem}

\subsection{Main theorem on smeared operators and (weak) categorical extensions associated to $Q=(\fk F^V_\CWX,\id)(P)$}

\subsubsection{The setting}\label{lb74}
Let $V$ satisfy Condition II in Def. \ref{lb34}. Let $\mc A_V$ be the CKLW net of $V$. Recall that $(\RepV,\boxdot,\ss)$ is the Huang-Lepowsky braided $C^*$-tensor category of unitary $V$-modules (Thm. \ref{lb61}) where, as usual, we have abbreviated $\boxdot_V$ to $\boxdot$ and $\ss^V$ to $\ss$. Recall from Def. \ref{lb44} that $(\scr C,\boxdot,\ss)$ is chosen to be
\begin{align*}
\big(\Rep^V(\mc A_V),\boxdot,\ss\big):=(\fk F^V_\CWX,\id)\big(\RepV,\boxdot,\ss\big)
\end{align*}
Recall from Thm. \ref{lb42} that $\Rep^V(\mc A_V)$ is a full replete $C^*$-subcategory of $\Rep(\mc A_V)$. As usual, for each $W_i\in\Obj(\RepV)$ we write $\mc H_i=\fk F_\CWX^V(W_i)$. Let
\begin{align*}
\scr E_V=(\mc A_V,\Rep^V(\mc A_V),\boxdot,\ss,\mc H)\qquad\text{with L, R operators }L^\boxdot,R^\boxdot
\end{align*}
be the categorical extension associated $V$ (Def. \ref{lb62}).

Let $P=(W_a,\mu,\iota)$ be a $C^*$-Frobenius algebra in $\RepV$. Let
\begin{align*}
Q=(\fk F^V_\CWX,\id)(P)=(\mc H_a,\mu,\iota)
\end{align*}
be the pushforward $C^*$-Frobenius algebra in $\Rep^V(\mc A_V)$. Thus, by Thm. \ref{lb26}, $Q$ gives a conformal net extension $\mc B_Q$ of $\mc A_V$ determined by
\begin{align}\label{eq77}
\mc B_Q(I)=\{\mu L^\boxdot(\xi,\wtd I)|_{\mc H_a}:\xi\in\mc H_a(I)\}
\end{align}

Recall that $\Rep^0(P)\equiv\Rep^0_\RepV(P)$ and $\Rep^0_{\Rep^V(\mc A_V)}(Q)$ are respectively the $C^*$-categories of dyslectic (unitary) $P$-modules (in $\RepV$) and $Q$-modules in $\Rep^V(\mc A_V)$. Fix a system of fusion products $(\boxdot_P,\mu_{\blt,\star})$ in $\Rep^0(P)$. Its pushforward system in $\Rep^0_{\Rep^V(\mc A_V)}(Q)$ via $(\fk F_\CWX^V,\id)$ is denoted by $(\boxdot_Q,\mu_{\blt,\star})$, i.e.
\begin{align*}
(\fk F_\CWX^V,\id)\big(\boxdot_P,\mu_{\blt,\star}\big)=\big(\boxdot_Q,\mu_{\blt,\star}\big)
\end{align*}
Then these two systems give braided $C^*$-tensor category structures on $\Rep^0(P)$ and $\Rep^0_{\Rep^V(\mc A_V)}(Q)$ (cf. Thm. \ref{lb19}) which are canonically isomorphic:
\begin{align}
(\fk F^V_\CWX,\id):\big(\Rep^0(P),\boxdot_P,\ss^P\big)\xlongrightarrow{\simeq} \big(\Rep^0_{\Rep^V(\mc A_V)}(Q),\boxdot_Q,\ss^Q\big)
\end{align}

Recall from Def. \ref{lb100} that $\Rep_{\Rep^V(\mc A_V)}(\mc B_Q)$ is the $C^*$-category of $\mc B_Q$-modules whose restrictions to $\mc A$ are objects in $\Rep^V(\mc A_V)$. By Recall from Def. \ref{lb70} that
\begin{align}\label{eq68}
\big(\Rep_{\Rep^V(\mc A_V)}(\mc B_Q),\boxdot_Q,\ss^Q\big)=(\fk F_\CN,\id)\big(\Rep^0_{\Rep^V(\mc A_V)}(Q),\boxdot_Q,\ss^Q\big)
\end{align}
where $\fk F_\CN$ is described in Thm. \ref{lb27}.

Let $(\RepUP,\boxdot_{U_P},\ss^{U_P})$ be the Huang-Lepowsky braided $C^*$-tensor category of unitary $U_P$-modules as in Thm. \ref{lb61}. Let
\begin{align}
(\fk F_\VOA,\fk V^\boxdot):\big(\Rep^0(P),\boxdot_P,\ss^P \big)\xlongrightarrow{\simeq}\big(\RepUP,\boxdot_{U_P},\ss^{U_P}\big)
\end{align}
be as in Thm. \ref{lb67}.

\begin{cv}
For each $(W_i,\mu^i)\in\Obj(\Rep^0(P))$, when talking about the corresponding unitary $U_P$-module $(W_i,Y^{U_P}_i)\in\Obj(\RepUP)$, the corresponding dyslectic $Q$-module $(\mc H_i,\mu^i)\in\Obj(\Rep^0_{\Rep^V(\mc A_V)}(Q))$, and the corresponding $\mc B_Q$-module $(\mc H_i,\pi_i)\in\Obj(\Rep_{\Rep^V(\mc A_V)}(\mc B_Q))$, we understand that they are related by the functors
\begin{equation}\label{eq69}
\begin{tikzcd}[row sep=large,column sep=2cm]
\Rep^0(P) \arrow[r,"{\fk F^V_\CWX}","\simeq"'] \arrow[d,"{\fk F_\VOA}"',"\simeq"] & \Rep^0_{\Rep^V(\mc A_V)}(Q) \arrow[d,"{\fk F_\CN}","\simeq"'] \\
\RepUP          & \Rep_{\Rep^V(\mc A_V)}(\mc B_Q)          
\end{tikzcd}
\qquad
\begin{tikzcd}[sep=large]
(W_i,\mu^i) \arrow[r,mapsto]\arrow[d,mapsto] & (\mc H_i,\mu^i)\arrow[d,mapsto]\\
(W_i,Y^{U_P}_i)& (\mc H_i,\pi_i)
\end{tikzcd}
\end{equation}
\end{cv}

\subsubsection{The main theorem}

Assume the setting in Subsubsec. \ref{lb74}. Recall the intertwining operators $\Gamma^V,\Delta^V$ of $V$ and (similarly) $\Gamma^{U_P},\Delta^{U_P}$ of $U_P$ (cf. Subsec. \ref{lb65}).

\begin{thm}[\textbf{Main Theorem}]\label{lb72}
Assume that $V$ satisfies Condition II. Let
\begin{align*}
\scr E_Q=(\mc B_Q,\Rep_{\Rep^V(\mc A_V)}(\mc B_Q),\boxdot_Q,\ss^Q,\mc H):=\mu_{\blt,\star}\circ\scr E_V
\end{align*}
be the categorical extension associated to $Q$, $\scr E_V$, and $(\boxdot_Q,\mu_{\blt,\star})$ with $\mathrm L,\mathrm R$ operators $L^Q,R^Q$ (cf. Def. \ref{lb69}). 
Let $W_j\in\Obj(\RepUP)$. Assume that every unitary intertwining operator of $U_P$ with charge space $W_j$ is energy-bounded. Let $w^{(j)}\in W_j$ be homogeneous. Let $\wtd J\in\Jtd$, and $\wtd g\in C_c^\infty(\wtd J)$. For each $W_k\in\Obj(\RepUP)$, let
\begin{subequations}\label{eq56}
\begin{gather}
\fk L^Q(\yk,\wtd J)\big|_{\mc H_k}=(\fk V_{j,k}^\boxdot)^{-1}\circ\ovl{\Gamma^{U_P}(w^{(j)},\wtd g)}\big|_{\mc H_k}\\
\fk R^Q(\yk,\wtd J)\big|_{\mc H_k}=(\fk V_{k,j}^\boxdot)^{-1}\circ\ovl{\Delta^{U_P}(w^{(j)},\wtd g)}\big|_{\mc H_k}
\end{gather}
\end{subequations}
which are closed operators with dense domains in $\mc H_k$. Then $\fk L^Q(\yk,\wtd J)$ and $\fk R^Q(\yk,\wtd J)$ are respectively left and right operators of $\scr E_Q$.
\end{thm}

Note that by \eqref{eq51} and \eqref{eq52}, we can rewrite \eqref{eq56} as
\begin{gather}\label{eq57}
\fk L^Q(\yk,\wtd J)\big|_{\mc H_k}=\ovl{(\mu_{j,k}\Gamma^V)(w^{(j)},\wtd g)}\big|_{\mc H_k}\qquad
\fk R^Q(\yk,\wtd J)\big|_{\mc H_k}=\ovl{(\mu_{k,j}\Delta^V)(w^{(j)},\wtd g)}\big|_{\mc H_k}
\end{gather}
the smeared intertwining operators defined by the energy-bounded $U_P$-intertwining operators $\mu_{j,k}\Gamma^V$ and $\mu_{k,j}\Delta^V$.\footnote{Note that it is not assumed that $\Gamma^V$ and $\Delta^V$ are always energy bounded. So one cannot write $\mu_{j,k}\circ\Gamma^V(w^{(j)},\wtd g)|_{\mc H_k}$ since it is not known whether the $V$-intertwining operator $\Gamma^V(w^{(j)},z)$ can be smeared.}

\begin{rem}\label{lb76}
In Thm. \ref{lb72}, if $W_j$ is the vacuum $U_P$-module $W_a=U_P$, then under the identifications $U_P\boxdot_{U_P} W_k=W_k=W_k\boxdot_{U_P}U_P$ via the unitors, we have $\fk V_{a,k}^\boxdot=\idt_k$ (cf. \eqref{eq61}). Thus, by \eqref{eq37}, the two equations in \eqref{eq56} become
\begin{align}
\fk L^Q(\yk,\wtd J)\big|_{\mc H_k}=\ovl{Y^{U_P}_k(u,\wtd g)}=\fk R^Q(\yk,\wtd J)\big|_{\mc H_k}
\end{align}
\end{rem}

\begin{proof}[\textbf{Proof of Thm. \ref{lb72}}]
Step 1. We prove only that $\fk R^Q(\yk,\wtd J)$ is a right operator; the treatment of $\fk L^Q(\yk,\wtd J)$ is similar. Let $\mc F^V$ be as in Condition II. Let $\scr E^w=(\mc A_V,\Rep^V(\mc A_V),\boxdot,\ss,\fk H)$ be the M\"obius covariant weak categorical extension in Thm. \ref{lb47}. Let $\scr E_Q^w=(\mc B_Q,\Rep_{\Rep^V(\mc A_V)}(\mc B_Q),\boxdot_Q,\ss^Q,\fk H)$ be the weak categorical extension as in Thm. \ref{lb71}, which is clearly also M\"obius covariant. Thus, for each $W_i,W_k\in\Obj(\Rep^0(P))$ and $\wtd I\in\Jtd$, and for each $\fk a=(T,w^{(i_1)},\dots,w^{(i_n)},\wtd f_1,\dots,\wtd f_n)$ in $\fk H_i(\wtd I)$ as described in \eqref{eq45}, we have 
\begin{subequations}\label{eq62}
\begin{gather}
\mc L^Q(\fk a,\wtd I)\big|_{\mc H_k^\infty}=\mu_{i,k}(T\boxdot\idt_j)\circ \Gamma^V(w^{(i_1)},\wtd f_1)\cdots \Gamma^V(w^{(i_n)},\wtd f_n)\big|_{\mc H_k^\infty}\label{eq62a}\\
\mc R^Q(\fk a,\wtd I)\big|_{\mc H_k^\infty}=\mu_{k,i}(\idt_j\boxdot T)\circ \Delta^V(w^{(i_n)},\wtd f_n)\cdots \Delta^V(w^{(i_1)},\wtd f_1)\big|_{\mc H_k^\infty}\label{eq62b}
\end{gather}
\end{subequations}
By Thm. \ref{lb71}, the closure of $\scr E^w_Q$ is $\scr E_Q$.

If $V$ satisfies Condition I, then $\Gamma^V,\Delta^V$ are always energy-bounded. Comparing \eqref{eq57} with \eqref{eq62b}, we see that $\fk R^Q(\yk,\wtd J)$ equals $\ovl{\mc R^Q(\fk a,\wtd J)}$ for some $\fk a\in\fk H_j(\wtd J)$, and hence is a right operator of $\scr E_Q$. However, we are only assuming that $V$ satisfies Condition II. So we need more effort. In fact, we shall use the powerful Thm. \ref{lb13}. \\[-1ex]

Step 2. Let $B^Q(\yk,\wtd J)$ be the operation associating to each $W_k\in\Obj(\Rep^0(P))$ the smooth and localizable operator
\begin{align*}
B^Q(\yk,\wtd J)\big|_{\mc H_k^\infty}=(\mu_{k,j}\Delta^V)(w^{(j)},\wtd g)\big|_{\mc H_k^\infty}
\end{align*}
By Thm. \ref{lb13}, it suffices to prove that $B^Q(\yk,\wtd J)$ is a weak right operator of $\scr E^w_Q$. To check Def. \ref{lb10}-(a), we choose any $W_{k'}\in\Obj(\Rep^0(P))$ and $G\in\Hom_P(W_k,W_{k'})=\Hom_{\mc B_Q}(\mc H_k,\mc H_{k'})$. Then for each $\chi\in\mc H_k^\infty$ we have
\begin{align*}
&B^Q(\yk,\wtd J)G\chi=\big(\mu_{k',j}\Delta^V\big)(w^{(j)},\wtd g)G\chi\xlongequal{\eqref{eq36}}\big(\mu_{k',j}(G\boxdot\idt_j)\Delta^V\big)(w^{(j)},\wtd g)\chi\\
\xlongequal{\eqref{eq19}}&\big((G\boxdot_Q\idt_j)\mu_{k,j}\Delta^V\big)(w^{(j)},\wtd g)\chi=(G\boxdot_Q\idt_j)B^Q(\yk,\wtd J)\chi
\end{align*}

To check Def. \ref{lb10}-(b), we need to show that if $\wtd J$ is clockwise to $\wtd I$ then $B^Q(\yk,\wtd J)$ commutes with $\mc L^Q(\fk a,\wtd I)$. Namly, for each $W_k\in\Obj(\Rep^0(P))$, we need to prove
\begin{align}
\begin{aligned}
&\big(\mu_{i\boxdot_Qk,j}\Delta^V\big)(w^{(j)},\wtd g)\cdot\mu_{i,k}(T\boxdot\idt_k) \Gamma^V(w^{(i_1)},\wtd f_1)\cdots \Gamma^V(w^{(i_n)},\wtd f_n)\big|_{\mc H_k^\infty}\\
=&\mu_{i,k\boxdot_Qj}(T\boxdot\idt_{k\boxdot_Qj})\Gamma^V(w^{(i_1)},\wtd f_1)\cdots \Gamma^V(w^{(i_n)},\wtd f_n)\cdot\big(\mu_{k,j}\Delta^V\big)(w^{(j)},\wtd g)\big|_{\mc H_k^\infty}
\end{aligned}
\end{align}
By Cor. 3.13 (or Prop. 3.12) of \cite{Gui19a}, it suffices to prove the following braiding relation: If $z_1,\dots,z_n,\zeta\in\Sbb^1$ are equipped with arg values such that $\arg z_1>\cdots>\arg z_n>\arg\zeta>\arg z_1-2\pi$, and if $w^{(k)}\in W_k$, then
\begin{align}
\begin{aligned}
&\mu_{i\boxdot_Qk,j}\Delta^V(w^{(j)},\zeta)\mu_{i,k}(T\boxdot\idt_k) \Gamma^V(w^{(i_1)},z_1)\cdots \Gamma^V(w^{(i_n)},z_n)w^{(k)}\\
=&\mu_{i,k\boxdot_Qj}(T\boxdot\idt_{k\boxdot_Qj})\Gamma^V(w^{(i_1)},z_1)\cdots \Gamma^V(w^{(i_n)},z_n)\mu_{k,j}\Delta^V(w^{(j)},\zeta)w^{(k)}
\end{aligned}
\end{align}
holds where both sides are linear functionals on $W_i\boxdot_Q W_j\boxdot_Q W_k$. (Note that these are not ordinary products of linear operators. They should be understood in terms of analytic continuation. We refer the readers to \cite{Gui19a} Sec. 2.2 (especially the paragraph before Thm. 2.8) for the precise meaning.)

Choose any $w^{(k)}\in W_k$. In view of \cite[Prop. 2.11]{Gui19a} (which says, roughly, that if one has the braid relation $\mc Y_1\mc Y_2\sim\mc Y_3\mc Y_4$, then one also has the braid relation $\mc X_1\mc Y_1\mc Y_2\mc X_2\sim\mc X_1\mc Y_3\mc Y_4\mc X_2$ where $\mc X_1$ and $\mc X_2$ are products of intertwining operators), we are able to do the following calculations:
\begin{align*}
&\mu_{i\boxdot_Qk,j}\Delta^V(w^{(j)},\zeta)\mu_{i,k}(T\boxdot\idt_k) \Gamma^V(w^{(i_1)},z_1)\cdots \Gamma^V(w^{(i_n)},z_n)w^{(k)}\\
\xlongequal{\eqref{eq36}}&\mu_{i\boxdot_Qk,j}(\mu_{i,k}\boxdot\idt_j)(T\boxdot\idt_k\boxdot\idt_j)\Delta^V(w^{(j)},\zeta) \Gamma^V(w^{(i_1)},z_1)\cdots \Gamma^V(w^{(i_n)},z_n)w^{(k)}\\
\xlongequal[\eqref{eq58}]{\eqref{eq35}}&\mu_{i,k\boxdot_Qj}(\idt_i\boxdot\mu_{k,j})(T\boxdot\idt_k\boxdot\idt_j) \Gamma^V(w^{(i_1)},z_1)\cdots \Gamma^V(w^{(i_n)},z_n)\Delta^V(w^{(j)},\zeta)w^{(k)}
\end{align*}
Since $(\idt_i\boxdot\mu_{k,j})(T\boxdot\idt_k\boxdot\idt_j)=T\boxdot\mu_{k,j}=(T\boxdot\idt_{k\boxdot_Qj})(\idt_{i_1\boxdot\cdots\boxdot i_n}\boxdot\mu_{k,j})$, the above expression equals
\begin{align*}
&\mu_{i,k\boxdot_Qj}(T\boxdot\idt_{k\boxdot_Qj})(\idt\boxdot\mu_{k,j})\Gamma^V(w^{(i_1)},z_1)\cdots \Gamma^V(w^{(i_n)},z_n)\Delta^V(w^{(j)},\zeta)w^{(k)}\\
\xlongequal{\eqref{eq36}}&\mu_{i,k\boxdot_Qj}(T\boxdot\idt_{k\boxdot_Qj})\Gamma^V(w^{(i_1)},z_1)\cdots \Gamma^V(w^{(i_n)},z_n)\mu_{k,j}\Delta^V(w^{(j)},\zeta)w^{(k)}
\end{align*}
This finishes the proof.
\end{proof}

With a little more effort, we can prove a stronger version of Thm. \ref{lb72}:

\begin{thm}\label{lb83}
Assume that $V$ satisfies Condition II, and let $\scr E_Q$ be as in Thm. \ref{lb72}. Choose unitary $U_P$-modules $W_{j_1},\dots,W_{j_m},W_j$ and $S\in\Hom_{U_P}(W_{j_1}\boxdot_{U_P}\cdots\boxdot_{U_P} W_{j_m},W_j)$. Assume that every unitary intertwining operator of $U_P$ whose charge space is one of $W_{j_1},\dots,W_{j_m}$ is energy-bounded. Let $w^{(j_1)}\in W_{j_1},\dots,w^{(j_m)}\in W_{j_m}$ be homogeneous. Let $\wtd J\in\Jtd$ and $\wtd g_1,\dots,\wtd g_m\in C_c^\infty(\wtd J)$. For each $W_k\in\RepUP$, let
\begin{subequations}
\begin{gather}
\mc L^Q(\xk,\wtd J)\big|_{\mc H_k^\infty}=(\fk V^\boxdot_{i,k})^{-1}(S\boxdot_{U_P}\idt_k)\circ \Gamma^{U_P}(w^{(j_1)},\wtd g_1)\cdots \Gamma^{U_P}(w^{(j_m)},\wtd g_m)\big|_{\mc H_k^\infty}\\
\mc R^Q(\yk,\wtd J)\big|_{\mc H_k^\infty}=(\fk V^\boxdot_{k,i})^{-1}(\idt_k\boxdot_{U_P} S)\circ \Delta^{U_P}(w^{(j_m)},\wtd g_m)\cdots \Delta^{U_P}(w^{(j_1)},\wtd g_1)\big|_{\mc H_k^\infty}
\end{gather}
\end{subequations}
Then $\ovl{\mc L^Q(\xk,\wtd J)}$ and $\ovl{\mc R^Q(\yk,\wtd J)}$ are respectively left and right operators of $\scr E_Q$ with charge spaces $\mc H_j$.
\end{thm}

\begin{rem}
In Thm. \ref{lb83}, we do not assume that $\wtd g_{\mu+1}$ is clockwise to $\wtd g_\mu$. Therefore, unlike in Thm. \ref{lb47}, we do not necessarily have $\xk=\yk$. So it is not necessarily true that $\mc R^Q(\yk,\wtd J)|_{\mc H_k^\infty}=\ss_{j,k}^Q\mc L^Q(\xk,\wtd J)|_{\mc H_k^\infty}$.
\end{rem}

\begin{proof}[\textbf{Proof of Thm. \ref{lb83}}]
By \eqref{eq33} and \eqref{eq65}, instead of proving Thm. \ref{lb83} for every system of fusion products $(\boxdot_P,\mu_{\blt,\star})$, it suffices to prove it for one system. So we assume that $(\boxdot_P,\mu_{\blt,\star})=(\boxdot_{U_P},\mu_{\blt,\star})$ is such that $\fk V^\boxdot=\id$ (cf. Rem. \ref{lb73}). Note that $\mc L^Q(\xk,\wtd J)$ is smooth and localizable by Rem. \ref{lb54}. Let $\scr E^w_Q$ be as in Thm. \ref{lb71}. It is easy to check that a product of weak left operators of $\scr E^w_Q$ is again a weak left operator provided that it is smooth and localizable. It is also easy to check that a weak left operator, multiplied by $(S\boxdot_{U_P}\idt)$, is again a weak left operator. Thus, by Thm. \ref{lb72}, $\mc L^Q(\xk,\wtd J)$ is a weak left operator of $\scr E^w_Q$. By Thm. \ref{lb13}, $\ovl{\mc L^Q(\xk,\wtd J)}$ is a left operator of the closure of $\scr E_Q^w$ (which equals $\scr E_Q$ by Thm. \ref{lb71}). The treatment of right operators is similar.
\end{proof}

\section{Applications of Thm. \ref{lb72}}

\subsection{Strong locality and strong integrability}

\begin{lm}\label{lb77}
Let $U$ be a unitary VOA extension of a unitary VOA $V$. Assume that $U$ is $C_2$-cofinite. Assume that $U$, as a unitary $V$-module, is a finite direct sum of unitary irreducible $V$-modules. Let $(W_i,Y^U_i)$ be an irreducible unitary $U$-module. Assume that $Y_i^U(v,z)$ is energy-bounded for each homogeneous $v\in V$. Then $Y_i^U$ is energy-bounded.
\end{lm}

\begin{proof}
This is \cite[Thm. 4.6]{CT23} when $W_i=U$. The general case can be proved in the same way: Since $U$ is $C_2$-cofinite, by \cite{Zhu96}, the series $\Tr\big(Y^U_i(u)_{\wt(u)-1}q^{2L_0}\big)$ of $q$ is absolutely convergent on $0<|q|<1$ for all homogeneous $u\in U$. As in the proof of \cite[Prop. 3.17]{CT23}, for every homogenoues $u\in U$ and $0<|q|<1$ one concludes that $Y^U_i(u)_{\wt u-1}q^{L_0}:W_i\rightarrow W_i$ is norm-bounded. Using the same proof as for \cite[Thm. 4.5]{CT23}, one shows that $Y^U_i$ is energy-bounded.
\end{proof}

\begin{thm}\label{lb78}
Assume that $V$ satisfies Condition II. Then every unitary VOA extension $U$ of $V$ is strongly local, and every unitary $U$-module is strongly integrable. 
\end{thm}
Thus, we have the CKLW net $\mc A_U$ and CWX functor $\fk F_\CWX^U:\RepU\rightarrow\Rep(\mc A_U)$.

\begin{proof}
By Thm. \ref{lb75}, we have $U=U_P$ for some haploid $C^*$-Frobenius algebra $P$ in $\RepV$. We assume the setting in Subsubsec. \ref{lb74}, where $(\boxdot_P,\mu_{\blt,\star})$ is chosen to satisfy the requirement in Rem. \ref{lb73} and hence $\fk V^\boxdot=\id$. Since $V$ is strongly energy-bounded, by Lem. \ref{lb77}, $U_P$ is strongly energy-bounded. Thus, for each homogeneous $u_1,u_2\in U_P$ and each $f\in C_c^\infty(I),g\in C_c^\infty(I')$ (where $I\in\mc J$), by Thm. \ref{lb72} and Rem. \ref{lb76}, $\ovl{Y^{U_P}_\blt(u_1,f)}$ is a left operator of $\scr E_Q$, and $\ovl{Y^{U_P}_\blt(u_2,g)}$ is a right operator of $\scr E_Q$. Thus, for each unitary $U_P$-module $W_k$, the closed operators $\ovl{Y^{U_P}_k(u_1,f)}$ and $\ovl{Y^{U_P}_k(u_2,g)}$ commute strongly by the Locality of $\scr E_Q$ (Thm. \ref{lb7}). This proves that $U_P$ is strongly local. 

The locality of $\scr E_Q$ also shows that for every right operator $\fk R^Q(\yk,\wtd J)$ with charge space $\mc H_k$ (where $\wtd J\in\Jtd$ is disjoint from $I$), the following diagram commutes strongly:
\begin{equation*}
\begin{tikzcd}[column sep=huge]
\mc H_a \arrow[d,"{\fk R^Q(\yk,\wtd J)}"'] \arrow[r,"{\ovl{Y^{U_P}(u_1,f)}}"] & \mc H_a \arrow[d,"{\fk R^Q(\yk,\wtd J)}"] \\
\mc H_k \arrow[r,"{\ovl{Y_k^{U_P}(u_1,f)}}"]           & \mc H_k        
\end{tikzcd}
\end{equation*}
Thus, by Lem. \ref{lb43} and the density of fusion products (Def. \ref{lb31}), $W_k$ is a strongly integrable $U_P$-module.
\end{proof}

\subsection{Condition I is preserved by unitary extensions}

\begin{thm}\label{lb80}
Assume that $V$ satisfies Condition II. Let $U$ be a unitary VOA extension of $V$. Let $W_i$ be a unitary $U$-module such that any unitary intertwining operator of $U$ with charge space $W_i$ is energy-bounded. Then every unitary intertwining operator of $U$ with charge space $W_i$ satisfies the strong intertwining property.
\end{thm}

\begin{proof}
As in the proof of Thm. \ref{lb78}, we let $U=U_P$, and assume $\boxdot_P=\boxdot_{U_P}$ and $\fk V^\boxdot=\id$. By Thm. \ref{lb72}, for every homogeneous $u\in U,w^{(i)}\in W_i$, any $\wtd I,\wtd J\in\Jtd$ with $\wtd J$ clockwise to $\wtd I$, and any $\wtd f\in C_c^\infty(\wtd I), g\in C_c^\infty(J)$, we have that $\ovl{Y^{U_P}_\blt(u,g)}$ and $\ovl{\Gamma^{U_P}(w^{(i)},\wtd f)}$ are respectively left and right operators of $\scr E_Q$. Thus, the locality of $\scr E_Q$ (Thm. \ref{lb7}) implies that for each $W_j\in\Obj(\RepUP)$, the intertwining operator $\Gamma^{U_P}(\cdot,z)|_{W_j}$ (of type $W_i\boxdot_{U_P}W_j\choose W_i~W_j$) satisfies the strong intertwining property. Thus, by \eqref{eq53} and Lem. \ref{lb35}, every unitary intertwining operator of $U$ with charge space $W_i$ and source space $W_j$ (which is the product of a homomorphism and $\Gamma^{U_P}(\cdot,z)|_{W_j}$) satisfies the strong intertwining property.
\end{proof}

\begin{co}\label{lb81}
Assume that $V$ satisfies Condition II and $U$ is a unitary VOA extension of $V$. Suppose that $\mc F^U$ is a set of unitary $U$-modules $\boxdot_U$-generating $\RepU$, and that every unitary intertwining operator of $U$ whose charge space is in $\mc F^U$ is energy-bounded. Then $U$ satisfies Condition II.
\end{co}

\begin{proof}
By Thm. \ref{lb67}, $U$ is completely unitary. By Lem. \ref{lb77}, $U$ is strongly energy-bounded. By Thm. \ref{lb78}, $U$ is strongly local. By Thm. \ref{lb80}, every unitary intertwining operator of $U$ with charge space in $\mc F^U$ satisfies the strong intertwining property. So $U$ satisfies Condition II.
\end{proof}

\begin{co}\label{lb90}
Suppose that $V$ satisfies Condition I. Then every unitary VOA extension of $V$ satisfies Condition I.
\end{co}

\begin{proof}
If $U$ is a unitary VOA extension of $V$, then every unitary intertwining operator of $U$ is also a unitary intertwining operator of $V$, which is energy-bounded. So the present corollary follows immediately from Cor. \ref{lb81}.
\end{proof}

\subsection{Some preliminary comparison theorems}

We assume the setting described in Subsubsec. \ref{lb74}. 

\begin{thm}\label{lb88}
Assume that $V$ satisfies Condition II. Then we have
\begin{align*}
\mc A_{U_P}=\mc B_Q
\end{align*}
as conformal nets acting on the Hilbert space $\mc H_a$. Moreover, the CWX functor $\fk F^{U_P}_\CWX:\RepUP\rightarrow\Rep(\mc B_Q)$ has image in $\Rep_{\Rep^V(\mc A_V)}(\mc A_{U_P})$, and the following diagram commutes
\begin{equation}\label{eq67}
\begin{tikzcd}[row sep=large,column sep=2cm]
\Rep^0(P) \arrow[r,"{\fk F^V_\CWX}","\simeq"'] \arrow[d,"{\fk F_\VOA}"',"\simeq"] & \Rep^0_{\Rep^V(\mc A_V)}(Q) \arrow[d,"{\fk F_\CN}","\simeq"'] \\
\RepUP \arrow[r,"{\fk F^{U_P}_\CWX}","\simeq"']           & \Rep_{\Rep^V(\mc A_V)}(\mc B_Q)          
\end{tikzcd}
\end{equation}
where each of the four arrows is an isomorphism of $C^*$-categories.
\end{thm}

Note that by Thm. \ref{lb78}, $\mc A_{U_P}$ and $\fk F^{U_P}_\CWX$ can be defined.

\begin{proof}
Choose any $\wtd I\in\Jtd$. By Thm. \ref{lb72} and Rem. \ref{lb76}, for each homogeneous $u\in U_P$ and $f\in C_c^\infty(I)$, $\ovl{Y^{U_P}_\blt(u,f)}$ is a right operator of $\scr E_Q$ with charge space $\mc H_a$ localized in $\wtd I$. By Exp. \ref{lb85}, there exists a closed operator $X$ on $\mc H_a$ with core $\mc H_a(I')$ such that for each $(\mc H_k,\pi_k)\in\Obj(\Rep_{\Rep^V(\mc A_V)}(\mc B_Q))$ we have $\pi_{k,I}(X)=\ovl{Y^{U_P}_k(u,f)}$. Taking $k=a$, we get $X=\ovl{Y^{U_P}(u,f)}$ (and hence $\ovl{Y^{U_P}(u,f)}$ is affiliated with $\mc B_Q(I)$) and
\begin{align}\label{eq70}
\pi_{k,I}\big(\ovl{Y^{U_P}(u,f)}\big)=\ovl{Y^{U_P}_k(u,f)}
\end{align}

That $\ovl{Y^{U_P}(u,f)}$ is affiliated with $\mc A_{U_P}(I)$ for all $u,f$ shows that $\mc A_{U_P}(I)\subset\mc B_Q(I)$. By Haag duality, we get $\mc A_{U_P}(I')\supset\mc B_Q(I')$. By replacing $I$ with $I'$, we get $\mc A_{U_P}(I)\supset\mc B_Q(I)$. This proves $\mc A_{U_P}=\mc B_Q$. To prove that \eqref{eq67} commutes, in view of \eqref{eq69}, it suffices to prove that $\fk F_\CWX^{U_P}(W_k,Y_k^{U_P})=(\mc H_k,\pi_k)$. But this follows directly from \eqref{eq70}
\end{proof}

We now extend the identification of conformal net extensions in Thm. \ref{lb88} to the identification of categorical extensions with the help of Thm. \ref{lb83}.

\begin{sett}\label{lb89}
In addition to the setting in Subsubsec. \ref{lb74}, we let
\begin{align}\label{eq71}
\scr E_Q=(\mc B_Q,\Rep_{\Rep^V(\mc A_V)}(\mc B_Q),\boxdot_Q,\ss^Q,\mc H):=\mu_{\blt,\star}\circ\scr E_V
\end{align}
be the categorical extension associated to $Q$, $\scr E_V$, and $(\boxdot_Q,\mu_{\blt,\star})$ (cf. Def. \ref{lb69}). Let
\begin{align}
\scr E_{U_P}=(\mc A_{U_P},\Rep^{U_P}(\mc A_{U_P}),\boxdot_{U_P},\ss^{U_P},\mc H)
\end{align}
be the categorical extension associated to $U_P$ (cf. Def. \ref{lb62}). Note that by Thm. \ref{lb88}, we have $\mc A_{U_P}=\mc B_Q$ and 
\begin{align}
\Rep^{U_P}(\mc A_{U_P})\xlongequal{\eqref{eq105}}\fk F^{U_P}_\CWX(\RepUP)=\Rep_{\Rep^V(\mc A_V)}(\mc B_Q)
\end{align}
\end{sett}

\begin{thm}\label{lb92}
Assume that both $V$ and $U_P$ satisfy Condition II. Assume Setting \ref{lb89}. Then
\begin{align}
\scr E_{U_P}=\fk V^\boxdot\circ\scr E_Q
\end{align}
\end{thm}

Thus, by our notations (cf. Rem. \ref{lb99}), if $\fk L^Q(\xk,\wtd I)$ and $\fk R^Q(\yk,\wtd I)$ are respectively left and right operators of $\scr E_{U_P}$ with charge space $\mc H_i$ localized in $\wtd I$, then $\fk L^{U_P}(\xk,\wtd I)$ and $\fk R^{U_P}(\yk,\wtd I)$ are respectively left and right operators of $\scr E_{U_P}$ with charge space $\mc H_i$ localized in $\wtd I$, where for each $\mc H_k\in\Obj(\Rep_{\Rep^V(\mc A_V)}(\mc B_Q))$ we set
\begin{subequations}
\begin{gather}
\fk L^{U_P}(\xk,\wtd I)\big|_{\mc H_k}=\fk V^{\boxdot}_{i,k}\circ\fk L^Q(\xk,\wtd I)\big|_{\mc H_k}\\
 \fk R^{U_P}(\yk,\wtd I)\big|_{\mc H_k}=\fk V^{\boxdot}_{k,i}\circ\fk R^Q(\yk,\wtd I)\big|_{\mc H_k}
\end{gather}
\end{subequations}
where
\begin{gather}
\fk V^\boxdot_{\blt,\star}:\mc H_\blt\boxdot_P\mc H_\star\rightarrow\mc H_\blt\boxdot_{U_P}\mc H_\star
\end{gather}
is the closure (i.e., the image under $\fk F^{U_P}_\CWX$) of $\fk V^\boxdot_{\blt,\star}:W_\blt\boxdot_PW_\star\rightarrow W_\blt\boxdot_{U_P}W_\star$.

\begin{proof}
Let $\mc F^{U_P}$ be a $\boxdot_{U_P}$-generating set of unitary $U_P$-modules as described in Condition II. Define a M\"obius covariant categorical extension $\scr E^w_{U_P}=(\mc A_{U_P},\Rep^{U_P}(\mc A_{U_P}),\boxdot_{U_P},\ss^{U_P},\fk H^{U_P})$ as in Thm. \ref{lb47} using $\mc L^{U_P}(\fk a,\wtd I),\mc R^{U_P}(\fk a,\wtd I)$ defined in a similar way as in \eqref{eq46}. So $\scr E_{U_P}$ is the closure of $\scr E_{U_P}^w$. By Thm. \ref{lb83}, $(\fk V^\boxdot)^{-1}\ovl{\mc L^{U_P}(\fk a,\wtd I)}$ and $(\fk V^\boxdot)^{-1}\ovl{\mc R^{U_P}(\fk a,\wtd I)}$ are left and right operators of $\scr E_Q$ localized in $\wtd I$ with the charge space unchanged. Thus, $\ovl{\mc L^{U_P}(\fk a,\wtd I)}$ and $\ovl{\mc R^{U_P}(\fk a,\wtd I)}$ are left and right operators of $\fk V^\boxdot\scr E_Q$, and hence $\fk V^\boxdot\scr E_Q$ is the closure of $\scr E_{U_P}^w$. This proves $\scr E_{U_P}=\fk V^\boxdot\scr E_Q$.
\end{proof}

Finally, we give an equivalent formulation of Thm. \ref{lb92} in terms of comparison of tensorators.

\begin{sett}\label{lb91}
In addition to the settings in Subsubsec. \ref{lb74} and Setting \ref{lb89}: We let 
\begin{align*}
\big(\Rep(\mc B_Q),\boxtimes_{\mc B_Q},\mathbb B^{\mc B_Q}\big)
\end{align*}
be the Connes braided $C^*$-tensor category for $\mc B_Q$ (cf. Thm. \ref{lb12}). Let
\begin{subequations}\label{eq91}
\begin{align}
\fk N^\boxdot:\mc H_\blt\boxdot_Q\mc H_\star\rightarrow\mc H_\blt\boxtimes_{\mc B_Q}\mc H_\star
\end{align}
be the tensorator associated to $\scr E_Q$ (cf. Cor. \ref{lb41}), i.e., it is determined by
\begin{align}\label{eq91b}
\fk N^\boxdot\circ\scr E_Q=\scr E_{\mc B_Q,\Connes}\big|_{\Rep_{\Rep^V(\mc A_V)}(\mc B_Q)}
\end{align}
\end{subequations}
where $\scr E_{\mc B_Q,\Connes}$ is the Connes categorical extension for $\mc B_Q$. Let
\begin{subequations}\label{eq92}
\begin{align}
\fk W^{U_P}:\mc H_\blt\boxdot_{U_P}\mc H_\star\rightarrow\mc H_\blt\boxtimes_{\mc B_Q}\mc H_\star
\end{align}
be the Wassermann tensorator for $U_P$ (cf. Def. \ref{lb62}), i.e., it is determined by
\begin{align}\label{eq92b}
\fk W^{U_P}\circ\scr E_{U_P}=\scr E_{\mc B_Q,\Connes}\big|_{\Rep_{\Rep^V(\mc A_V)}(\mc B_Q)}
\end{align}
\end{subequations}
\end{sett}

\begin{thm}\label{lb94}
Assume that both $V$ and $U_P$ satisfy Condition II. Assume Setting \ref{lb91}. Then the following diagram of braided $*$-functors commutes:
\begin{equation}\label{eq79}
\begin{tikzcd}[row sep=large,column sep=3cm]
\big(\Rep^0(P),\boxdot_P,\ss^P\big) \arrow[r,"{(\fk F^V_\CWX,\id)}","\simeq"'] \arrow[d,"{(\fk F_\VOA,\fk V^\boxdot)}"',"\simeq"] & \big(\Rep^0_{\Rep^V(\mc A_V)}(Q),\boxdot_Q,\ss^Q \big)\arrow[d,"{(\fk F_\CN,\fk N^\boxdot)}","\simeq"'] \\
\big(\RepUP,\boxdot_{U_P},\ss^{U_P} \big)\arrow[r,"{(\fk F^{U_P}_\CWX,\fk W^{U_P})}","\simeq"']           & \big(\Rep_{\Rep^V(\mc A_V)}(\mc B_Q),\boxtimes_{\mc B_Q},\mathbb B^{\mc B_Q} \big)         
\end{tikzcd}
\end{equation}
\end{thm}

\begin{proof}
By Thm. \ref{lb92}, we have
\begin{align*}
\fk W^{U_P}\circ\fk V^\boxdot\circ\scr E_Q=\fk W^{U_P}\circ\scr E_{U_P}\xlongequal{\eqref{eq92b}}\scr E_{\mc B_Q,\Connes}\big|_{\Rep_{\Rep^V(\mc A_V)}(\mc B_Q)}\xlongequal{\eqref{eq91b}}\fk N^\boxdot\circ\scr E_Q
\end{align*}
Therefore $\fk W^{U_P}\circ\fk V^\boxdot=\fk N^\boxdot$. (More precisely, we have $\fk W^{U_P}\circ\fk F^{U_P}_\CWX(\fk V^\boxdot)=\fk N^\boxdot$.)
\end{proof}

\subsection{The main comparison theorem}

In this subsection, we reformulate the comparison Thm. \ref{lb88} and \ref{lb94} in a more accessible way. For that purpose, we need a setting different from that in Subsubsec. \ref{lb74}. The difference is mainly due to choosing a different braided $C^*$-tensor structure on $\Rep^V(\mc A_V)$ (i.e. the one defined by Connes fusion). 


In the following subsubsection, we lay out the assumptions required for our main comparison Thm. \ref{lb95}. For the reader’s convenience in locating the relevant definitions, the four categories involved in this theorem are highlighted in bold, and the four braided $*$-functors are enclosed in boxes.

\subsubsection{The setting}\label{lb93}

Let $V$ satisfy Condition II. Again, we let $\mc A_V$ be the CKLW net of $V$, and let $(\RepV,\boxdot,\ss)$ be the Huang-Lepowsky braided $C^*$-tensor category of unitary $V$-modules (Thm. \ref{lb61}). Note that the $C^*$-subcategory
\begin{align}
\Rep^V(\mc A_V):=\FVCWX(\RepV)
\end{align}
of $\Rep(\mc A_V)$ is clearly full and replete. Thus, by the tensorator $\mc H_\blt\boxdot\mc H_\star\rightarrow\mc H_\blt\boxtimes\mc H_\star$, $\Rep^V(\mc A_V)$ is clearly closed under Connes fusion $\boxtimes=\boxtimes_{\mc A_V}$. Thus, the Connes braided $C^*$-tensor category for $\mc A_V$ (Thm. \ref{lb12}) restricts to
\begin{align*}
\big(\Rep^V(\mc A_V),\boxtimes,\mbb B \big)\equiv\big(\Rep^V(\mc A_V),\boxtimes_{\mc A_V},\mbb B^{\mc A_V}\big)
\end{align*}
By Thm. \ref{lb1}, we have an isomorphism of braided $C^*$-tensor categories
\begin{align}\label{eq72}
(\FVCWX,\fk W^V):\big(\RepV,\boxdot,\ss\big)\xlongrightarrow{\simeq}\big(\Rep^V(\mc A_V),\boxtimes,\mbb B \big)
\end{align}
where $\fk W^V:\mc H_\blt\boxdot\mc H_\star\rightarrow\mc H_\blt\boxtimes\mc H_\star$ is the Wassermann tensorator for $V$ (cf. Def. \ref{lb62}), i.e., it is determined by
\begin{align}\label{eq76}
\fk W^V\circ\scr E_V=\scr E_{\mc A_V,\Connes}\big|_{\Rep^V(\mc A_V)}
\end{align}
where $\scr E_{\mc A_V,\Connes}$ is the Connes categorical extension for $\mc A_V$ whose L and R operators are denoted by $L^\boxtimes,R^\boxtimes$.

Let $P=(W_a,\mu,\iota)$ be a haploid commutative $C^*$-Frobenius algebra in $\RepV$. Let
\begin{align*}
\Theta=(\FVCWX,\fk W^V)(P)=(\mc H_a,\fk m,\iota)
\end{align*}
be the pushforward of $P$ in $\Rep^V(\mc A_V)$ (and hence in $\Rep(\mc A_V)$). Thus,  $\mc H_a$ is $\FVCWX(W_a)$ as an $\mc A_V$-module, and $\mk:\mc H_a\boxtimes\mc H_a\rightarrow\mc H_a$ is related to $\mu:W_a\boxdot W_a\rightarrow W_a$ (or more precisely, related to the continuous extension $\mu:\mc H_a\boxdot\mc H_a\rightarrow\mc H_a$) by
\begin{align}\label{eq106}
\mk=\mu\circ(\fk W^V_{a,a})^{-1}
\end{align}
Then \eqref{eq72} is lifted to an isomorphism of $C^*$-categories
\begin{align*}
\tFVCWX:\Rep^0(P)\xlongrightarrow{\simeq}\Rep^0_{\Rep^V(\mc A_V)}(\Theta)
\end{align*}
where the target is the category of dyslectic $\Theta$-modules in $\Rep^V(\mc A_V)$. $\tFVCWX$ can be explicitly described as follows: For each $(W_i,\mu^i)\in\Obj(\Rep^0(P))$, 
\begin{align}\label{eq80}
\tFVCWX(W_i,\mu^i)=(\mc H_i,\mk^i)\quad\text{where}\quad\mk^i=\mu^i\circ(\fk W_{a,i}^V)^{-1}:\mc H_a\boxtimes\mc H_i\rightarrow\mc H_i
\end{align}
and action of $\tFVCWX$ on the morphisms is the identity map.\footnote{Since $\fk W^V$ is natural, for each $W_i,W_{i'}\in\Obj(\Rep^0(P))$, an element $T\in\Hom_V(W_i,W_{i'})$ belongs to $\Hom_P(W_i,W_{i'})$ if and only if its closure $T:\mc H_i\rightarrow\mc H_{i'}$ belongs to $\Hom_\Theta(\mc H_i,\mc H_{i'})$.}

Let $(\boxdot_P,\mu_{\blt,\star})$ and $(\boxtimes_\Theta,\mk_{\blt,\star})$ be systems of fusion products in $\Rep^0(P)$ and $\Rep^0_{\Rep^V(\mc A_V)}(\Theta)$ respectively. Our discussion of (unitary) linking map in ensures that for each $W_i,W_j\in\Obj(\Rep^0(P))$, there is a unitary map $\wtd{\fk W}^V_{i,j}:\mc H_i\boxdot_P\mc H_j\rightarrow H_i\boxtimes_\Theta\mc H_j$ intertwining the actions of $P$ and $\Theta$ such that the following diagram commutes
\begin{equation}\label{eq85}
\begin{tikzcd}[column sep=large]
\mc H_i\boxdot\mc H_j\arrow[r,"\fk W^V_{i,j}","\simeq"']\arrow[d,"\mu_{i,j}"'] & \mc H_i\boxtimes\mc H_j\arrow[d,"\mk_{i,j}"]\\
\mc H_i\boxdot_P\mc H_j\arrow[r,"\wtd{\fk W}^V_{i,j}","\simeq"']& \mc H_i\boxtimes_\Theta\mc H_j
\end{tikzcd}
\end{equation}
Then, similar to Rem. \ref{lb39}, we have an isomorphism of braided $C^*$-tensor categories
\begin{align}\label{eq73}
\tcboxmath{(\tFVCWX,\wtd{\fk W}^V):\big(\Rep^0(P),\boxdot_P,\ss^P\big)\xlongrightarrow{\simeq}\big(\Rep_{\Rep^V(\mc A_V)}^0(\Theta),\boxtimes_\Theta,\ss^\Theta\big)}
\end{align}
where $\pmb{(\Rep^0(P),\boxdot_P,\ss^P)}$ and $\pmb{(\Rep_{\Rep^V(\mc A_V)}^0(\Theta),\boxtimes_\Theta,\ss^\Theta)}$ are the braided $C^*$-tensor categories associated to $(\boxdot_P,\mu_{\blt,\star})$ and $(\boxtimes_\Theta,\mk_{\blt,\star})$ respectively (cf. Thm. \ref{lb19}).

Let $\mc B_\Theta$ be the conformal net extension of $\mc A_V$ associated to $\Theta$ (cf. Thm. \ref{lb26}), i.e., for each $\wtd I\in\Jtd$ we have
\begin{align}\label{eq75}
\mc B_\Theta(I)=\big\{\mk L^\boxtimes(\xi,\wtd I)|_{\mc H_a}:\xi\in\mc H_a(I) \big\}
\end{align}
Let $\pmb{(\Rep_{\Rep^V(\mc A_V)}(\mc B_\Theta),\boxtimes_{\mc B_\Theta},\mbb B^{\mc B_\Theta})}$ be the Connes braided $C^*$-tensor category of $\mc B_\Theta$-modules whose restrictions to $\mc A_V$ are objects in $\Rep^V(\mc A_V)$. By Cor. \ref{lb41}, we have an isomorphism of braided $C^*$-tensor categories
\begin{align}\label{eq87}
\tcboxmath{(\fk F_\CN,\fk N^\boxtimes):\big(\Rep_{\Rep^V(\mc A_V)}^0(\Theta),\boxtimes_\Theta,\ss^\Theta\big)\xlongrightarrow{\simeq} \big(\Rep_{\Rep^V(\mc A_V)}(\mc B_\Theta),\boxtimes_{\mc B_\Theta},\mbb B^{\mc B_\Theta}\big)}
\end{align}
where $\fk F_\CN$ is defined in a similar way as in Thm. \ref{lb27}, i.e.,
\begin{gather}\label{eq84}
\begin{gathered}
\fk F_\CN(\mc H_i,\mk^i)=(\mc H_i,\pi_i')\\
\text{where }\pi_{i,I}'(\mk L^\boxtimes(\xi,\wtd I)|_{\mc H_a})=\mk^i L^\boxtimes(\xi,\wtd I)|_{\mc H_i}\qquad(\forall\wtd I\in\Jtd,\forall\xi\in\mc H_a(I))
\end{gathered}
\end{gather}
\footnote{We will see in Thm. \ref{lb95} that $\pi_i'$ equals the $\pi_i$ in \eqref{eq69}} and the action of $\fk F_\CN$ on the morphism spaces is $\id$;   $\fk N^\boxtimes$ is determined by
\begin{align}\label{eq82}
\fk N^\boxtimes\circ \mk_{\blt,\star}\circ\scr E_{\mc A_V,\Connes}\big|_{\Rep^V(\mc A_V)}=\scr E_{\mc B_{\Theta},\Connes}\big|_{\Rep_{\Rep^V(\mc A_V)}(\mc B_\Theta)}
\end{align}
where $\scr E_{\mc B_\Theta,\Connes}$ is the Connes categorical extension for $\mc B_\Theta$.

Finally, we let $\pmb{(\RepUP,\boxdot_{U_P},\ss^{U_P})}$ be the Huang-Lepowsky braided $C^*$-tensor category for $U_P$. Let
\begin{align}
\tcboxmath{(\fk F_\VOA,\fk V^\boxdot):\big(\Rep^0(P),\boxdot_P,\ss^P \big)\xlongrightarrow{\simeq}\big(\RepUP,\boxdot_{U_P},\ss^{U_P}\big)}
\end{align}
be as in Thm. \ref{lb67}.

\subsubsection{The main comparison theorem}

Assume the setting in Subsubsec. \ref{lb93}. If $U_P$ satisfies Condition II, let
\begin{align*}
\tcboxmath{\fk W^{U_P}:\mc H_\blt\boxdot_{U_P}\mc H_\star\rightarrow\mc H_\blt\boxtimes_{\mc A_{U_P}}\mc H_\star}
\end{align*}
be the Wassermann tensorator for $U_P$ (cf. Def. \ref{lb62}). Then Thm. \ref{lb88} and \ref{lb94} can be rewritten as follows.

\begin{thm}[Main comparison theorem]\label{lb95}
Assume that $V$ satisfies Condition II. Then we have
\begin{align*}
\mc A_{U_P}=\mc B_\Theta
\end{align*}
as conformal nets acting on the Hilbert space $\mc H_a$, and the following diagram commutes
\begin{equation}\label{eq74}
\begin{tikzcd}[row sep=large,column sep=2cm]
\Rep^0(P) \arrow[r,"\tFVCWX","\simeq"'] \arrow[d,"{\fk F_\VOA}"',"\simeq"] & \Rep^0_{\Rep^V(\mc A_V)}(\Theta) \arrow[d,"{\fk F_\CN}","\simeq"'] \\
\RepUP \arrow[r,"{\fk F^{U_P}_\CWX}","\simeq"']           & \Rep_{\Rep^V(\mc A_V)}(\mc B_\Theta)          
\end{tikzcd}
\end{equation}
Moreover, if $U_P$ also satisfies Condition II, then \eqref{eq74} can be extended to a commutative diagram of braided $*$-functors
\begin{equation}\label{eq81}
\begin{tikzcd}[row sep=large,column sep=3cm]
\big(\Rep^0(P),\boxdot_P,\ss^P\big) \arrow[r,"{(\tFVCWX,\wtd{\fk W}^V)}","\simeq"'] \arrow[d,"{(\fk F_\VOA,\fk V^\boxdot)}"',"\simeq"] & \big(\Rep^0_{\Rep^V(\mc A_V)}(\Theta),\boxtimes_\Theta,\mathbb B^\Theta \big)\arrow[d,"{(\fk F_\CN,\fk N^\boxtimes)}","\simeq"'] \\
\big(\RepUP,\boxdot_{U_P},\ss^{U_P} \big)\arrow[r,"{(\fk F^{U_P}_\CWX,\fk W^{U_P})}","\simeq"']           & \big(\Rep_{\Rep^V(\mc A_V)}(\mc B_\Theta),\boxtimes_{\mc B_\Theta},\mathbb B^{\mc B_\Theta} \big)         
\end{tikzcd}
\end{equation}
\end{thm}

Recall that by Cor. \ref{lb90}, if $V$ satisfies Condition I, then both $V$ and $U_P$ satisfy Condition II. Then Thm. \ref{lb95} applies to this case.

\begin{proof}
We use the notations in Subsubsec. \ref{lb74}. In \eqref{eq75}, we have
\begin{align}\label{eq83}
\mk L^\boxtimes(\xi,\wtd I)|_{\mc H_a}\xlongequal{\eqref{eq106}}\mu\circ(\fk W^V_{a,a})^{-1}L^\boxtimes(\xi,\wtd I)|_{\mc H_a}\xlongequal{\eqref{eq76}}\mu L^\boxdot(\xi,\wtd I)|_{\mc H_a}
\end{align}
Thus, by \eqref{eq77}, we have
\begin{align*}
\mc B_\Theta=\mc B_Q
\end{align*}
By Thm. \ref{lb88}, we get $\mc B_\Theta=\mc A_{U_P}$. 

Consider the following diagram
\begin{equation}\label{eq78}
\begin{tikzcd}[row sep=large,column sep=large]
\big(\Rep^0_{\Rep^V(\mc A_V)}(Q),\boxdot_Q,\ss^Q \big) \arrow[r,"{(\clubsuit,\wtd{\fk W}^V)}","\simeq"'] \arrow[d,"{(\fk F_\CN,\fk N^\boxdot)}"',"\simeq"] & \big(\Rep^0_{\Rep^V(\mc A_V)}(\Theta),\boxtimes_\Theta,\mathbb B^\Theta \big)\arrow[d,"{(\fk F_\CN,\fk N^\boxtimes)}","\simeq"'] \\
\big(\Rep_{\Rep^V(\mc A_V)}(\mc B_Q),\boxtimes_{\mc B_Q},\mathbb B^{\mc B_Q} \big)\arrow[r,"="]           & \big(\Rep_{\Rep^V(\mc A_V)}(\mc B_\Theta),\boxtimes_{\mc B_\Theta},\mathbb B^{\mc B_\Theta} \big)         
\end{tikzcd}
\end{equation}
The two braided $C^*$-categories in the second row are identical since $\mc B_Q=\mc B_\Theta$. The left vertical arrow of \eqref{eq78} equals the right vertical arrow of \eqref{eq79}. $\clubsuit$ is the $*$-functor sending each $(\mc H_i,\mu^i)$ to $(\mc H_i,\mk^i)$ where $\mk^i=\mu^i\circ(\fk W_{a,i}^V)^{-1}$ as in \eqref{eq80}, and acting on the morphism spaces as the identity. So the first row of \eqref{eq79}, composed with the first row of \eqref{eq78}, equals the first row of \eqref{eq81}. We claim that the diagram \eqref{eq78} commutes. Then the composition of \eqref{eq67} (in Thm. \ref{lb88}) with \eqref{eq78} proves that \eqref{eq74} commutes; when $U_P$ also satisfies Condition II, the composition of \eqref{eq79} (in Thm. \ref{lb94}) with \eqref{eq78} proves that \eqref{eq81} commutes.

Choose any $(W_i,\mu^i)\in\Obj(\Rep^0(P))$. Then $(\mc H_i,\mu^i)\in\Obj(\Rep^0_{\Rep^V(\mc A_V)}(Q))$ is sent by $\clubsuit$ to $(\mc H_i,\mk^i)$, and then sent by $\fk F_\CN$ to the $\mc B_\Theta$-module $(\mc H_i,\pi'_i)$ where for each $\wtd I\in\Jtd,\xi\in\mc H_a(I)$ we have (by \eqref{eq84})
\begin{align*}
\pi'_{i,I}(\mk L^\boxtimes(\xi,\wtd I)|_{\mc H_a})=\mk^i L^\boxtimes(\xi,\wtd I)|_{\mc H_i}
\end{align*}
On the other hand, $(\mc H_i,\mu^i)\in\Obj(\Rep^0_{\Rep^V(\mc A_V)}(Q))$ is sent by the left vertical arrow of \eqref{eq78} to $(\mc H_i,\pi_i)$ where (by Thm. \ref{lb27})
\begin{align*}
\pi_{i,I}(\mu L^\boxdot(\xi,\wtd I)|_{\mc H_a})=\mu^i L^\boxdot(\xi,\wtd I)|_{\mc H_i}
\end{align*}
By \eqref{eq83} and by
\begin{align}
\mk^i L^\boxtimes(\xi,\wtd I)|_{\mc H_i}\xlongequal{\eqref{eq80}}\mu^i\circ(\fk W^V_{a,i})^{-1}L^\boxtimes(\xi,\wtd I)|_{\mc H_i}\xlongequal{\eqref{eq76}}\mu^i L^\boxdot(\xi,\wtd I)|_{\mc H_i}
\end{align}
we have $\pi_i'=\pi_i$. So \eqref{eq78} commutes as a diagram of functors.

Finally, let us take care of the tensorators in \eqref{eq78}. Choose any $W_i,W_j\in\Rep^0(P)$. We need to show that the following diagram commutes:
\begin{equation}\label{eq86}
\begin{tikzcd}
\mc H_i\boxdot_P\mc H_j \arrow[r,"{\wtd {\fk W}^V_{i,j}}","\simeq"'] \arrow[d,"\fk N^\boxdot_{i,j}"',"\simeq"] & \mc H_i\boxdot_\Theta\mc H_j \arrow[d,"\fk N^\boxtimes_{i,j}","\simeq"'] \\
\mc H_i\boxdot_{\mc B_\Theta}\mc H_j \arrow[r,"="]           & \mc H_i\boxdot_{\mc B_\Theta}\mc H_j         
\end{tikzcd}
\end{equation}
(More precisely, we want to prove  $\fk N^\boxdot=\fk N^\boxtimes\circ\clubsuit(\wtd{\fk W}^V)$, and we recall that $\clubsuit$ acts as the identity on morphisms.) We compute that
\begin{align*}
&\fk N^\boxtimes\circ\wtd{\fk W}^V\circ\scr E_Q\xlongequal{\eqref{eq71}}\fk N^\boxtimes\circ\wtd{\fk W}^V\circ \mu_{\blt,\star}\circ\scr E_V\xlongequal{\eqref{eq85}}\fk N^\boxtimes\circ\mk_{\blt,\star}\circ \fk W^V\circ\scr E_V\\
\xlongequal{\eqref{eq76}}&\fk N^\boxtimes\circ\mk_{\blt,\star}\circ \scr E_{\mc A_V,\Connes}|_{\Rep^V(\mc A_V)}\xlongequal{\eqref{eq82}}\scr E_{\mc B_{\Theta},\Connes}\big|_{\Rep_{\Rep^V(\mc A_V)}(\mc B_\Theta)}
\end{align*}
By \eqref{eq91b}, the rightmost term above equals $\fk N^\boxdot\circ\scr E_Q$. This proves $\fk N^\boxdot=\fk N^\boxtimes\circ\wtd{\fk W}^V$. So \eqref{eq86} commutes.
\end{proof}



\begin{co}\label{lb96}
Assume that $V$ satisfies Condition II, and let $U$ be a unitary VOA extension of $V$. Then the CWX functor $\fk F^U_\CWX:\RepU\xrightarrow{\simeq}\Rep_{\Rep^V(\mc A_V)}(\mc A_U)$ can be extended to a braided $*$-functor implementing an isomorphism of braided $C^*$-tensor categories:
\begin{align*}
(\fk F^U_\CWX,?):\big(\RepU,\boxdot_U,\ss^U\big)\xlongrightarrow{\simeq} \big(\Rep_{\Rep^V(\mc A_V)}(\mc A_U),\boxtimes_{\mc A_U},\mathbb B^{\mc A_U} \big) 
\end{align*}
where the RHS is the Connes braided $C^*$-tensor category of $\mc A_U$-modules whose restrictions to $\mc A_V$ are objects of $\Rep^V(\mc A_V)$.
\end{co}

\begin{proof}
This is immediate from the first half of Thm. \ref{lb95} and
\begin{equation}
\begin{tikzcd}[row sep=large,column sep=3cm]
\big(\Rep^0(P),\boxdot_P,\ss^P\big) \arrow[r,"{(\tFVCWX,\wtd{\fk W}^V)}","\simeq"'] \arrow[d,"{(\fk F_\VOA,\fk V^\boxdot)}"',"\simeq"] & \big(\Rep^0_{\Rep^V(\mc A_V)}(\Theta),\boxtimes_\Theta,\mathbb B^\Theta \big)\arrow[d,"{(\fk F_\CN,\fk N^\boxtimes)}","\simeq"'] \\
\big(\RepUP,\boxdot_{U_P},\ss^{U_P} \big)     & \big(\Rep_{\Rep^V(\mc A_V)}(\mc B_\Theta),\boxtimes_{\mc B_\Theta},\mathbb B^{\mc B_\Theta} \big)         
\end{tikzcd}
\end{equation}
if we let $U=U_P$. (Note that this Corollary can also be proved by Thm. \ref{lb88} using a similar argument.)
\end{proof}

\subsection{Surjectivity of the CWX functors}

In this subsection, we enhance our main comparison Thm. \ref{lb95} to Cor. \ref{lb117} and provide examples that Cor. \ref{lb117} can be applied. For each conformal net $\mc A$, we let
\begin{align*}
\Repf(\mc A)=\text{the category of dualizable $\mc A$-modules}
\end{align*}
i.e., the category of all $\mc A$-modules which have dual objects in $\Rep(\mc A)$. Then $\Repf(\mc A_V)$ is a rigid braided full $C^*$-tensor subcategory of $\RepV$.

Recall that if $V$ satisfies Condition II and $U=U_P$, then $\Rep^V(\mc A_V):=\FVCWX(\RepV)$, and $\Rep_{\Rep^V(\mc A_V)}(\mc A_U)$ is the category of $\mc A_U$-modules whose restrictions to $\mc A_V$ are objects of $\Rep^V(\mc A_V)$.

\begin{thm}[\cite{Gui26}, Thm. 3.38]\label{lb115}
Assume that $V$ satisfies Condition II, and let $U$ be a unitary VOA extension of $V$. Then $\mc A_V$ is completely rational if and only if $\mc A_U$ is so.
\end{thm}

\begin{proof}
Let $U=U_P$. Then $\mc A_{U_P}$ is a finite-index extension of $\mc A_V$. So the complete rationality of the two nets are equivalent by \cite[Thm. 24]{Lon03}.
\end{proof}

\begin{thm}\label{lb97}
Assume that $V$ satisfies Condition II. Then $\Rep^V(\mc A_V)$ is a full replete subcategory of $\Repf(\mc A_V)$, and $\Rep_{\Rep^V(\mc A_V)}(\mc A_U)$ is a full replete subcategory of $\Repf(\mc A_U)$. Moreover, consider the conditions
\begin{enumerate}[label=(\arabic*)]
\item $\Rep^V(\mc A_V)=\Repf(\mc A_V)$
\item $\Rep_{\Rep^V(\mc A_V)}(\mc A_U)=\Repf(\mc A_U)$
\end{enumerate}
Then we have (1)$\Rightarrow$(2). We have (2)$\Rightarrow$(1) when $\mc A_V$ is completely rational (equivalently, $\mc A_U$ is completely rational, cf. Thm. \ref{lb115}). In other words, consider the conditions
\begin{enumerate}[label=(\alph*)]
\item The functor $\fk F^V_\CWX:\RepV\rightarrow\Repf(\mc A_V)$ is (essentially) surjective.
\item The functor $\fk F^U_\CWX:\RepU\rightarrow\Repf(\mc A_U)$ is (essentially) surjective.
\end{enumerate}
Then we have (a)$\Rightarrow$(b). In view of Cor. \ref{lb96}, we have (b)$\Rightarrow$(a) when $\mc A_V$ is completely rational (equivalently, $\mc A_U$ is completely rational).
\end{thm}

We clearly have (1)$\Leftrightarrow$(a) and (by Cor. \ref{lb96}) (2)$\Leftrightarrow$(b). If (1) or (a) is satisfied, we simply say that \textbf{$V$ has surjective CWX functor}. Then (2) and (b) both mean that $U$ has surjective CWX functor. It is also clear that $V$ has surjective CWX functor if and only if $\RepV$ and $\Repf(\mc A_V)$ have the same number of irreducibles. 

\begin{proof}
Let $U=U_P$. Recall that $\mc A_{U_P}=\mc B_\Theta$ by Thm. \ref{lb95}. Clearly $\Rep^V(\mc A_V)$ and $\Rep_{\Rep^V(\mc A_V)}(\mc A_{U_P})$ are full replete subcategories of $\Rep(\mc A_V)$ and $\Rep(\mc A_{U_P})$ respectively. Since $\Rep^V(\mc A_V)$ is isomorphic to $\RepV$ as braided  $C^*$-tensor categories, and since every object in $\RepV$ is dualizable (by \cite{Hua08b}), $\Rep^V(\mc A_V)$ is rigid. So $\Rep^V(\mc A_V)\subset \Repf(\mc A_V)$. 

By Cor. \ref{lb41}, the braided $C^*$-tensor category $\Rep^0(\Theta)\equiv\Rep^0_{\Rep(\mc A_V)}(\Theta)$ is canonically isomorphic to $\Rep(\mc A_{U_P})$. Under this isomorphism, $\Rep_{\Rep^V(\mc A_V)}(\mc A_{U_P})$ and $\Repf(\mc A_{U_P})$ become $\Rep^0_{\Rep^V(\mc A_V)}(\Theta)$ and $\Rep^{0,\mathrm f}(\Theta)$ (the category of dualizable dyslectic $\Theta$-modules) respectively. Therefore, condition (2) is equivalent to
\begin{align}
\Rep^0_{\Rep^V(\mc A_V)}(\Theta)=\Rep^{0,\mathrm f}(\Theta)
\end{align}
Therefore, if (1) holds, then (2) is equivalent to that $\Rep^0_{\Repf(\mc A_V)}(\Theta)=\Rep^{0,\mathrm f}(\Theta)$, i.e., that for each object $(\mc H_i,\mk^i)$ of $\Rep^0(\Theta)$ we have
\begin{gather}\label{eq88}
\text{$\mc H_i$ is dualizable as an $\mc A_V$-module iff $(\mc H_i,\mk^i)$ is dualizable}
\end{gather}
But this is well-known, cf. \cite[Thm. 1.15]{KO02}, \cite[Sec. 6]{NY16}, \cite[Thm. 3.18]{Gui22}. (Alternatively, one can directly show that an $\mc A_{U_P}$-module $(\mc H_i,\pi_i)$ is dualizable iff it is dualizable as an $\mc A_V$-module. This is equivalent to that, for any fixed $I\in\mc J$, the subfactor $\pi_{i,I}(\mc A_{U_P}(I))\subset\pi_{i,I}(\mc A_{U_P}(I'))'$ has finite index iff $\pi_{i,I}(\mc A_V(I))\subset\pi_{i,I}(\mc A_V(I'))'$ has finite index, which is true since $\mc A_{U_P}$ is a finite-index extension of $\mc A_V$. See \cite{Lon89,Lon90,Kos98,BDH14}). Thus, we have proved (1)$\Rightarrow$(2), equivalently, (a)$\Rightarrow$(b).

Assuming that $\mc A_V$ (and hence $\mc A_{U_P}$) is completely rational, the direction (b)$\Rightarrow$(a) follows from the same proof as that of \cite[Thm. 3.39]{Gui26} by showing that the global dimension of $\RepUP$ divided by that of $\RepV$ equals the global dimension of $\Repf(\mc A_{U_P})$ divided that of $\Repf(\mc A_V)$. 
\end{proof}

\begin{rem}
Thm. \ref{lb97} generalizes \cite[Thm. 3.39]{Gui26} in that $U$ is not assumed to satisfy Condition II, and that the complete rationality of $\mc A_V$ or $\mc A_U$ is not assumed in the proof of (1)$\Rightarrow$(2) and (equivalently) (a)$\Rightarrow$(b).
\end{rem}

\begin{co}\label{lb117}
Assume the setting in Subsubsec. \ref{lb93}. Assume that $V$ satisfies Condition II and has surjective CWX functor. Then we have
\begin{align*}
\mc A_{U_P}=\mc B_\Theta
\end{align*}
as conformal nets acting on the Hilbert space $\mc H_a$, and the following diagram commutes
\begin{equation}\label{eq89}
\begin{tikzcd}[row sep=large,column sep=2cm]
\Rep^0(P) \arrow[r,"\tFVCWX","\simeq"'] \arrow[d,"{\fk F_\VOA}"',"\simeq"] & \Rep^0_{\Repf(\mc A_V)}(\Theta) \arrow[d,"{\fk F_\CN}","\simeq"'] \\
\RepUP \arrow[r,"{\fk F^{U_P}_\CWX}","\simeq"']           & \Repf(\mc B_\Theta)          
\end{tikzcd}
\end{equation}
Moreover, if $U_P$ also satisfies Condition II, then \eqref{eq89} can be extended to a commutative diagram of braided $*$-functors
\begin{equation}
\begin{tikzcd}[row sep=large,column sep=3cm]
\big(\Rep^0(P),\boxdot_P,\ss^P\big) \arrow[r,"{(\tFVCWX,\wtd{\fk W}^V)}","\simeq"'] \arrow[d,"{(\fk F_\VOA,\fk V^\boxdot)}"',"\simeq"] & \big(\Rep^0_{\Repf(\mc A_V)}(\Theta),\boxtimes_\Theta,\mathbb B^\Theta \big)\arrow[d,"{(\fk F_\CN,\fk N^\boxtimes)}","\simeq"'] \\
\big(\RepUP,\boxdot_{U_P},\ss^{U_P} \big)\arrow[r,"{(\fk F^{U_P}_\CWX,\fk W^{U_P})}","\simeq"']           & \big(\Repf(\mc B_\Theta),\boxtimes_{\mc B_\Theta},\mathbb B^{\mc B_\Theta} \big)         
\end{tikzcd}
\end{equation}
\end{co}

\begin{proof}
This is clear from Thm. \ref{lb95} and \ref{lb97}.
\end{proof}

In the remainder of this subsection, we present examples to which Cor. \ref{lb117} applies. To simplify the discussion, we introduce the following definition:
\begin{df}
We say $V$ satisfies \textbf{Condition I+} (resp. \textbf{II+}) if $V$ satisfies Condition I (resp. II), if $\mc A_V$ is completely rational, and if $V$ has surjective CWX functor. 
\end{df}

Note that Condition I+ clearly implies Condition II+.

\begin{thm}\label{lb116}
The following are true.
\begin{enumerate}[label=(\arabic*)]
\item $V,V'$ are unitary VOAs satisfying Condition I+ (resp. II+) if and only if the tensor product unitary VOA $V\otimes V'$ satisfies Condition I+ (resp. II+).
\item Assume that $U$ is a unitary VOA extension of $V$. If $V$ satisfies Condition I+, then $U$ satisfies Condition I+. If $V$ satisfies Condition II+ and $U$ satisfies Condition II, then $U$ satisfies Condition II+.
\item Assume that $U$ is a unitary VOA satisfying Condition I+ (resp. II+). Assume that $V$ is a unitary subalgebra of $U$ such that both $V$ and the coset VOA $V^c$ are $C_2$-cofinite and rational. Then $V^c$ satisfies Condition I+ (resp. II+).
\end{enumerate}
\end{thm}

\begin{proof}
We only prove the complete rationality and the surjectivity of CWX functors, since the claims about Conditions I and II follow from Thm. \ref{lb102}. (Note that in (2), if $V$ satisfies I, then by Cor. \ref{lb90}, $U$ satisfies I.)

(1): We can assume that $V,V',V\otimes V'$ all satisfy Condition I (resp. II). By \cite[Cor. 8.2]{CKLW18}, we have $\mc A_{V\otimes V'}\simeq\mc A_V\otimes\mc A_{V'}$. Therefore, by \cite[Lem. 25]{Lon03}, $\mc A_{V\otimes V'}$ is completely rational if and only if both $\mc A_V$ and $\mc A_{V'}$ are completely rational.

Let us assume $\mc A_V,\mc A_{V'},\mc A_{V\otimes V'}$ are completely rational, and study the equivalence of the surjectivities of CWX functors. Since every irreducible representation of a completely rational conformal net has finite index (\cite{LX04}), each of $V,V',V\otimes V$ has surjective CWX functor iff the VOA and its CKLW net have the same number of irreducibles. By \cite[Thm. 4.7.4]{FHL93}, the number of irreducibles of $V\otimes V'$ is the multiplication of the numbers of the irreducibles of $V$ and $V'$. By Cor. 14 and Lem. 27 of \cite{KLM01}, the number of irreducibles of $\mc A_V\otimes\mc A_{V'}$ is the multiplication of the numbers of the irreducibles of $\mc A_V$ and $\mc A_{V'}$. This finishes the proof of (1).

(2): This follows from Thm. \ref{lb115} and \ref{lb97}.

(3): As explained at the beginning of \cite[Sec. 3.6]{Gui26}, $U$ is naturally a unitary VOA extension of $V\otimes V^c$. Therefore, by Thm. \ref{lb115}, $\mc A_{V\otimes V^c}$ is completely rational. By Thm. \ref{lb97}, $V\otimes V^c$ has surjective CWX functor. Therefore, as argued for (1), $\mc A_{V^c}$ is completely rational, and $V^c$ has surjective CWX functor. Therefore, $V^c$ satisfies Condition I+ (resp. II+). 
\end{proof}

\begin{eg}\label{lb118}
The following examples satisfy Condition II+:
\begin{itemize}
\item[(a)] All unitary affine VOAs. All even lattice VOAs. All discrete series $W$-algebras of type $ADE$ (in the sense of \cite{ACL19}).  All parafermion VOAs (in the sense of \cite{DR17}).
\end{itemize}
The following examples satisfy Condition I+:
\begin{itemize}
\item[(b)] All unitary affine VOAs of type $ADE$. All even lattice VOAs. All discrete series $W$-algebras of type $ADE$.  All parafermion VOAs of type $ADE$.
\end{itemize}
\end{eg}

\begin{proof}
See \cite[Sec. 3.7]{Gui26}.
\end{proof}

The combination of Exp. \ref{lb118} and Thm. \ref{lb116} yields a vast collection of examples where Cor. \ref{lb117} are applicable, cf. Exp. \ref{lb111} and \ref{lb112}.

\noindent {\small \sc Yau Mathematical Sciences Center, Tsinghua University, Beijing, China.}

\noindent {\textit{E-mail}}: binguimath@gmail.com\qquad bingui@tsinghua.edu.cn
\end{document}